\documentclass[10pt]{article}
\usepackage{amsthm}
\usepackage{amsfonts, amssymb, amsmath, amscd, latexsym}
\usepackage{mathrsfs,epsfig}
\usepackage[latin1]{inputenc}
\usepackage[all]{xy}
\usepackage{cite}
\usepackage{subcaption}
\usepackage[dvipsnames]{xcolor}

\newtheorem{knowntheorem}{Theorem}

\newtheorem{knownproposition}[knowntheorem]{Proposition}
\newtheorem{knownlemma}[knowntheorem]{Lemma}

\def\today{\ifcase\month\or
January\or February\or March\or April\or May\or June\or July \or
August\or September\or October\or November\or December\fi
\space\NNumber\day, \NNumber\year}
\numberwithin{equation}{section}

\def\eu{ {\, \textrm{\rm e} }}

\def\sbwd{\sigma^{\textrm{bwd}}}

\def\sfwd{\sigma^{\textrm{fwd}}}
\def\sbwd{\sigma^{\textrm{bwd}}}
\def\sTfwd{\Sigma^{\textrm{fwd}}}
\def\sTbwd{\Sigma^{\textrm{bwd}}}
\def\dist{{\mathscr{D}}}
\def\al{\alpha}

\def\ga{\vec{\gamma}}
\def\de{\delta}
\def\ep{\varepsilon}
\def\la{\lambda}
\def\si{\sigma}

\def\Om{\Omega}
\def\na{\vec{\nabla}}

\def\N{{\mathbb N}}
\def\Z{{\mathbb Z}}
\def\R{{\mathbb R}}

\def\ee{{\mathrm e}}
\def\EE{{\mathcal E}}

\def\MM{{\mathcal M}}

\def\SS{{\mathcal S}}
\def\TT{{\mathcal T}}

\def\DDD{\mathscr{D}}
\def\TTT{\mathscr{T}}
\def\PPP{\mathscr{P}}

\def\dd{\delta}
\def\aup{a^{\uparrow}}
\def\adown{a^{\downarrow}}
\def\lzero{\Lambda^0}
\def\l1{\Lambda^{1}}

\def\Q{\vec{Q}}
\def\P{\vec{P}}

\def\xxi{\vec{\xi}}

\def\x{\vec{x}}

\def\f{\vec{f}}
\def\g{\vec{g}}

\newcommand{\lit}{{\lim_{t \to +\infty}}}
\newcommand{\lito}{{\lim_{t \to -\infty}}}

\def\Ae(t){\boldsymbol{A^\ep}}

\def\({\left(}
\def\){\right)}

\DeclareMathOperator{\tr}{tr}
\newcommand\bs[1]{{\boldsymbol{#1}}}

\def\und{\underline}
\def\ov{\overline}
\DeclareMathOperator{\sign}{sign}

\def\ogap{\Theta}   
\def\aaa{\theta}

\newcommand\assump[1]{{\rm\textbf{#1}}}
\def\inn{\textrm{in}}

\newtheorem{theorem}{Theorem}[section]
\newtheorem{proposition}[theorem]{Proposition}
\newtheorem{lemma}[theorem]{Lemma}
\newtheorem{cor}[theorem]{Corollary}
\theoremstyle{definition}

\newtheorem{example}[theorem]{Example}

\theoremstyle{remark}
\newtheorem{remark}[theorem]{Remark}
\begin{document}
\title{A new construction for Melnikov chaos in piecewise-smooth planar systems}
\pagestyle{myheadings}
\date{}
\markboth{}{Melnikov chaos for planar systems}

\author{	
A. Calamai\thanks{Dipartimento di Ingegneria Civile, Edile e Architettura,
Universit\`a Politecnica delle Marche, Via Brecce Bianche 1, 60131 Ancona -
Italy. Partially supported by G.N.A.M.P.A. - INdAM (Italy) and
PRIN 2022 - Progetti di Ricerca di rilevante Interesse Nazionale, \emph{Nonlinear differential
problems with applications to real phenomena} (Grant Number: 2022ZXZTN2).
},
M. Franca\thanks{Dipartimento di Scienze Matematiche, Universit\`a di Bologna, 40126 Bologna, Italy.
 Partially supported by G.N.A.M.P.A. - INdAM (Italy) and PRIN 2022 - Progetti di Ricerca di rilevante Interesse Nazionale, \emph{Nonlinear
 differential
problems with applications to real phenomena} (Grant Number: 2022ZXZTN2). },
M. Posp\' i\v sil\thanks{Department of Mathematical Analysis and Numerical Mathematics, Faculty of Mathematics, Physics and Informatics,
Comenius University in Bratislava,
Mlynsk\'a dolina, 842 48 Bratislava, Slovakia; Mathematical Institute, Slovak Academy of Sciences, \v Ste\-f\'a\-ni\-ko\-va 49, 814 73
Bratislava, Slovakia.
Partially supported by the Slovak Research and Development Agency under the Contract no. APVV-23-0039, and by the Grants VEGA 1/0084/23 and
VEGA-SAV 2/0062/24.}
}

\maketitle

\begin{abstract}
In this paper we consider a piecewise smooth $2$-dimensional system
\begin{equation*}\label{e.ab}
\dot{\x}=\f (\x)+\ep\g(t,\x,\ep)
\end{equation*}
where $\ep>0$ is a small parameter and $\f$ is discontinuous along a curve $\Om^0$.
We assume that $\vec{0}$ is a critical point for any $\ep \ge 0$,  and that  for $\ep=0$ the system  admits a trajectory
$\ga(t)$ homoclinic to $\vec{0}$ and crossing transversely $\Om^0$ in $\ga(0)$.

 In a previous paper we have shown that, also in an $n$-dimensional setting,   the classical Melnikov condition is enough to guarantee the persistence of the homoclinic
to perturbations, but  more recently  we have found an open condition, a geometric obstruction which is not possible in the smooth case,
which prevents chaos  for $2$-dimensional systems  when $\g$ is periodic in $t$.

In this paper we show that when this obstruction is removed we have chaos as in the smooth case. The proofs involve
a new construction of the set $\Sigma$ from which the chaotic pattern originates.

The results are illustrated by examples.
\end{abstract}

{\bf Keywords: Melnikov theory, piecewise-smooth systems, chaos, \\ homoclinic trajectories, non-autonomous dynamical systems,
non-autonomous perturbation.}


\vskip 24bp


 \section{Introduction}
In this paper we deal with planar, piecewise-smooth systems and we prove, by Melnikov theory, the
occurrence of a chaotic behavior when the fixed point of the unperturbed system lies on the discontinuity level.

More in detail, we consider the piecewise-smooth system
\begin{equation}\label{eq-disc}\tag{PS}
	\dot{\x}=\f^\pm(\x)+\ep\g(t,\x,\ep),\quad \x\in\Om^\pm ,
\end{equation}
where $\Om^{\pm} = \{ \x\in\Om \mid \pm
G(\x)>0\}$, $\Om^{0}= \{ \x\in\Om \mid G(\x) = 0 \}$, $\Omega \subset \R^2$ is an open set, $G$ is a $C^{r}$-function on $\Om$ with $r> 1$
such that $0$ is a regular value of
$G$. Next, $\ep\in\R$ is a small parameter, and $\f^\pm \in C^{r}_b(\Om^\pm \cup \Omega^0, \R^2)$, $\g\in C_b^r(\R\times\Om\times\R,\R^2)$ and
$G\in C_b^r(\Om,\R)$, i.e., the
derivatives of $\f^\pm$, $\g$ and $G$ are uniformly continuous and bounded up to the $r$-th order, respectively,
 if $r\in \N$ and up to $r_0$ if $r=r_0+r_1$, with $r_0 \in \N$ and $0< r_1 <1$ and the $r_0$-th derivatives are $r_1$
H\"older continuous.

  Here and in the sequel, we
use the shorthand notation $\pm$
to
represent both the $+$ and $-$ equations and functions.

We assume that both the systems
$\dot{\vec{x}}=\f^{\pm}(\x)$
admit the origin $\vec{0} \in \R^2$ as a fixed point,
and that $\vec{0}$ lies
on the discontinuity level $\Omega^0$.
Moreover, the $2\times2$ matrices $\bs{f_x^\pm}(\vec{0})$ have real eigenvalues of opposite sign, and the corresponding eigenvectors are
transversal to $\Om^0$ at $\vec{0}$,
see Section \ref{S.setting} for more details.
Further, we assume that for $\ep=0$ system \eqref{eq-disc} admits a piecewise-smooth
homoclinic trajectory $\ga(t)$, crossing transversely the discontinuity level $\Omega^0$ at some point $\ga(0)\neq \vec{0}$.
That is, the homoclinic trajectory ${\ga}(t)$ satisfies:
 $$\ga(t)\in\begin{cases}\Om^-,& t<0,\\ \Om^0,&t=0,\\ \Om^+,&t>0.\end{cases}$$
So $\bs{\Gamma}= (\{\ga(t) \mid t \in \R \} \cup \{(0,0) \})$ crosses $\Om^0$ transversely twice.

Our  aim is to obtain a sufficient condition for the occurrence of a chaotic pattern, when we turn on a small perturbation $\ep \ne 0$ in
\eqref{eq-disc}.

Discontinuous problems are motivated by several physical applications, for instance mechanical systems with impacts, see e.g.~\cite{B99},
power electronics when we
have state dependent switches \cite{BV01}, walking machines \cite{GCRC98}, relay feedback systems \cite{BGV02}, biological systems \cite{PK};
see also \cite{DE13, LSZH16} and
the references therein.
Further they are also a  good source of examples since it is somehow easy to produce piecewise linear systems exhibiting an explicitly known
homoclinic
trajectory, so giving rise to chaos if subject to perturbations, whereas this is not an easy task in general for the smooth case (especially
if
the system
is not Hamiltonian).

In the smooth case, the problem of detecting the occurrence of chaotic solutions for non-autonomous dynamical systems is now classical.
In particular, it is
known that  perturbing periodically a (smooth) dynamical system
which admits a transversal homoclinic point of the period map, an invariant Smale horseshoe set originating a chaotic pattern arises; we
refer the interested reader, e.g.,
to \cite{Pa84, GH, WigBook, Wig99, SSTC}.
 More in detail,  let us consider the smooth planar system
\begin{equation}\label{eq-smooth}\tag{S}
	\dot{\x}=\f (\x)+\ep\g(t,\x,\ep), \quad \quad \x\in\Om,
\end{equation}
 having a homoclinic trajectory,  where both $\f$ and $\g$ are of class  $C^2$  in the whole $\Om$.
Melnikov theory gives an  integral condition   sufficient to ensure the persistence
of the homoclinic trajectory for $\ep \ne 0$, and, if the system is periodic in $t$, the insurgence of a chaotic pattern:
the classical requirement is that a certain computable function $\tau\mapsto\MM(\tau)$, which for system \eqref{eq-disc} is given by formula
\eqref{melni-disc} below, has a
non-degenerate zero.
In this framework, after the pioneering work of Melnikov \cite{Me}, an important step forward
was performed by Chow et al.\ in \cite{CHM}.
Another milestone in this context is \cite{Pa84} where Palmer  addressed the
$n$-dimensional case where $n \ge 2$.
When $n>2$ an additional non-degeneracy condition on the family of unperturbed homoclinic is needed,
a condition which is automatically satisfied in the $n=2$ case. Palmer's approach is analytic and combines
Fredholm alternative with the use of exponential dichotomy and shadowing lemmas.
These results have been improved from many points of view, in particular they have been extended to the almost periodic case,
see e.g.\ \cite{Sch, PS},  and to the case where the zeros of $\MM(\tau)$ are degenerate, see e.g.\ \cite{BF02C, BF02M}.
Afterwards
the so-called perturbation approach has been widely developed by many authors.
Such a theory is by now well-established for smooth systems, and there are many works devoted to it.
For example we refer to \cite{BL,CHM,G1,GH, Pa84, Pa00, PS,Sch, St,MS,WigBook, Wig99}.

In the last 20 years many authors addressed the problem of generalizing Melnikov theory
to a discontinuous piecewise-smooth setting: among the other papers, let us mention e.g.,
 \cite{K2,KRG,LPT,BF08, BF11, CaFr,CDFP,LiDu}.
In fact most of the cited results are set in the planar case $n=2$.
However we wish to quote the nice  work performed by Battelli and Fe\v ckan, who managed  to extend the theory to the $n \ge 2$ case,
but assuming that the critical point of the unperturbed system \emph{does not} belong to the discontinuity hypersurface
 (in the case $n=2$, the curve $\Om^0$),
see e.g.\ \cite{BF11}. For a generalization of Melnikov theory in other directions, see e.g.\ the recent \cite{BF22, HL24, GZ23, MV21,
ZZL21}.

 Let us assume for illustrative purposes that $\g$ is $1$-periodic in the $t$-variable
 (in fact the results of Battelli and Fe\v ckan require very weak recurrence properties, see e.g.\ \cite{BF11}).
 Roughly speaking, in order to prove the existence of the chaotic pattern
 in a discontinuous setting, Battelli and Fe\v ckan  glued together infinitely many
solutions of the perturbed problem,
  which mimic
different time-translation of the unperturbed homoclinic.
For this purpose, they had to solve infinitely many equations via implicit function theorem, using
a Lijapunov-Schimdt reduction (needed in the $n>2$ case) together with the Melnikov condition.
 A key point of their contribution, and one of the main technical and theoretical  issues,
was to understand how the discontinuity of the flow affects the jumps in the projections
of the exponential dichotomy, and how the non-degeneracy condition required in the $n>2$ case
is translated in this discontinuous setting.
This way they managed to prove the persistence of the homoclinic \cite{BF08}, the insurgence of chaos
when the unperturbed homoclinic exhibits no sliding phenomena \cite{BF11} and when it does \cite{BF10}.

If, instead, we assume that the critical point \emph{lies on the discontinuity hypersurface $\Omega^0$},
the problem of detecting a chaotic behavior becomes even more challenging.  Why shall one insist on studying this case?
 In many real applications this is what really happens, e.g.\ in systems with dry friction.
In this setting,
as a preliminary step, in \cite{CaFr} it was shown that the Melnikov condition found by Battelli and Fe\v ckan together with a further
(generic) transversality requirement
(always satisfied in two dimensions)  are enough to prove the persistence of the homoclinic trajectory.

Subsequently, in \cite{CDFP} we considered the case of a ``one-sided'' homoclinic.
In other words, in the discontinuous setting like \eqref{eq-disc} we assumed that
${\ga}(t)\in \Omega^+$ for any $t \in \R$.
In this case, with suitable assumptions we proved that the chaotic trajectories are still located in  $\Omega^+$.
 We point out that, even if the main motivation and applications are set in the discontinuous case,
a crucial idea in both the papers \cite{CaFr} and \cite{CDFP}  was that the problem addressed   could be
equivalently stated for continuous systems and reduced, in some sense, to a problem of location of trajectories.
Moreover, we stress that in \cite{CaFr} and \cite{CDFP}, like in the paper by Battelli and Fe\v ckan \cite{BF11}, we studied systems in
$\R^n$.

Hence, in the case in which the   origin lies on the discontinuity hypersurface $\Om^0$ and $\bs{\Gamma}$ crosses
$\Om^0$ transversely
 twice, so that it
 splits in two branches,  the problem of the occurrence of chaotic patterns was still open.
 The purpose of the paper is to fill this gap, at least in the two-dimensional case.

We think it is worth mentioning the recent papers \cite{HL24, HuLi}, still concerning the $n$-dimensional case,
$n \ge 2$, and in which
the origin lies on the discontinuity hypersurface $\Om^0$, proving persistence of the homoclinic trajectory
in the $n>2$ degenerate case, and persistence of the periodic solutions.

As pointed out, here we focus on the two-dimensional case. In this setting, in the recent paper \cite{FrPo},
quite surprisingly,
it was shown
that the Melnikov condition,
which ensures the persistence of the homoclinic \cite{CaFr} and the transversality of
the crossing between stable and unstable leaves, \emph{does not guarantee} the insurgence of chaos
differently from the smooth setting, and also from the piecewise-smooth setting considered, e.g., in \cite{K2,KRG,LPT,BF08,BF10, BF11}.
In fact, anytime we have sliding phenomena
 close to the origin, a geometrical obstruction prevents the formation of chaotic patterns, and new bifurcation
 phenomena take place, scenarios which may exist just in a discontinuous context.

 In this paper we complete the picture
 of the two-dimensional, discontinuous case showing
 that, when close to the origin no sliding phenomena occur,
  system \eqref{eq-disc} exhibits a chaotic behavior
as in the smooth setting; while according to \cite{CaFr}, if there is sliding close to the origin, even if
 it does not involve the homoclinic trajectory, we have persistence of the homoclinic, but no chaos when the perturbation is turned on.

\begin{figure}[t!]
\centerline{\epsfig{file=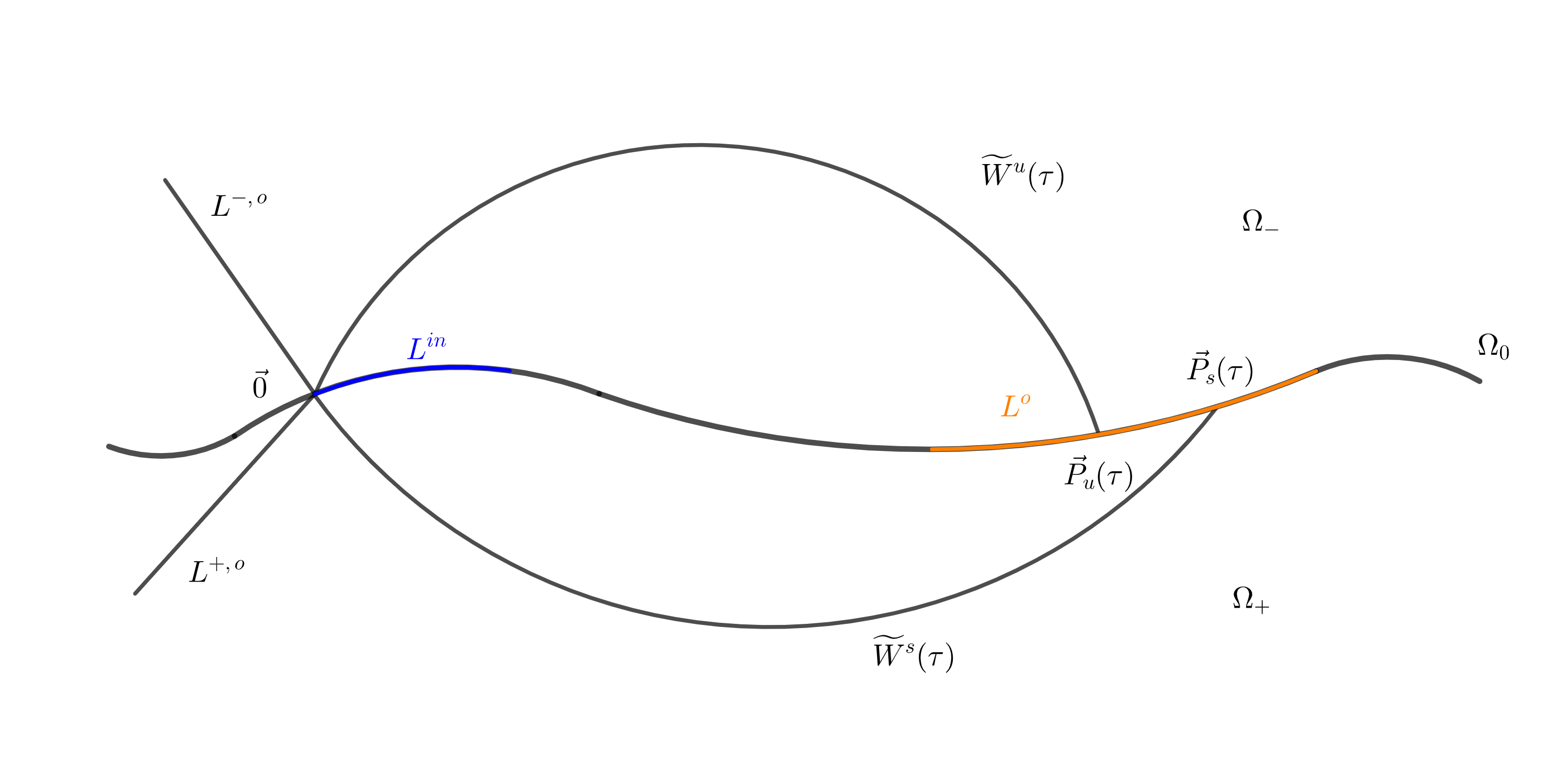, width = 10 cm} }
\caption{Stable and unstable leaves.}
\label{LinL0}
\end{figure}

  For the proofs we profit of the results obtained in our previous papers \cite{CFPRimut, CFPrestate}.
  Let us now summarize the main ideas.
   In \cite{CFPrestate} we have proved the existence of  trajectories which are chaotic only in the future or chaotic only in the past.
  Assume again for illustrative purposes that,
 $\g$ is $1$-periodic in $t$   and  that the Melnikov function $\MM(\tau)$ has a non-degenerate zero, i.e., $\MM(\tau_0)=0 \ne \MM'(\tau_0)$
 for some $\tau_0 \in [0,  1]$.
 Let $\tilde{W}^s(\tau)$ and $\tilde{W}^u(\tau)$ denote the branch of  the stable/unstable manifolds between the origin and $\Om^0$
 and denote by $\vec{P}_s(\tau)$ the (transversal) intersection between $\tilde{W}^s(\tau)$ and $\Om^0$ and by
 $\vec{P}_u(\tau)$ the (transversal) intersection between $\tilde{W}^u(\tau)$ and $\Om^0$.
  Let $\x(t, \tau; \xxi)$ denote the trajectory of \eqref{eq-disc}  leaving from $\xxi$ at $t=\tau$.
  Let $\Z^+=\{j \in \Z \mid j \ge 1\}$,
 $\Z^-=\{j \in \Z \mid j \le -1\}$, $\EE^+=\{0; 1 \} ^{\Z^+}$, $\EE^- = \{0; 1 \} ^{\Z^-}$, $\EE =\{0; 1 \} ^{\Z}$.

 Firstly in \cite{CFPRimut} we have
developed a careful analysis of the fly time needed for a trajectory starting in $\Om^0$ close to $\vec{P}_s(\tau)$ at $t=\tau$
to perform a loop close to $\tilde{W}(t)= \tilde{W}^s(t)\cup \tilde{W}^u(t)$, till it intersects again $\Om^0$ close to  $\vec{P}_u(t)$, together with a careful
evaluation of the space displacement with respect to $\tilde{W}(t)$.

 Then in  \cite{CFPrestate}, using the analysis of \cite{CFPRimut} we have shown that there is $\ep_0$ such that for any $0< \ep \le \ep_0$  there are compact connected sets
$X^+(\tau,\TT^+, \ee^+)$ and $X^-(\tau,\TT^-, \ee^-)$ such that
if $\xxi^+ \in X^+$ then $\x(t,\tau; \xxi^+)$ is ``chaotic   in the future", i.e.,
it mimics a translate of $\ga(t)$  if $\ee_j=1$ and it is close to the origin if $\ee_j=0$ in $t \in [T_{2j-1}, T_{2j+1}]$ and $j \in \Z^+$;
while  if $\xxi^- \in X^-$ then $\x(t,\tau; \xxi^-)$ is ``chaotic in the past", i.e., it has a similar property when $j \in \Z^-$.

In this article we show some sort of continuity of the sets $X^+(\tau,\TT^+, \ee^+)$ and $X^-(\tau,\TT^-, \ee^-)$ with respect to $\tau$
so that we can
apply the intermediate value theorem to find a  $\tau_{\star}\in [0,1]$ such that
the intersection
$X^+(\tau_{\star},\TT^+, \ee^+)\cap X^-(\tau_{\star},\TT^-, \ee^-)$ is non-empty.

Then we construct
the set $\Sigma(\tau_{\star}, \TT)$ made up by all the initial conditions giving rise to a chaotic pattern.
 We stress that an advantage of the application of the intermediate value theorem,  instead of the implicit function theorem, is that we
can allow $\MM(\tau_{\star})$ to be a degenerate zero
(as far as we know, this fact is new in the piecewise-smooth but discontinuous context,
while  analogous results in the smooth case have been obtained  via degree theory
in \cite{BF02C, BF02M}). However, a drawback  is that we just obtain a semi-conjugation with the Bernoulli shift.

Further, our approach, inspired by the work by Battelli and Fe\v ckan, see \cite{BF08, BF10, BF11, BF12}, allows  to weaken the recurrence properties of the Melnikov
function $\MM(\tau)$.

The paper is divided as follows: in \S \ref{S.setting} we list the main hypotheses and we state our main results, i.e.\ Theorems \ref{main.periodic1}, \ref{main.weak1}
and \ref{veryweak}; in \S  \ref{S.prel} we introduce some preliminary constructions and some key results concerning time needed to perform a loop and space displacement with respect to $\tilde{W}(t)$, mainly borrowed from \cite{CFPRimut}; we close the section by illustrating how the results are applied to similar scenarios;
in \S \ref{S.connection}, roughly speaking, we show  that the compact sets $X^+(\tau, \TT^+, \ee^+)$ and $X^-(\tau, \TT^-, \ee^-)$ of initial
conditions giving rise to trajectories chaotic in the future and in the past respectively, depend continuously on $\tau$, see Theorems
\ref{T.connectionLa} and \ref{T.connection}:
the argument relies heavily on \cite{CFPrestate} and it is detailed in \S \ref{S.theorem2.1} in the weaker setting of Theorem
\ref{veryweak}, and then quickly adapted to the setting of Theorems \ref{main.weak1} and \ref{main.periodic1} in  \S
\ref{proof.connection2};
in \S \ref{S.mainproof} we use the intermediate value theorem   to prove
that the Melnikov condition guarantees the existence of $\tau_{\star}$ such that there is a non-empty intersection between the sets $X^+(\tau_{\star}, \TT^+, \ee^+)$
and $X^-(\tau_{\star}, \TT^-, \ee^-)$, thus proving Theorems~\ref{theoremkey} and \ref{theoremkeybis}, from which
Theorems~\ref{main.periodic1}, \ref{main.weak1}
and \ref{veryweak} easily follow;  in \S \ref{proof.Bernoulli} we show that in the setting of Theorems
\ref{main.periodic1} and \ref{main.weak1} the action of the flow of \eqref{eq-disc} on the set $\Sigma$ giving rise to chaos
is semi-conjugated to the Bernoulli shift in $\{0 ; 1 \}^{\Z}$;  finally in \S \ref{s.ex} we illustrate our hypotheses and our results by some examples.


\section{Statement of the main results}\label{S.setting}

\subsection{Basic assumptions, constants and notation}
Following \cite{CFPRimut, CFPrestate} firstly  we give a notion of solution and we collect the basic hypotheses which we assume through the whole paper together with the main
constants. We emphasize that we keep the notation as in \cite{CFPRimut} and \cite{CFPrestate}.
By a \emph{solution} of
\eqref{eq-disc} we mean a continuous, piecewise $C^r$ function
$\x(t)$ that satisfies
\begin{align}
\dot{\x}(t)=\f^+(\x(t)) + \ep \g(t,\x(t),\ep), \quad \textrm{whenever } \x(t) \in \Om^+,  \label{eq-disc+} \tag{PS$+$}\\
\dot{\x}(t)=\f^-(\x(t)) + \ep \g(t,\x(t),\ep), \quad \textrm{whenever } \x(t) \in \Om^-. \label{eq-disc-}\tag{PS$-$}
\end{align}
 Moreover, if  $\x(t_{0})$ belongs to
$\Om^{0}$ for some $t_0$,   then  we assume
either $\x(t)\in\Om^{-}$ or $\x(t)\in\Om^{+}$  for $t$ in some left neighborhood of
$t_{0}$, say $]t_{0}-\tau,t_{0}[$ with $\tau>0$.
In the first case, the
left derivative of $\x(t)$ at $t=t_0$ has to satisfy $\dot
\x(t_{0}^{-}) = \f^{-}(\x(t_0))+ \ep \g(t_0,\x(t_0),\ep)$; while in the
second case, $\dot \x(t_{0}^{-}) = \f^{+}(\x(t_0))+ \ep
\g(t_0,\x(t_0),\ep)$. A similar condition is required for the right
derivative $\dot \x(t_{0}^{+})$.
We stress that, in this paper, we do not consider solutions of equation
\eqref{eq-disc} that belong to $\Om^{0}$ for $t$ in some
nontrivial interval, i.e., sliding solutions.

\subsubsection*{Notation}
Throughout the paper we will use the following notation. We denote scalars by small letters,
e.g.\ $a$, vectors in $\R^2$ with an arrow, e.g.\ $\vec{a}$, and $n \times n$ matrices by bold letters, e.g.\ $\bs{A}$.
By $\vec{a}^*$ and $\bs{A}^*$ we mean the transpose of the vector $\vec{a}$ and of the matrix $\bs{A}$, resp.,
so that $\vec{a}^*\vec{b}$ denotes the scalar product of the vectors $\vec{a}$, $\vec{b}$.
We denote by $\|\cdot\|$ the Euclidean norm in $\R^2$, while for matrices
we use the functional norm $\|\bs{A}\|= \sup_{\|\vec{w}\| \le 1} \|\bs{A} \vec{w}\|$.
We will use the shorthand notation $\bs{f_x}=\bs{\frac{\partial f}{\partial x}}$ unless this may cause confusion.

Further, given $\delta>0$ we denote by $B(\vec{P},\delta):=\{\Q \mid \|\Q-\vec{P}\| < \delta \}$ and,
given a set $D$ and $\delta>0$, we let
\begin{equation*}
  \begin{split}
     B(D, \delta):= &  \{ \vec{Q} \in \R^2 \mid   \exists \vec{P} \in D : \, \|\vec{Q}-\vec{P} \| <  \delta \}= \cup \{ B(\vec{P}, \delta)
 \mid \vec{P} \in D \}.
  \end{split}
\end{equation*}

\begin{figure}[t!]
\begin{center}
	\epsfig{file=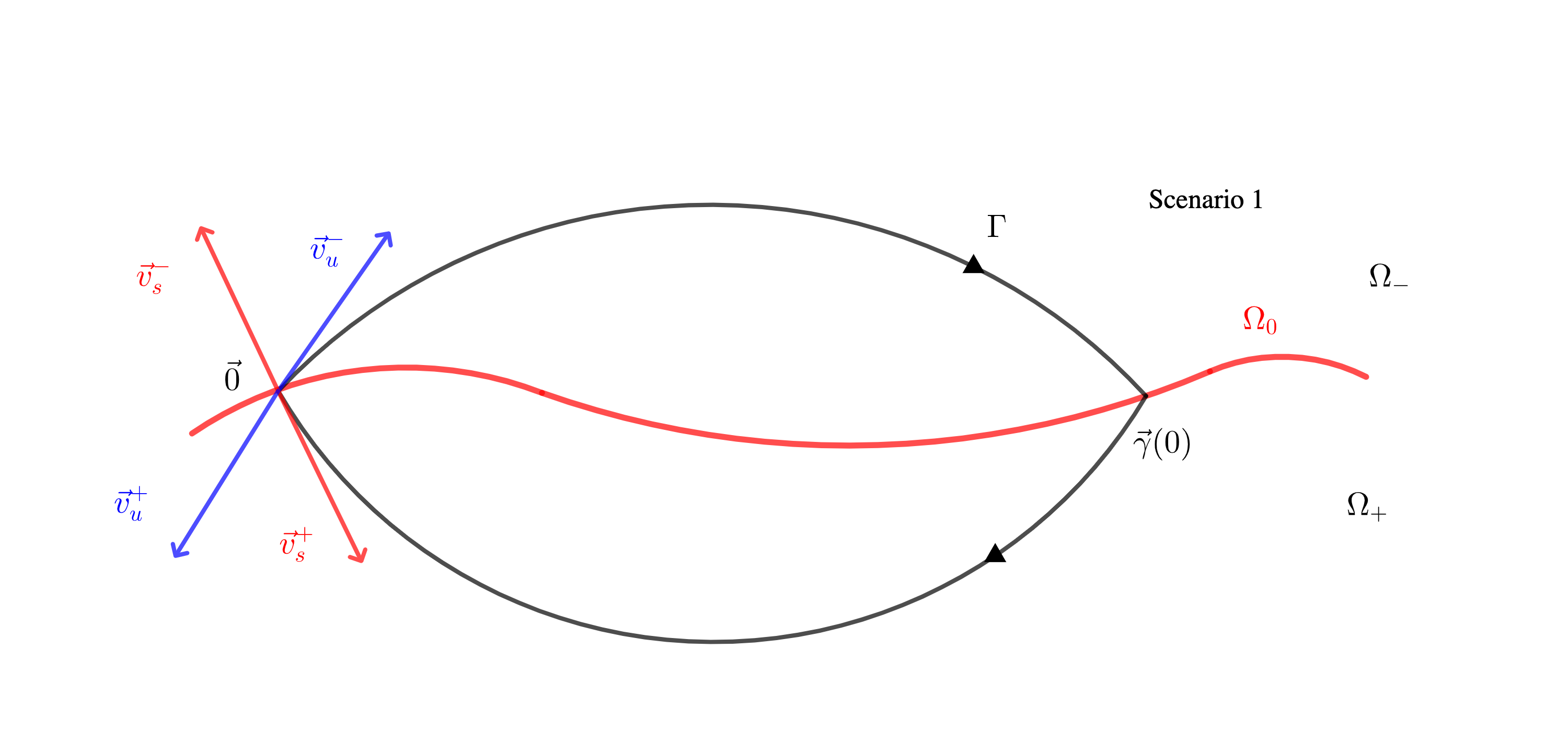, width = 10 cm}\\
	\epsfig{file=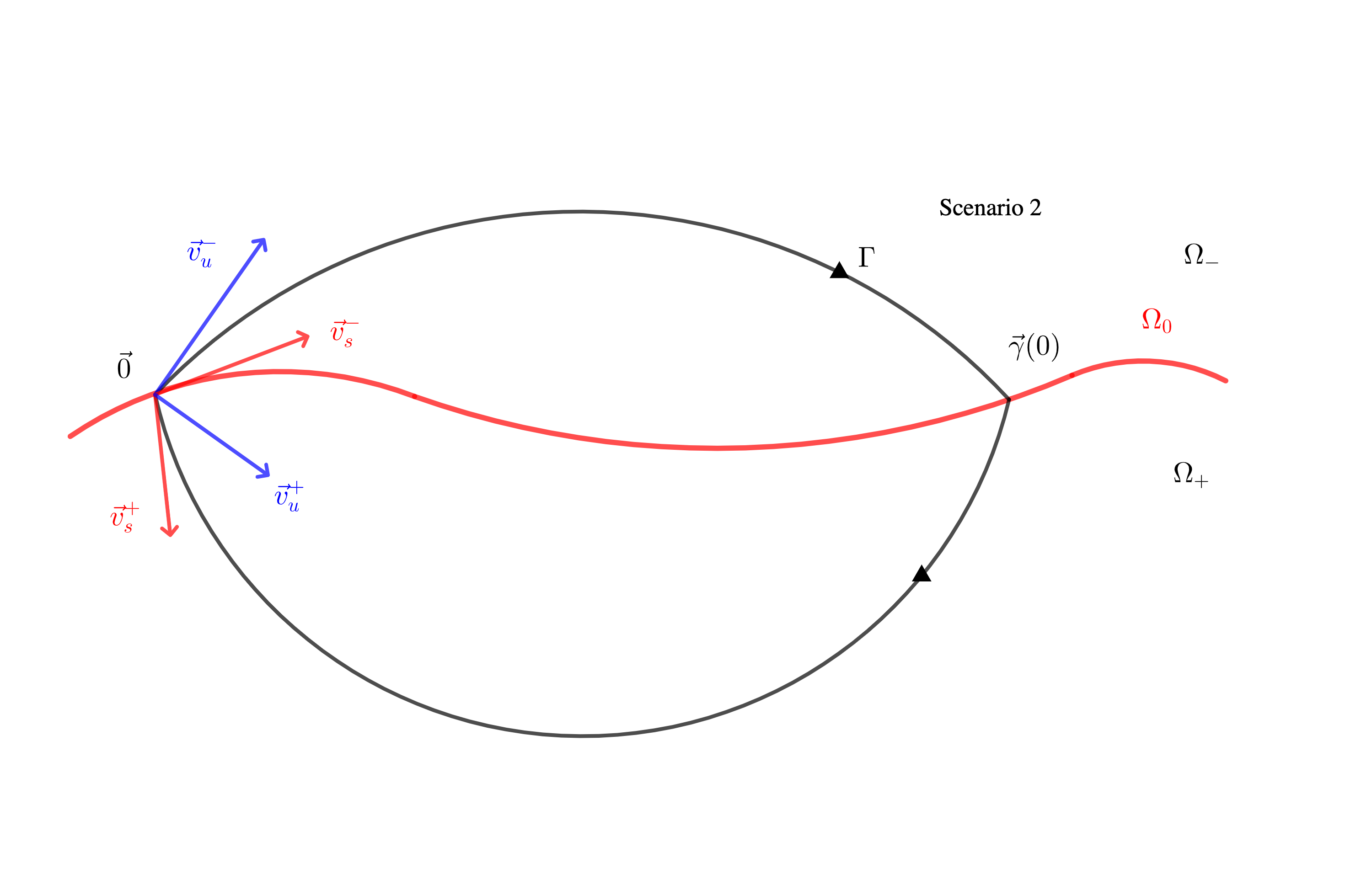, width = 10 cm}
\end{center}
\caption{Scenarios 1 and 2. In both the cases \assump{F2} holds and
we can apply our results, Theorems \ref{main.periodic1},
 \ref{main.weak1} and \ref{veryweak}, and we find chaotic patterns. See also \S \ref{S.scenarios}.
}
\label{scenario123}
\end{figure}

We list here some hypotheses which we assume through the whole paper.

\begin{description}
	\item[\assump{F0}] We have $\vec{0}\in\Omega^0$,  $\f^\pm(\vec{0})=\vec{0}$, and the eigenvalues $\la_s^\pm$, $\la_u^\pm$ of
$\bs{f_x^\pm}(\vec{0})$ are such that
$\la_s^\pm<0<\la_u^\pm$.
\end{description}

Denote by $\vec{v}_s^{\pm}$, $\vec{v}_u^{\pm}$ the normalized eigenvectors of $\bs{f_x^\pm}(\vec{0})$ corresponding to $\la_s^\pm$,
$\la_u^\pm$.
 We assume that the eigenvectors $\vec{v}_s^{\pm}$, $\vec{v}_u^{\pm}$ are not orthogonal to $\na G(\vec{0})$. To fix the ideas we require
\begin{description}
	\item[\assump{F1}]   $[\na G(\vec{0})]^*\vec{v}^{-}_u    <0<[\na G(\vec{0})]^*\vec{v}^{+}_u$,
$\qquad \qquad [\na G(\vec{0})]^*\vec{v}^{-}_s<0<[\na G(\vec{0})]^*\vec{v}^{+}_s$.
\end{description}

\begin{figure}[t!]
\begin{center}
	\epsfig{file=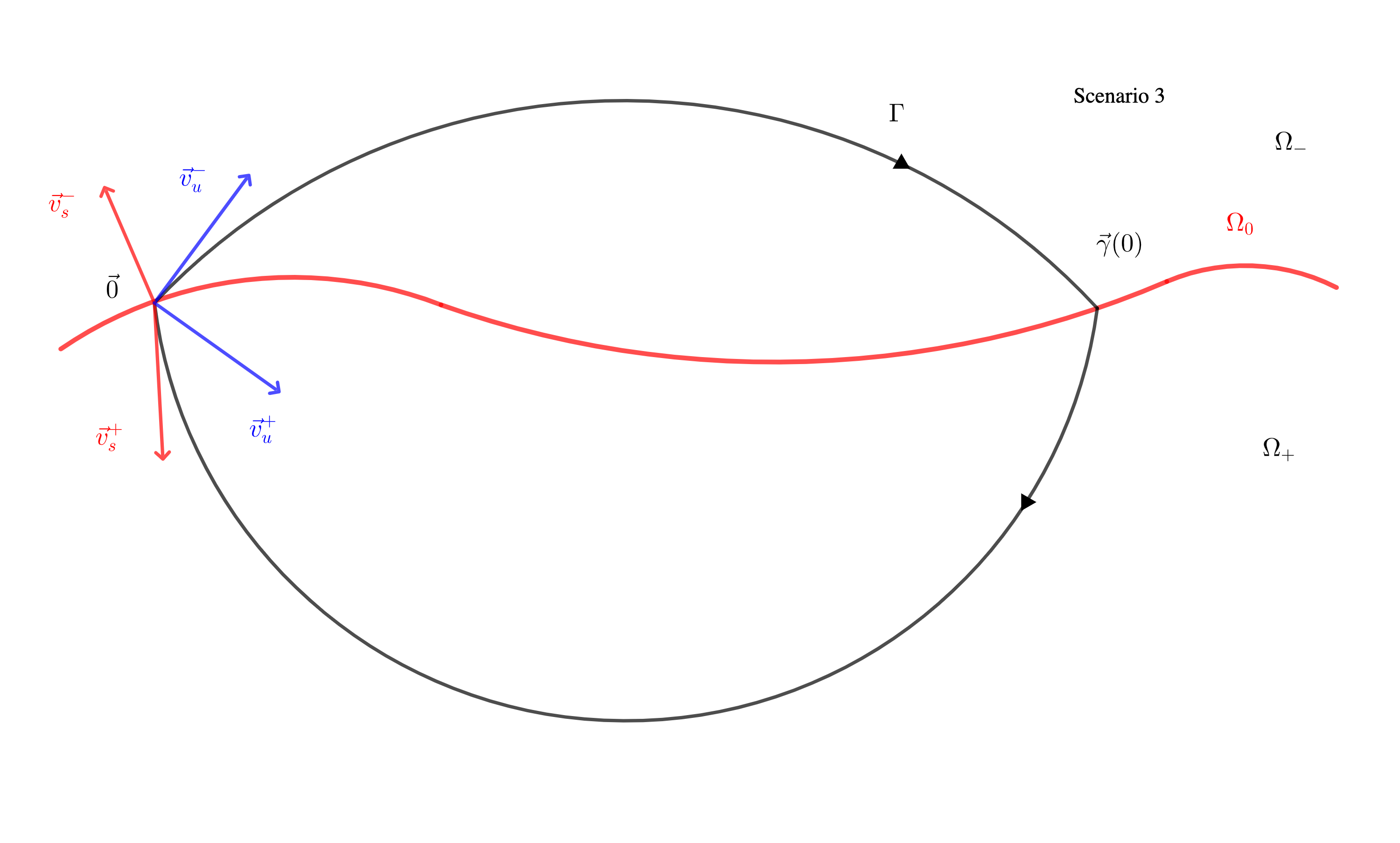, width = 9 cm} \\
	\epsfig{file=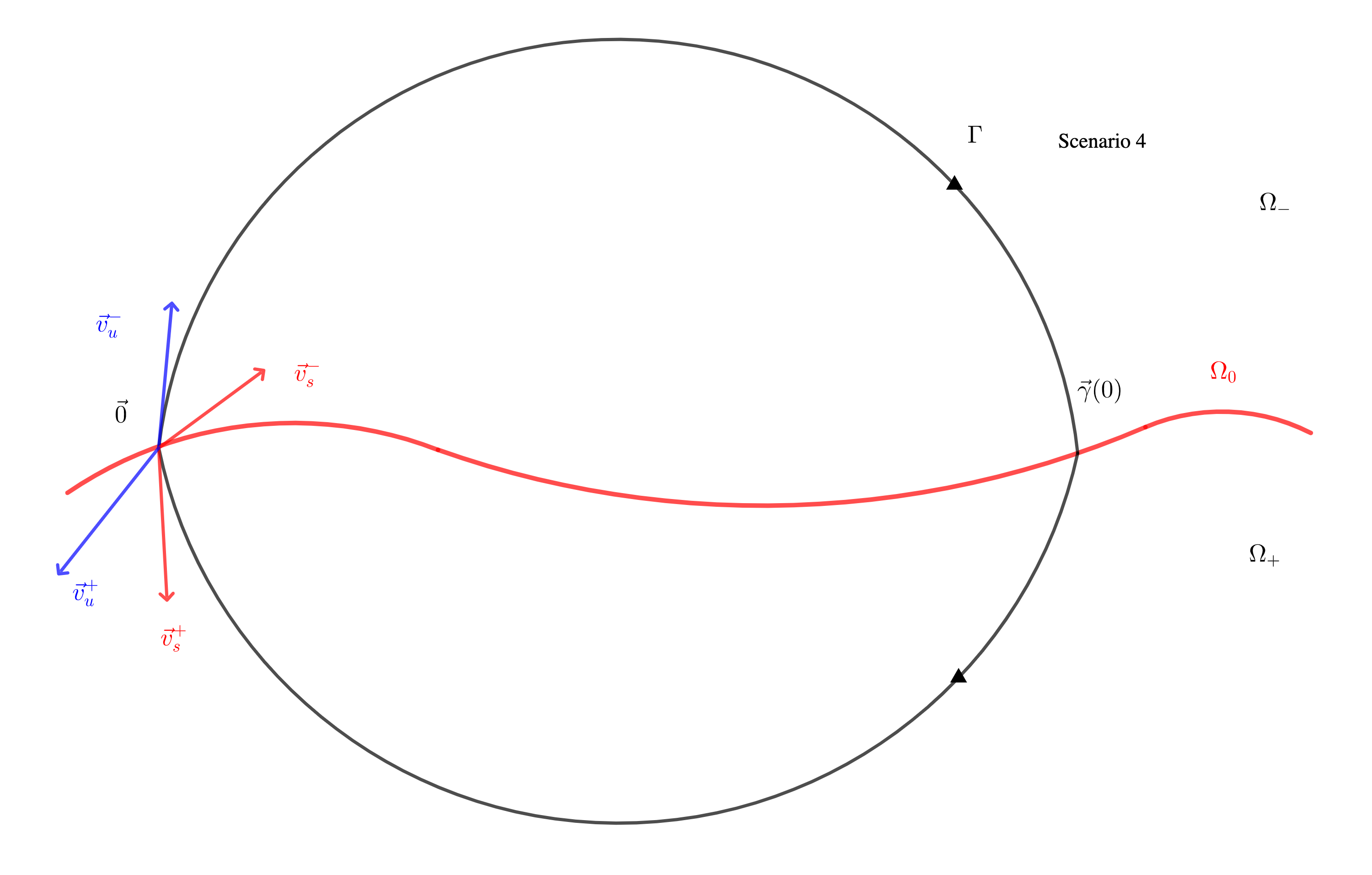, width = 9 cm}
\end{center}
\caption{In these settings \assump{F2} does not hold and sliding may occur close to the origin, so our analysis does not apply directly. Further Melnikov theory guarantees persistence of the homoclinic trajectories
  \cite{CaFr}, but chaos is forbidden  generically,  \cite{FrPo}. See also \S \ref{S.scenarios}.}
\label{scenario45}
\end{figure}

Following \cite{CFPRimut, CFPrestate} we require a further condition on the mutual positions
 of the
directions spanned by
$\vec{v}_s^{\pm}$, $\vec{v}_u^{\pm}$.

More precisely, set $\mathcal{T}_u^{\pm}:=\{ c\vec{v}_u^{\pm} \mid c \ge 0 \}$, and
denote by $\Pi_u^1$ and $\Pi_u^2$ the disjoint open sets
in which $\R^2$ is divided by the polyline $\mathcal{T}^u:=\mathcal{T}_u^+ \cup \mathcal{T}_u^-$.
We require that $\vec{v}_s^+$ and $\vec{v}_s^-$ lie on ``opposite sides'' with respect to
	$\mathcal{T}^u$.
Hence, to fix the ideas, we assume: \begin{description}
	\item[\assump{F2}]   $\vec{v}_s^+ \in \Pi_u^1$ and $\vec{v}_s^- \in \Pi_u^2$.
\end{description}
In fact  it is the \emph{mutual position of the eigenvectors} that plays a role in the argument and all our results hold
if  $\vec{v}_s^- \in \Pi_u^1$ and $\vec{v}_s^+ \in \Pi_u^2$, too: we assume \assump{F2} for definiteness, see Figure \ref{scenario123}, compare also with \cite[Remark 2.1]{CFPRimut}.

We emphasize that if \assump{F2} holds, there is no sliding on $\Om^0$ close to $\vec{0}$.
On the other hand, sliding may occur when both  $\vec{v}_s^{\pm}$ lie in
$\Pi_u^1$, or they both lie in  $ \Pi_u^2$, see Figure \ref{scenario45}, compare also with \cite[\S 3]{FrPo}.
As pointed out in  \cite{CFPRimut, CFPrestate}, the case $\vec{v}_s^- \in \Pi_u^1$ and $\vec{v}_s^+ \in \Pi_u^2$
can be handled similarly.

We assume further:
\begin{description}
	\item[\assump{K}] For $\ep=0$ there is  a unique solution $\ga(t)$ of \eqref{eq-disc} homoclinic to the origin such that
	$$\ga(t)\in\begin{cases}\Om^-,& t<0,\\\Om^0,&t=0,\\\Om^+,&t>0.\end{cases}$$
	Furthermore, $(\na G(\ga(0)))^*\f^\pm(\ga(0))>0$.
\end{description}
Recalling  the orientation of $\vec{v}_s^{\pm}$, $\vec{v}_u^{\pm}$ chosen in \assump{F1}, we assume w.l.o.g.\  that
\begin{equation}\label{ass.scenario}
\lito\frac{\dot{\ga}(t)}{\|\dot{\ga}(t)\|}=\vec{v}_u^-
\quad \mbox{ and } \quad
\lit\frac{\dot{\ga}(t)}{\|\dot{\ga}(t)\|}=-\vec{v}_s^+.
\end{equation}

Concerning the perturbation term $\g$, we assume the following:
\begin{description}
	\item[\assump{G}]  $\g(t,\x, \ep)$ is bounded together with its derivatives for any $t \in \R$, $\x \in B(\bs{\Gamma}, 1)\cap\Omega$, $\ep \ge 0$ and
$\g(t,\vec{0},\ep)=\vec{0}$ for any $t\in\R$ $\ep \ge 0$.
\end{description}
Hence, the origin is a critical point for the perturbed problem too.

Let us define the Melnikov function $\mathcal{M}:\R\to\R$  which, for planar
 piecewise-smooth systems as \eqref{eq-disc}, takes the following form.
For the computation of the Melnikov function we refer to the Appendix of \cite{CFPrestate} in which
we corrected a small error appeared in \cite{CaFr}.

\begin{equation}\label{melni-disc}
\begin{split}
   \mathcal{M}(\al) &= c^-_{\perp} \int_{-\infty}^{0} \eu^{-\int_0^t
			\tr\boldsymbol{\f_x^-}(\ga(s))ds} \f^-(\ga(t))
		\wedge \g(t+\al,\ga(t),0) dt\\
		&\quad {}+ c^+_{\perp} \int_{0}^{+\infty} \eu^{-\int_0^t
			\tr\boldsymbol{\f_x^+}(\ga(s))ds} \f^+(\ga(t))
		\wedge \g(t+\al,\ga(t),0) dt,
\end{split}
	\end{equation}
where
\[
c_{\perp}^{\pm} =
         \frac{\|\na G(\ga(0))\|}{ (\na G(\ga(0)))^* \, \f^\pm(\ga(0))} >0,
\]
and by ``$\wedge$" we denote the wedge product in $\R^2$ defined by $\vec{a} \wedge \vec{b}= a_1b_2-a_2b_1$
for any vectors $\vec{a}=(a_1,a_2)$, $\vec{b}=(b_1,b_2)$.
In fact, also in the piecewise-smooth case the function $\mathcal{M}$ is $C^r$ and, following \cite{CaFr}, we find
\begin{equation}\label{melni-disc.diff}
\begin{split}
   \mathcal{M}'(\al) &= c^-_{\perp} \int_{-\infty}^{0} \eu^{-\int_0^t
			\tr\boldsymbol{\f_x^-}(\ga(s))ds} \f^-(\ga(t))
		\wedge \frac{\partial\g}{\partial t}(t+\al,\ga(t),0) dt\\
		&\quad {}+ c^+_{\perp} \int_{0}^{+\infty} \eu^{-\int_0^t
			\tr\boldsymbol{\f_x^+}(\ga(s))ds} \f^+(\ga(t))
		\wedge  \frac{\partial\g}{\partial t}(t+\al,\ga(t),0) dt.
\end{split}
	\end{equation}

Under our assumptions, in \cite[Theorem 2.9]{CaFr}    the persistence of the homoclinic trajectory was proved
(in fact in the general $\Om \subset \R^n$, $n \ge 2$ case).

 \begin{knowntheorem}[\hspace{1sp}\cite{CaFr}]\label{TFSF2_2}%
	Assume that \assump{F0}, \assump{F1}, \assump{K} and \assump{G} are satisfied, and that there is $\al_0\in\R$ such that $\MM(\al_0)=0$ and
$\MM'(\al_0)\neq 0$.
Then  there is $\ep_{0}>0$ such that for any $0<|\ep|<\ep_0$
system \eqref{eq-disc} admits a  unique piecewise  $C^{r-1}$
solution $\x_b(t;\ep)$ bounded on $\R$, and there is a unique $C^{r-1}$ function
$\al(\ep)$, satisfying $\al(0) = \al_{0}$, such that
     \begin{equation*}
     \sup_{t\in\R}||\x_b(t+\al(\ep); \ep) - \ga(t)|| \to 0  \quad
     \hbox{as $\ep\to 0$}.
     \end{equation*}\end{knowntheorem}

Before stating the main results of the article we collect here for convenience of the reader and future reference,  the main constants which
will play a role in our argument:
\begin{equation}\label{defsigma}
\displaystyle
\begin{array}{ccc}
\sfwd_+= \frac{|\la_s^+|}{\la_u^++|\la_s^+|}, & \sfwd_-= \frac{\la_u^-+|\la_s^-|}{\la_u^-}, &
\sfwd=  \sfwd_+ \sfwd_-,  \\
  \sbwd_+=   \frac{1}{\sfwd_+}, & \sbwd_- =\frac{1}{\sfwd_-}, &  \sbwd=   \sbwd_+ \sbwd_-, \\
 \underline{\sigma} = \min \{ \sfwd_+ , \sbwd_-  \} , &  & \overline{\sigma}= \max \{ \sfwd_+ , \sbwd_-  \},
  \end{array}
\end{equation}
 \begin{equation*}
\displaystyle
\begin{array}{ccc}
  \sTfwd_+= \frac{1}{\la_u^++|\la_s^+|}, & & \sTbwd_- = \frac{1}{\la_u^- +|\la_s^-|},
  \\
\sTfwd=\frac{\la_u^-+|\la_s^+|}{\la_u^-(\la_u^++|\la_s^+|)} ,&  & \sTbwd= \frac{\la_u^{-}+|\la_s^+|}{|\la_s^+|(\la_u^-+|\la_s^-|)},\\
\underline{\Sigma}= \min \{ \sTfwd , \sTbwd  \} , &  & \overline{\Sigma}= \max \{ \sTfwd , \sTbwd  \},\\
\und{\la}= \min \{ \la_u^-; \la_u^+ ;|\la_s^-|;|\la_s^+| \} , & & \qquad \ov{\la}= \max \{ \la_u^-; \la_u^+ ;|\la_s^-|;|\la_s^+| \}.
\end{array}
\end{equation*}
 Notice that, in
particular,
$\underline{\sigma}\leq\overline{\sigma}<1$.

\begin{remark}
We will see  in Theorem \ref{key} below that trajectories starting close to $\ga(0)$ may perform a loop close to $\bs{\Gamma}$  and arrive back
close to $\ga(0)$.
  The constants $\sigma$ are used to measure the space displacement while the constants $\Sigma$ are used to measure the
  time needed to perform a loop or a part of it. We use the ``$\pm$" pedices for ``staying in $\Om^{\pm}$" and the ``fwd/bwd" apices
  for ``forward and backward time".
\end{remark}

\begin{remark}\label{Ronconstants}
  Notice that if $\la_u^-=\la_u^+=\la_u$ and $\la_s^-=\la_s^+=\la_s$ as in the smooth case, the constants
  in \eqref{defsigma} simplify as follows:
  $$\begin{array}{cc}
      \sfwd= \frac{|\la_s|}{\la_u},   & \sbwd = \frac{\la_u}{|\la_s|}=(\sfwd)^{-1}, \\
      \sTfwd = \frac{1}{\la_u} , &  \sTbwd=\frac{1}{|\la_s|}. \\
    \end{array}   $$
    In the Hamiltonian case we find $\sfwd=\sbwd=1$ and $\sfwd_+=\sbwd_-=1/2$.
\end{remark}

We also need to fix the following absolute constants $K_0$ and $\nu_0$, which depend only on the eigenvalues in \assump{F0}
  \begin{equation}\label{defK0-new}
  K_0:=    \frac{3\overline{\Sigma} }{2  \underline{\sigma}} , \qquad  \nu_0:= \max \left\{  3\overline{\sigma}-1 ; 1   \right\}.
\end{equation}

Roughly speaking, $\nu_0$ is the lower bound for $\nu$, which is used to tune the size of the sets
$\Sigma$   giving rise
to chaotic phenomena and accordingly the time needed by the trajectories to perform a loop: the larger $\nu \ge \nu_0$, the smaller the
diameter
of
$\Sigma$, the longer the time needed for a loop; this aspect will be the subject of a forthcoming paper.
 In fact in this paper the reader can always keep  $\nu=\nu_0$.
On the other hand, the constant $K_0$ is needed  to measure  the minimal time needed to perform a loop, which  goes, roughly speaking,
as
$2K_0(1+\nu) |\ln(\ep)|$, see   \eqref{TandKnu}, \eqref{TandKnunew}  below.
\begin{remark}\label{data.smooth}
 Let us observe that we can choose $\beta>0$ small enough  so that within $B( \bs{\Gamma}, \beta)$ no sliding phenomena are possible (for any
$0<\ep \le \ep_0$, e.g.\ we can choose $\beta= |\ln(\ep_0)|^{-1}$). Then it follows that local uniqueness of the solutions is ensured within the set
$B( \bs{\Gamma}, \beta)$ where all our analysis will take place, see Remarks 3.7 and 3.8 in \cite{CFPRimut}. Further continuous
dependence on initial data and parameters     of the solutions of \eqref{eq-disc} within this set is ensured.
\end{remark}

 \subsection{Recurrence assumptions and statement of the main results} \label{S.results}

 Now we detail the assumptions on the Melnikov function $\MM(\tau)$, defined in \eqref{melni-disc},
hypotheses which replace the standard
assumption that $\g$ is periodic in $t$ (and hence $\MM(\cdot)$ is periodic):  roughly speaking we ask for infinitely many sign changes of the Melnikov function $\MM(\tau)$.
 We stress once again that the same setting has been considered in
\cite{CFPrestate}, where we discussed the existence of trajectories chaotic in the future or chaotic in the past.
 \begin{description}
   \item[P1] There is a constant $\bar{c}>0$ and an increasing sequence $(b_i)$, $i \in \Z$, such that
   $b_{i+1}-b_i \ge 1/10$   and
   $$\MM(b_{2i})<-\bar{c}<0<\bar{c}<\MM(b_{2i+1}),$$
   for any $i \in \Z$.
 \end{description}
\begin{remark}
Notice that $b_i \to \pm \infty$ as $i \to \pm \infty$. Further the assumption $b_{i+1}-b_i \ge  1/10$ could be replaced by
$b_{i+1}-b_i \ge c$ where $c>0$ is a constant independent of $i$.
\end{remark}
If \assump{P1} holds, the intermediate value theorem implies that  $\MM(\tau)$ has at least one zero in $]b_k,b_{k+1}[$ for any $k
\in \Z$.

 Before stating our results for \eqref{eq-disc}, we illustrate what happens  replacing \assump{P1} by more restrictive
recurrence properties,
for clarity and in order to illustrate the novelties introduced by our approach with more usual periodicity
assumptions.
 Namely we consider the following two alternative conditions.

 \begin{description}
   \item[P2] There are  $b_0,b_1$ such that $\MM(b_0)<0<\MM(b_1)$
    and $\g(t,\x,\ep)$ is almost periodic in $t$, i.e., for any  $\varsigma>0$ there is $\bar{N}=\bar{N}(\varsigma)>0$
   such that
   $$ \|\g(t,\x,\ep)- \g(t+\bar{N},\x,\ep)\| \le \varsigma \,, \quad \textrm{for any } (t,\x,\ep) \in\R\times\Om\times\R.$$
     \item[P3]  There are $b_0,b_1$ such that $\MM(b_0)<0<\MM(b_1)$ and $\g(t,\x,\ep)$ is  $1-$periodic in $t$, i.e
   $$ \g(t,\x,\ep)= \g(t+1,\x,\ep)  \, ,\quad \textrm{for any } (t,\x,\ep) \in\R\times\Om\times\R.$$
 \end{description}

It is classically known that the periodicity and the almost periodicity of $\g$ imply the periodicity and the almost periodicity
of $\MM$, so   \assump{P3} implies \assump{P2}, which implies \assump{P1}.

We want to consider an increasing sequence of times $\mathcal{T}=(T_j)$, $j \in \Z$ such that   $\MM(T_{2j})=0$
and  $T_{j+1}-T_j$ becomes larger and larger as $\ep \to 0$, see \eqref{TandKnu} and \eqref{TandKnunew} below.
Correspondingly, we find a subsequence $\beta_j:=b_{n_j}$ such that
\begin{equation}\label{defkj}
  \beta_j:=b_{n_j} <T_j < \beta'_j:= b_{n_j +1} \, , \qquad  B_{j}=\beta'_j- \beta_j= b_{n_j+1}-b_{n_j} >0,
\end{equation}
$j \in \Z$.
 Concerning the sequence $\mathcal{T}$, when condition \assump{P1} holds
we follow two different settings of assumptions; the first is slightly more
restrictive but  includes the cases of \assump{P2}, \assump{P3} and the case where $B_j$ is bounded:  it allows to obtain more precise
(and clearer) results,
i.e., Theorems \ref{main.periodic1} and \ref{main.weak1}; in the second approach we
ask for weaker assumptions but we obtain weaker results, i.e., Theorem \ref{veryweak} (where we have a very weak control of the size of
$\alpha^{\ee}_j(\ep)$ introduced below,  and we cannot guarantee the semi-conjugation with the Bernoulli shift, unless $B_j$ is bounded).

We consider now the first setting of assumptions.

We assume first that   there may be some $j \in \Z$ such that $T_{2j}$ is an accumulation point of the zeros of $\MM(\tau)$,
so there are  $0<\dd<1 $, $1/10>\l1 >2\lzero \ge 0$   and increasing sequences $(\aup_j)$, $(\adown_j)$,
$j\in \Z$   such that
$$\beta_{2j}<\aup_j<T_{2j}-\lzero<T_{2j}<T_{2j}+\lzero<\adown_{j}<\beta'_{2j}$$  and
\begin{equation}\label{minimal}
\begin{split}
 & |\MM(\aup_j)|=|\MM(\adown_j)|= \dd \bar{c}, \qquad \MM(\aup_j) \MM(\adown_j) <0, \\ & \MM(\tau) \ne 0 \quad \forall \tau \in [\aup_j,
  \adown_j] \setminus [T_{2j}-\lzero, T_{2j}+\lzero],\\
  & \textrm{$0<\adown_j-\aup_j \le \l1$,
for any $j \in \Z$.}
  \end{split}
\end{equation}
We emphasize that  both $\l1$ and $\lzero$ are independent of $j \in \Z$.

In fact \eqref{minimal} is enough to construct a chaotic pattern, but
we obtain sharper results  if we assume that there is a strictly monotone increasing and continuous function
$\omega_M(h):[0,\l1] \to [0,+\infty[$ (independent of $j$) such that
$\omega_M(0)=0$ and
\begin{equation}\label{isogen}
\begin{split}
      |\MM(T_{2j}- \lzero - h)| & \ge \omega_M(h)  \quad  \textrm{for any } \;  (T_{2j}- \lzero - h)\in  [\aup_j, T_{2j}- \lzero],  \\
  |\MM(T_{2j}+ \lzero + h)| & \ge \omega_M(h)  \quad  \textrm{for any } \;  (T_{2j}+ \lzero + h)\in  [ T_{2j}+ \lzero, \adown_j],
\end{split}
\end{equation}
and any $j \in \Z$.

Notice that \eqref{minimal} and \eqref{isogen} are always satisfied if \assump{P2} or \assump{P3} holds.
In fact \eqref{minimal} and \eqref{isogen} can be satisfied even if the sequence $(B_j)$ becomes unbounded, but they require  a control
from below on the ``slope" of the Melnikov function ``close to its zeros".

  We stress that in the easier and most significant case where $T_{2j}$ is an isolated zero for $\MM(\cdot)$ for any $j$, we
  can assume  $\lzero=0$ so that \eqref{minimal} simplifies as follows
  \begin{equation}\label{defbetaj}
\begin{split}
    &  |\MM(\aup_j)|=|\MM(\adown_j)|= \dd \bar{c}, \qquad \MM(\aup_j) \MM(\adown_j) <0, \\
     & \MM(\tau) \ne 0 \qquad \forall \tau \in [\aup_j,
  \adown_j] \setminus \{T_{2j} \}.
\end{split}
\end{equation}
  Further in this case   \eqref{isogen} reduces to
\begin{equation}\label{isolated}
  |\MM(T_{2j}+ h)| \ge \omega_M(|h|)  \quad  \textrm{for any }  (T_{2j}+ h)\in [\aup_j, \adown_j]\, , \quad h \ne 0,
\end{equation}
and any $j \in \Z$.

When we assume the classical hypothesis that
there is   $C >0$ such that
\begin{equation}\label{nondeg}
  \MM(T_{2j})=0   \quad   \text{and} \quad |\MM'(T_{2j})|>C  \;  \text{for any $j \in \Z$,}
\end{equation}
then \eqref{minimal} and \eqref{isolated} hold and we may assume  simply $ \omega_M(h) = \tfrac{C}{2} h$
 and $\l1 = 2\frac{\bar{c}}{C}  \delta<\tfrac{1}{10}$.

When  \assump{P1} and \eqref{minimal} hold we need the following condition concerning the sequences $(T_j)$ and the time gap
$T_{j+1}-T_j$:
\begin{equation}\label{TandKnu}
\begin{split}
  &  \MM(T_{2j})=0, \quad T_0 \in [b_0,b_1],\\
& T_{j+1}-T_j > \l1 + K_0 (1+\nu) |\ln(\ep)| ,   \quad \textrm{ for any $j \in \Z$}
  \end{split}
\end{equation}
and for any $\nu \ge \nu_0$, where the absolute constants $K_0$ and $\nu_0$ are as in \eqref{defK0-new}.

\begin{remark}
 If \assump{P3} holds and there is $t_0$ such that $\MM(t_0)=0 \ne \MM'(t_0)$,
  then we can choose  $\TT=(T_j)$, $T_j=t_0+j \ogap$, where $\ogap=     \left[  K_0(1+\nu) |\ln(\ep)|    +2 \right]$, where $[\cdot]$
  denotes the integer part,
  so that it satisfies \eqref{TandKnu} and \eqref{nondeg}.
\end{remark}

In the second approach we drop \eqref{minimal}, but
we need to ask for a slightly more restrictive condition on  the time gap
$T_{j+1}-T_j$, i.e., we assume that for a given  $\nu \ge \nu_0$ we have:
\begin{equation}\label{TandKnunew}
\begin{split}
  &  \MM(T_{2j})=0, \quad T_0 \in [b_0,b_1],\\
& T_{j+1}-T_j >    \max\{ B_{j+1}; B_{j}\}+ K_0 (1+\nu) |\ln(\ep)| ,   \textrm{ for any $j \in \Z$.}
  \end{split}
\end{equation}

 Now we state our first result concerning the existence of a chaotic pattern which takes place in the whole of $\R$,
obtained assuming  that the non-degeneracy assumption \eqref{nondeg} holds.  The  proofs of the results of this section are postponed to Section \ref{S.mainproof}.

\begin{theorem}\label{main.periodic1}
	Assume    $\f^{\pm}$ and $\g$ are $C^r$, $r \ge 2$ and that \assump{F0}, \assump{F1}, \assump{F2}, \assump{K} and \assump{G} hold true; assume further
\assump{P1}.
Then we can choose $\ep_0$ small enough so that for any $0< \ep \le \ep_0$
and any increasing sequence $\mathcal{T}=(T_j)$, $j \in \Z$  satisfying  \eqref{TandKnu} and \eqref{nondeg},
 there is $\tilde{c}^*>0$ (independent of $\ep>0$)
for which the following holds.
\\
For any  $\ee=(\ee_j) \in \EE$,
 there is a compact set $X(\ee, \mathcal{T})$
   such that if $\xxi\in X(\ee, \mathcal{T})$ the trajectory $\x(t,T_{0}; \xxi)$    satisfies \textbf{C}, i.e.
  \begin{description}
  \item[C]
	if $\ee_j=1$ then
	\begin{equation}\label{ej=1}
	\begin{split}
	&   \|\x(t,T_{0}; \xxi)-\ga(t-T_{2j}) \| \le \tilde{c}^*\ep \quad \textrm{when $t \in [T_{2j-1}, T_{2j+1} ]$,}
	\end{split}
	\end{equation}
	while  if $\ee_j=0$ we have
	\begin{equation}\label{ej=0}
	\|\x(t,T_{0}; \xxi)  \| \le \tilde{c}^*\ep  \quad \textrm{when $t \in [T_{2j-1}, T_{2j+1} ] $}
	\end{equation}
	for any $j \in \Z$.
\end{description}
Further, the set
$\Sigma(\TT) = \cup_{\ee \in \EE} X(\ee, \mathcal{T}) \subset [B(\ga(0), \tilde{c}^* \ep) \cup
B(\vec{0}, \tilde{c}^* \ep)]$.
\end{theorem}


We emphasize that the flow of \eqref{eq-disc} on $\Sigma$   is semi-conjugated to the Bernoulli shift on $\EE$, see Theorem \ref{thm.bernoulli}  for more details.

\begin{remark}\label{explainP}
In Theorem   \ref{main.periodic1} we require a lower bound
for $ T_{j+1}-T_j$, but not an upper bound: this allows a big flexibility in the choice of $\TT$.
 Hence,
even if we assume periodicity, i.e.\ \assump{P3}, we can consider gaps $ T_{j+1}-T_j$ which are not constant, e.g.\ periodic but even aperiodic or unbounded; see Section~\ref{proof.Bernoulli} for more details.
In fact this possibility can be obtained also through the approach
introduced by the works by Battelli and Fe\v ckan in a similar context, see e.g.\ \cite{BF10,BF11,BF12}.
However this point is not discussed in those papers.
\end{remark}

Now we state some weaker results concerning the case where $\f^{\pm}$ and $\g$ are just $C^r$, $r>1$  (i.e., the derivatives are H\"older
continuous), and we do not
assume the non-degeneracy of the zeroes of $\MM$. In the next Theorem \ref{main.weak1} we drop \eqref{nondeg} and we just require \eqref{minimal}.

\begin{theorem}\label{main.weak1}
	Assume that   $\f^{\pm}$ and $\g$ are $C^r$ with  $r > 1$,   and that \assump{F0}, \assump{F1}, \assump{F2}, \assump{K} and \assump{G} hold true;
assume further    \assump{P1}.
Then we can choose $\ep_0$ small enough so that for any $0< \ep \le \ep_0$
and any increasing sequence $\mathcal{T}=(T_j)$, $j \in \Z$,  satisfying  \eqref{TandKnu} and \eqref{minimal},
 there are $c^*>0$ (independent of $\ep>0$) and  a continuous increasing function $\omega$ satisfying
  $\omega(0)=0$
so that we have the following.
For any $\ee=(\ee_j) \in \EE$,
there is a compact set $\bar{X}(\ee, \mathcal{T})$ and a sequence $(\al_j^{\ee}(\ep))$, $j \in \Z$,
such that $|\al_j^{\ee}(\ep)| \le \omega(\ep)$ and for any $\xxi(\ee) \in \bar{X}(\ee, \mathcal{T})$
  the trajectory $\x(t,T_0 + \al_0^{\ee}(\ep); \xxi(\ee))$ has the property $\mathbf{C_e}$, i.e.
  \begin{description}
  \item[$\mathbf{C_e}$]
	if $\ee_j=1$, then
	\begin{equation}\label{ej=1w}
	\begin{split}
	&   \|\x(t,T_0+ \al_0^{\ee}(\ep); \xxi(\ee) )-\ga(t-T_{2j}-\al_j^{\ee}(\ep)) \| \le c^*\ep  \\
 & \textrm{when $t \in [T_{2j-1}, T_{2j+1} ]$ }
	\end{split}
	\end{equation}
	while if $\ee_j=0$, we have
	 \begin{equation}\label{ej=0w}
\begin{split}
	&	\|\x(t,T_0+ \al_0^{\ee}(\ep); \xxi(\ee))  \| \le c^*\ep \\
 & \textrm{when $t \in [T_{2j-1}, T_{2j+1} ]$ }
	\end{split}
	\end{equation}
	for any $j \in \Z$.
\end{description}
Moreover, we have the following estimates for $\alpha_j^{\ee}(\ep)$:
\begin{itemize}
  \item  Assume \eqref{minimal}, then $|\alpha_j^{\ee}(\ep)| \le  \l1$ for any $j \in \Z$.
  \item Assume \eqref{isolated}, then there is   a continuous increasing function $\omega_{\alpha}(\ep)$ (independent of $\ee$) satisfying
  $\omega_{\alpha}(0)=0$, such that $|\alpha_j^{\ee}(\ep)| \le \omega_{\alpha}(\ep)$ for any $j \in \Z$.
\end{itemize}
\end{theorem}

\begin{cor}\label{c.sigma}
  Let the assumption of Theorem \ref{main.weak1} be satisfied and denote by
  \begin{equation}\label{defX}
  \begin{split}
       X(\ee,\TT):= & \{ \xxi_0(\ee)=\x(T_0, T_0+ \alpha_{0}^{\ee}(\ep); \xi(\ee)) \mid \xi(\ee) \in \bar{X}(\ee, \TT) \}, \\
        \Sigma(\TT):= & \cup_{\ee \in \EE} X(\ee, \TT).
  \end{split}
  \end{equation}
  Then $X(\ee,\TT)$ is compact for any $\ee \in \EE$.
Further,\\
if we assume \eqref{isolated} then
$\Sigma(\TT) \subset [B(\ga(0), \omega_{\al}(\ep)) \cup B(\vec{0}, c^* \ep)]$;
\\ if we assume \eqref{minimal} then
$\Sigma(\TT) \subset [B(\bs{\Gamma^{\l1}}, c^* \ep) \cup B(\vec{0}, c^* \ep)]$
where
\begin{equation}\label{ga1}
	\bs{\Gamma^{\l1}}:= \{ \ga(t) \mid |t| \le \l1     \}.
\end{equation}
\end{cor}

 If the assumptions of either Theorem \ref{main.periodic1} or Theorem \ref{main.weak1} are satisfied, the flow of \eqref{eq-disc} on $\Sigma$   is semi-conjugated to the Bernoulli shift on $\EE$, see Theorem~\ref{thm.bernoulli}   for more details.

\begin{remark}\label{r.nonunique}
 We stress that in the smooth case (but also in the piecewise-smooth case if $\vec{0} \not\in \Omega^0$, see \cite{BF11}), the
classical results found in literature, see e.g. \cite{G1,Pa84,Sch,WigBook,St},
 assume  that  the Melnikov function  has non-degenerate simple zeros and  state that
$X(\ee, \TT)$ is a singleton.
However
 in Theorem  \ref{main.periodic1}  we may have two (or more) different points $\{\xxi_1,  \xxi_2\}\subset X(\ee,\mathcal{T})$.
  We believe this lack of uniqueness is a drawback of the use of
the intermediate value theorem
in our construction, so we think it is just a technical issue.
 Possibly, this problem may be overcome by the use  of the implicit function theorem
as in the approach by Battelli-Fe\v ckan e.g.\ in \cite{BF10,BF11,BF12}: the proof of the uniqueness of $\xxi \in  X(\ee, \TT)$ will be
the object
of future investigations.

In fact, this lack of uniqueness is the reason for which in Theorem \ref{thm.bernoulli} we just get
semi-conjugation
 instead of conjugation as one may wish.

 We recall, however, that our approach is most probably optimal for the general case where we do not require $\mathcal{T}$ to satisfy
 \eqref{nondeg}:
 if the function $\MM(\tau)$ may be null in some intervals, uniqueness should not be expected. An alternative approach for this degenerate
 case, based on degree theory,
 has been developed in \cite{BF02C}, already in dimension $n \ge 2$ but just in the smooth case.
\end{remark}

\begin{remark}\label{regularity2}
Let us stress, once again, that in Theorem \ref{main.periodic1}  we require the  regularity assumptions which are
standard in this
context, i.e.\ $\f^{\pm}$ and $\g$ need to be $C^r$ with $r \ge 2$, however, in Theorem \ref{main.weak1}  we can ask
simply for  $r>1$.
\end{remark}

Let us set
\begin{equation}\label{L0bis}
 L^0 =  L^0(\delta)  := \{ \vec{Q} \in \Omega^0 \mid \|\vec{Q}- \ga(0) \| <\delta \}.
\end{equation}
The trajectory
$\x(t, T_0; \xxi )$ crosses $\Om^0$ close to the origin exactly once between   two consecutive 1s of the sequence $\ee$. Namely we have the
following  two results borrowed from \cite{CFPrestate}.
\begin{remark}\label{consecutive}
 Let  the assumptions of Theorem \ref{main.weak1}  be satisfied.
Fix $\ee \in \EE$ and let $j'<j''$ correspond to two consecutive 1s of $\ee$, i.e., $\ee_{j'}=1=\ee_{j''}$ and $\ee_k=0$ for any $j'<k<j''$
(if any).
Fix $\xxi \in \bar{X}(\ee, \mathcal{T})$; then there is  $\hat{T}(\xxi)$ such that
  $\hat{T}(\xxi) \in ]T_{2j'},T_{2j''}[$ and
the trajectory $\x(t, T_0+\alpha_0^{\ee}(\ep); \xxi)$
crosses transversely $\Om^0$ close to $\ga(0)$ (in $L^0(c^* \ep)$)
at $t= T_{2j'}+\alpha_{j'}^{\ee}(\ep), T_{2j''}+\alpha_{j''}^{\ee}(\ep)$, and close to the origin
(in $B(\vec{0},c^*\ep)$) at $t=\hat{T}(\xxi)$; further $\x(t, T_0+\alpha_0^{\ee}(\ep); \xxi ) \in \Om^+$ if $t \in
]T_{2j'}+\alpha_{j'}^{\ee}(\ep), \hat{T}(\xxi)[$ and
$\x(t, T_0+\alpha_{0}^{\ee}(\ep); \xxi ) \in \Om^-$ if $t \in ]\hat{T}(\xxi),T_{2j''}+\alpha_{j''}^{\ee}(\ep)[$.
Finally,
$$ \hat{T}(\xxi)-T_{2j'}>  K_0 (1+\nu) |\ln(\ep)| \quad \text{and} \quad T_{2j''}-\hat{T}(\xxi) >  K_0 (1+\nu)
|\ln(\ep)| .$$
An analogous remark holds for Theorem   \ref{main.periodic1}.
\end{remark}

\begin{cor}\label{r.all}
In Theorem \ref{main.weak1}, if \eqref{nondeg} holds, we can suppress the dependence on $\al_j^{\ee}(\ep)$  of the
relations in \eqref{ej=1w},
but paying the price of a loss of precision of the estimate. That is,  we can replace
\eqref{ej=1w}  by
	\begin{equation}\label{ej=1w.all}
	\begin{split}
	&   \|\x(t,T_0; \xxi_0)-\ga(t-T_{2j}) \| \le \omega(\ep) \quad \textrm{when $t \in [T_{2j-1}, T_{2j+1} ]$.}
	\end{split}
	\end{equation}
\end{cor}

 We conclude the section with a result in which we drop \eqref{minimal} paying the price
of a loss in the precision of the estimates of $\alpha_j^{\ee}(\ep)$, which may even be unbounded.
In Theorem \ref{veryweak} below, condition \eqref{TandKnu} is replaced by the slightly stronger
 \eqref{TandKnunew}.

\begin{theorem}\label{veryweak}
	Assume that   $\f^{\pm}$ and $\g$ are $C^r$ with  $r > 1$,   and that \assump{F0}, \assump{F1}, \assump{F2}, \assump{K} and \assump{G} hold true;
assume further    \assump{P1}.
Then we can choose $\ep_0$ small enough so that for any $0< \ep \le \ep_0$
and any increasing sequence $\mathcal{T}=(T_j)$, $j \in \Z$,  satisfying  \eqref{TandKnunew},
 there is $c^*>0$ (independent of $\ep>0$ and  of the sequence $(T_j)$)
so that we have the following.
For any $\ee=(\ee_j) \in \EE$,
there is a compact set $\bar{X}(\ee, \mathcal{T})$ and a sequence $(\al_j^{\ee}(\ep))$, $j \in \Z$,
such that $|\al_j^{\ee}(\ep)| \le B_{2j}$ and for any $\xxi(\ee) \in \bar{X}(\ee, \mathcal{T})$
  the trajectory $\x(t,T_0+ \al_0^{\ee}(\ep); \xxi(\ee))$ has the property
  $\mathbf{C_e}$.
\end{theorem}

Note that under the assumptions of Theorem \ref{veryweak} the flow of \eqref{eq-disc} on $\Sigma$ need not be properly semi-conjugated to the Bernoulli shift on $\EE$, see
Section \ref{proof.Bernoulli} for more details.

\begin{remark}\label{remP1}
We stress that \assump{P1} does not require an upper bound on $B_j$ when $|j| \to +\infty$.
Further, the self-similarity required on $\MM$ is very weak and we do not need any non-degeneracy of the zeros of the Melnikov function
$\MM(\tau)$, which may be null in some intervals. Hence with our results we can deal with, e.g.:
\begin{equation}\label{e.remP1}
\begin{split}
&     \MM_1(\tau)= 3\sin ( \tau)+ k(\tau), \\
    & \MM_2(\tau)= 3\sin (\sqrt[3]{1+\tau^2})+ k(\tau)
\end{split}
\end{equation}
where $|k(\tau)| \le 2$ is arbitrary, and can be regarded as a noise on which we have little information
and does not really need to be small.
\\
Notice that $\MM_1(\tau)$  satisfies \eqref{minimal} and \eqref{isogen},
  while
$\MM_2(\tau)$ does not even satisfy
\eqref{minimal}. So in the former case
we can apply Theorem \ref{main.weak1} and in the latter just Theorem \ref{veryweak}.

Further notice that if we assume $k(\tau) \equiv 0$, in the former case we can apply Theorem \ref{main.periodic1}, while in the latter we can apply again just Theorem \ref{veryweak}, since we do not have a
uniform lower bound on     $|\MM'(T_j)|$.
\end{remark}
\begin{remark}\label{toaddinPrestate}
  As pointed out in \cite{CFPrestate} classically the recurrence conditions such as periodicity or almost periodicity, are required directly on $\g$ and
  then inherited by $\MM$. Here we prefer to ask for conditions on $\MM$ (which in general are more difficult to be verified on $\g$)
  in order to include the ``large perturbations" as $k(\tau)$ in \eqref{e.remP1}. We think this might be useful for application since
  it is possible that the perturbation $\g$ is made up by a periodic part of which we have a control and a ``noise", possibly not small, of which
  we just know the size: sometimes this allows to write  $\MM$ as a periodic part and a ``noise" as in \eqref{e.remP1} which does not need to be small, see   Examples \ref{exgen0} and \ref{exgen2}.
\end{remark}

\subsection{Some remarks concerning similar scenarios} \label{S.scenarios}

Now, following again \cite{CFPRimut}, we need to distinguish between several possible scenarios to which our results apply: see Figures \ref{scenario123} and \ref{scenario45}. Let us denote by $E^{\textrm{in}}$ the bounded open set enclosed by $\bs{\Gamma}$ and by
  $E^{\textrm{out}}$ the unbounded open set outside~$\bs{\Gamma}\cup E^{\textrm{in}}$.
 \begin{description}
\item[Scenario 1] Assume \assump{K} and that there is $\rho>0$ such that $d\vec{v}_u^+ \in E^{\textrm{out}}$, $d\vec{v}_s^- \in
    E^{\textrm{out}}$ for any $0<d<\rho$.
\item[Scenario 2] Assume \assump{K} and that there is $\rho>0$ such that $d\vec{v}_u^+ \in E^{\textrm{in}}$, $d\vec{v}_s^- \in
    E^{\textrm{in}}$ for any $0<d<\rho$.
\end{description}
Notice that  \assump{F2} holds in both Scenarios 1 and 2, see \cite{CFPRimut} for details.

Theorems \ref{main.periodic1},  \ref{main.weak1}  and \ref{veryweak} apply in both the scenarios, with
no changes.
  To fix the ideas in the whole paper we
focus on Scenario 1, but everything can be adapted without effort to Scenario 2 as well: the main difference is given
in Remark \ref{Rscenarios12} below.

We emphasize that, if we remove \assump{F2}, i.e.\ if we allow sliding trajectories along $\Om^0$ close to the origin,
we have persistence of the homoclinic trajectory to perturbation but chaos is forbidden. In fact, as was shown in \cite{FrPo}, in this case we have different phenomena which cannot be present in a smooth setting.
\begin{description}
\item[Scenario 3]
Assume \assump{K} and that there is $\rho>0$ such that $d\vec{v}_u^+ \in E^{\textrm{in}}$, $d\vec{v}_s^- \in E^{\textrm{out}}$ for any
$0<d<\rho$, so
    \assump{F2} does not hold.
\item[Scenario 4]
Assume \assump{K} and that there is $\rho>0$ such that $d\vec{v}_u^+ \in E^{\textrm{out}}$, $d\vec{v}_s^- \in E^{\textrm{in}}$ for any
$0<d<\rho$, so
    \assump{F2} does not hold.
\end{description}

Finally we emphasize that Scenarios 1, 2    may take place in the smooth case too: in this case we can consider $\Om^0$ as a generic curve
passing
through the origin. However, Scenarios 3, 4   require \eqref{eq-disc} to be   piecewise-smooth and do not have a smooth counterpart.

 \emph{In this paper we will just consider Scenario 1 for definiteness,
 although Scenario~2 can be handled in a similar way}.

 \section{Preliminary construction and lemmas}\label{S.prel}

Now we define  the stable and the unstable leaves $W^s(\tau)$ and $W^u(\tau)$ of \eqref{eq-disc}.
 In this section we borrow some facts from \cite{CFPRimut} and we mainly follow its notation.
 We refer the reader to \cite{CFPRimut} for more details.

Assume \assump{F0}, \assump{F1}, \assump{K} and $\tau\in\R$;
following \cite[\S 2]{CFPrestate}, see also\cite[\S 2]{CFPRimut} which is based on \cite[Theorem 2.16]{Jsell},  we can define the global stable and unstable
 leaves as follows:
\begin{equation}\label{mm}
\begin{array}{c}
W^u(\tau) := \{ \P \in \R^2 \mid    \lito \x(t,\tau;\P)=\vec{0} \}, \\
W^s(\tau) := \{ \P \in \R^2 \mid   \lit  \x(t,\tau;\P)=\vec{0} \}.
\end{array}
\end{equation}
In fact  $W^u(\tau)$ and $W^s(\tau)$ are $C^r$ immersed $1$-dimensional manifolds if \eqref{eq-disc} is smooth.

Follow $W^u(\tau)$ (respectively $W^s(\tau)$) from the origin towards
$L^0(\sqrt{\ep})$: it intersects $L^0(\sqrt{\ep})$ transversely in a point denoted by $\P_u(\tau)$
(respectively by $\P_s(\tau)$), see Figure \ref{LinL0}.
We denote by $\tilde{W}^u(\tau)$ the branch of $W^u(\tau)$ between the origin and $\P_u(\tau)$
(a path), and by $\tilde{W}^s(\tau)$ the branch of $W^s(\tau)$ between the origin and $\P_s(\tau)$, in both the cases including the endpoints.
Since $\tilde{W}^u(\tau)$ and $\tilde{W}^s(\tau)$ coincide with $\bs{\Gamma} \cap (\Om^- \cup \Om^0)$
and $\bs{\Gamma} \cap (\Om^+ \cup \Om^0)$ if $\ep=0$, respectively, and vary in a $C^r$ way,
 we find $\tilde{W}^u(\tau)\subset  (\Om^- \cup \Om^0)$ and
$\tilde{W}^s(\tau)\subset  (\Om^+ \cup \Om^0)$ for any $\tau \in \R$ and any $0\le \ep \le \ep_0$ (see  \cite{CFPRimut}, which is based on \cite[Theorem 2.16]{Jsell}
or  \cite[Appendix]{mFaS}). Hence $\tilde{W}^u(\tau)$ and $\tilde{W}^s(\tau)$ are $1$-dimensional $C^r$ manifolds and  $\P_u(\tau)$ and $\P_s(\tau)$
are $C^r$ in $\ep$ and $\tau$ even when \eqref{eq-disc} is just piecewise $C^r$.
Further
\begin{equation}\label{proprieta}
\begin{gathered}
 \Q_u \in \tilde{W}^u(\tau)   \Rightarrow  \x(t,\tau; \Q_u) \in \tilde{W}^u(t) \subset (\Om^- \cup \Om^0) \quad \textrm{for any $t\le
 \tau$},\\
\Q_s \in \tilde{W}^s(\tau) \Rightarrow  \x(t,\tau; \Q_s) \in \tilde{W}^s(t) \subset (\Om^+ \cup \Om^0) \quad \textrm{for any $t\ge \tau$}.
\end{gathered}
\end{equation}

We set
\begin{equation}\label{tildeW}
\tilde{W}(\tau):= \tilde{W}^u(\tau) \cup \tilde{W}^s(\tau),
\end{equation}
and we denote by $\tilde{V}(\tau)$ the compact set enclosed by $\tilde{W}(\tau)$ and the
branch of $\Om^0$ connecting $\P_u(\tau)$ and $\P_s(\tau)$.

 \begin{remark}\label{Rscenarios12}
  Assume  the setting of either one of Theorems \ref{main.periodic1},  \ref{main.weak1} and \ref{veryweak} and let $\xxi \in \Sigma$ (cf.~Corollary \ref{c.sigma}).
  Then if we are in Scenario 1 the trajectory $\x(t,T_0; \xxi) \in \tilde{V}(t)$ for any $t \in \R$, while if we are in Scenario 2,
  then
  $\x(t,T_0; \xxi) \in \tilde{\wedge}(t)$ for any $t \in \R$, where $\tilde{\wedge}(t)$ is the unbounded closed set defined by $\tilde{\wedge}(t)=\overline{\R^2\setminus \tilde{V}(t)}$.
\end{remark}

Recall that we are assuming Scenario 1,
and we introduce the following notation.

We denote by $A^{\textrm{fwd}}(\tau)$ and   $B^{\textrm{fwd} }(\tau)$
the two branches in which $\vec{P}_s(\tau)$ splits $L^0(\sqrt{\ep})$ and by
$A^{\textrm{bwd}}(\tau)$ and  $B^{\textrm{bwd} }(\tau)$
 the two branches in which $\vec{P}_u(\tau)$ splits $L^0(\sqrt{\ep})$; we assume that  the whole of  $A^{\textrm{fwd}}(\tau)$
 and $A^{\textrm{bwd}}(\tau)$ are contained in $\tilde{V}(\tau)$ while at least parts of  $B^{\textrm{fwd}}(\tau)$
 and $B^{\textrm{bwd}}(\tau)$ lie outside $\tilde{V}(\tau)$.

    \begin{knownlemma}\label{L.loop}\cite[Lemma 3.6]{CFPRimut}
  Assume \assump{F0}, \assump{F1}, \assump{F2}, \assump{K}, \assump{G}.
    Let $\Q \in A^{\textrm{fwd}}(\tau)$. Then there are $\mathscr{T}^{\textrm{fwd}}(\Q,\tau)>\tau_1(\Q,\tau)>\tau$ such that
    the trajectory $\x(t,\tau; \Q) $
      crosses transversely $\Om^0$ at $t\in\{\tau, \tau_1(\Q,\tau), \mathscr{T}^{\textrm{fwd}}(\Q,\tau)\}$.
     Hence,
     $$\mathscr{P}^{\textrm{fwd}}_+(\Q,\tau):=\x(\tau_1(\Q,\tau),\tau; \Q) \in L^{\inn},$$
     $$\mathscr{P}^{\textrm{fwd}}(\Q,\tau):=\x(\mathscr{T}^{\textrm{fwd}}(\Q,\tau),\tau; \Q) \in
   L^0,$$
     where $L^{\inn}=L^{\inn}(\de):=\{\Q\in(\Om^0\cap\tilde{V}(0))\mid \|\Q\|<\de\}$, and $\de$ is a small parameter.

    Analogously, let $\Q \in A^{\textrm{bwd}}(\tau)$. Then there are $\mathscr{T}^{\textrm{bwd}}(\Q,\tau)<\tau_{-1}(\Q,\tau)<\tau$ such that
    $\x(t,\tau; \Q) $,
    crosses transversely $\Om^0$ at $t\in\{\tau, \tau_{-1}(\Q,\tau), \mathscr{T}^{\textrm{bwd}}(\Q,\tau)\}$. Hence
     $$ \mathscr{P}^{\textrm{bwd}}_-(\Q,\tau):=\x(\tau_{-1}(\Q,\tau),\tau; \Q) \in L^{\inn},$$
     $$ \mathscr{P}^{\textrm{bwd}}(\Q,\tau):=\x(\mathscr{T}^{\textrm{bwd}}(\Q,\tau),\tau; \Q) \in
    L^0.$$

     Further the functions $\mathscr{P}^{\textrm{fwd}}_+(\Q,\tau)$, $\mathscr{P}^{\textrm{fwd}}(\Q,\tau)$, $ \mathscr{P}^{\textrm{bwd}}_-(\Q,\tau)$
     and $ \mathscr{P}^{\textrm{bwd}}(\Q,\tau)$ are $C^r$ in both the variables.
  \end{knownlemma}

Let us denote by $\ga^-(s)=\ga(s)$ and  by $\ga^+(t)=\ga(t)$ when $s \le 0 \le t$ so that
$\ga^{\pm}(\tau) \in (\Omega^\pm \cup \Omega^0)$ for all $\tau\in\R$.
We define a directed distance between the points on $\Om^0$ by arc length, in order to be able to determine their mutual positions:
this is possible since $\Om^0$ is a regular curve.
We choose as the positive orientation on $\Om^0$ the one that goes from the origin to $\ga(0)$.
So, for any $\Q\in L^0$ we define
$\ell(\Q)=\int_{\Om^0(\vec{0},\Q)}ds>0$ where $\Om^0(\vec{0},\Q)$ is the (oriented) path of $\Om^0$ connecting $\vec{0}$ with $\Q$,
and we set
\begin{equation}\label{dist}
    \dist(\Q,\P):= \ell(\P)-\ell(\Q)
\end{equation}
for $\Q,\P\in L^0$.
Notice that $\dist(\Q,\P)>0$
 means that $\Q$ lies on $\Om^0$ between $\vec{0}$ and $\P$.
Now, we introduce some further crucial notation.

\textbf{Notation.}
We denote by $\Q_s(d,\tau)$ the point in $L^0(\sqrt{\ep})$ such that
\[
\dist(\Q_s(d,\tau), \P_s(\tau))=d>0,
\]
and by $\Q_u(d,\tau)$ the point in $L^0(\sqrt{\ep})$ such that
\[
\dist(\Q_u(d,\tau), \P_u(\tau))=d>0.
\]

Let us set
\begin{equation}
\begin{split}
     J_0  & :=  \left\{ d \in \R \mid 0< d \le \ep^{\frac{1+\nu}{\und{\si}}} \right\}, \qquad \nu \ge \nu_0.
  \label{J0}
\end{split}
\end{equation}

As in \cite{CFPRimut},
 for any $d \in J_0$,
we  define the $C^r$ maps
\begin{equation}\label{Q1T1}
\begin{gathered}
\mathscr{T}_1(d,\tau):=   \mathscr{T}^{\textrm{fwd}}(\Q_s(d,\tau),\tau) \, , \qquad \mathscr{P}_1(d,\tau):=
\mathscr{P}^{\textrm{fwd}}(\Q_s(d,\tau),\tau),\\
\mathscr{T}_{-1}(d,\tau):=   \mathscr{T}^{\textrm{bwd}}(\Q_u(d,\tau),\tau) \, , \qquad \mathscr{P}_{-1}(d,\tau):=
\mathscr{P}^{\textrm{bwd}}(\Q_u(d,\tau),\tau).
\end{gathered}
\end{equation}
Sometimes we will also make use of the maps
\begin{equation}\label{Q1T1half}
\begin{gathered}
\mathscr{T}_{\frac{1}{2}}(d,\tau):=   \tau_1(\Q_s(d,\tau),\tau) \, , \qquad \mathscr{P}_{\frac{1}{2}}(d,\tau):=
\mathscr{P}^{\textrm{fwd}}_+(\Q_s(d,\tau),\tau),\\
\mathscr{T}_{-\frac{1}{2}}(d,\tau):=   \tau_{-1}(\Q_u(d,\tau),\tau) \, , \qquad \mathscr{P}_{-\frac{1}{2}}(d,\tau):=
\mathscr{P}^{\textrm{bwd}}_-(\Q_u(d,\tau),\tau).
\end{gathered}
\end{equation}

Let us    define
\begin{equation}\label{mu0}
\mu_0= \frac{1}{4} \min \left\{ \sTfwd_+ , \sTbwd_- ,  \und{\sigma}^2    \right\}.
\end{equation}
We introduce a new parameter $\mu \in ]0, \mu_0]$, which gives an upper bound for the estimate of the errors
in the evaluations of the maps defined in \eqref{Q1T1} and \eqref{Q1T1half}.

Now we collect some results borrowed from \cite{CFPRimut} and \cite{CFPrestate}, which will be used in
\S \ref{S.connection}.

 \begin{knowntheorem}\label{key}\cite[Theorem 4.2]{CFPRimut}
Assume \assump{F0}, \assump{F1}, \assump{F2}, \assump{K}, \assump{G}  and    let $\f^\pm$ and $\g$ be $C^r$ with $r>1$.
We can find $\ep_0>0$, $\delta>0$,  such that for any $0<\ep \le \ep_0$,
the functions $\TTT_{\pm 1}(d,\tau)$, $\PPP_{\pm 1}(d,\tau)$ are $C^r$ when
	$0<d \le \delta$ and $\tau \in \R$.
 Further, for any $0< \mu <\mu_0$ we get
	\begin{equation}\label{D1T1}
	\begin{split}
	 	   d^{\sfwd +\mu} & \le \dist (\PPP_1(d,\tau),\P_u(\TTT_1(d,\tau)))
 \le d^{\sfwd-\mu}, \\
	 d^{\sbwd+\mu} & \le \dist (\PPP_{-1}(d,\tau), \P_s (\TTT_{-1}(d,\tau)))
 \le d^{\sbwd-\mu},\\
 \|\PPP_{\frac{1}{2}}(d,\tau)\| & \le d^{\sfwd_+ -\mu}     , \qquad    \|\PPP_{-\frac{1}{2}}(d,\tau)\| \le d^{\sbwd_- -\mu},
	\end{split}
	\end{equation}
	\begin{equation}\label{T1T-1}
	\begin{split}
	 	 \left[\sTfwd -\mu \right] |\ln(d)|& \le  (\TTT_1(d,\tau)-\tau)    \le   \left[\sTfwd+\mu \right] |\ln(d)|, \\
	 \left[\sTbwd -\mu \right] |\ln(d)|
& \le  \tau-\TTT_{-1}(d,\tau)    \le    \left[\sTbwd+\mu \right] |\ln(d)|, \\
 \left[\sTfwd_+ -\mu \right] |\ln(d)|& \le  (\TTT_\frac{1}{2}(d,\tau)-\tau)   \le   \left[\sTfwd_++\mu \right] |\ln(d)|, \\
	 \left[\sTbwd_- - \mu \right] |\ln(d)|
& \le  \tau -\TTT_{-\frac{1}{2}}(d,\tau)     \le    \left[\sTbwd_-+\mu \right] |\ln(d)|,
	\end{split}
	\end{equation}
	and all the expressions in \eqref{D1T1} are uniform with respect to any   $\tau \in \R$ and $0< \ep \le \ep_0$.
\end{knowntheorem}

 Let us set $\tilde{V}^{\pm}(\tau)=\tilde{V}(\tau) \cap \Om^{\pm}$.

\begin{knowntheorem}\label{keymissed}\cite[Theorem 4.3]{CFPRimut}
Assume \assump{F0}, \assump{F1}, \assump{F2}, \assump{K}, \assump{G}  and    let $\f^\pm$ and $\g$ be $C^r$ with $r>1$.
We can find $\ep_0>0$, $\delta>0$,  such that for any $0<\ep \le \ep_0$,   $0<d \le \delta$,   $0< \mu <\mu_0$   and any $\tau \in \R$
   we find
  \begin{equation}\label{keymissed.es-}
   \|\x(t,\tau; \Q_s(d,\tau))-\x(t,\tau; \P_s(\tau))\| \le  d^{\sfwd_+ - \mu }
  \end{equation}
  for any $\tau \le t \le \TTT_{\frac{1}{2}}(d,\tau)$,   and
  \begin{equation}\label{keymissed.es+}
  \|\x(t,\tau; \Q_s(d,\tau))-\x(t,\TTT_1(d,\tau); \P_u(\TTT_1(d,\tau)))\| \le  d^{\sfwd_+ - \mu }
  \end{equation}
    for any $\TTT_{\frac{1}{2}}(d,\tau) \le t \le \TTT_1(d,\tau)$.
    Further   $\x(t,\tau; \Q_s(d,\tau))$ is in  $\tilde{V}^+(t)$  for any $\tau<t<
\mathscr{T}_{\frac{1}{2}}(d,\tau)$ and it is
in
$\tilde{V}^-(t)$ for any $\mathscr{T}_{\frac{1}{2}}(d,\tau)<t< \mathscr{T}_{1}(d,\tau)$.

Similarly, for any $0<\ep \le \ep_0$,   $0<d \le \delta$,   $0< \mu <\mu_0$   and any $\tau \in \R$  we find
\begin{equation}\label{keymissed.eu+}
	\|\x(t,\tau; \Q_u(d,\tau))-\x(t,\tau; \P_u(\tau))\| \le  d^{\sbwd_- - \mu}
\end{equation}
  for any $   \TTT_{-\frac{1}{2}}(d,\tau)  \le t \le \tau$,    and
\begin{equation}\label{keymissed.eu-}
	\|\x(t,\tau; \Q_u(d,\tau))-\x(t,\TTT_{-1}(d,\tau); \P_s(\TTT_{-1}(d,\tau)))\| \le  d^{\sbwd_- - \mu}
\end{equation}
    for any $\TTT_{-1}(d,\tau) \le t \le \TTT_{-\frac{1}{2}}(d,\tau)$.
    Further $\x(t,\tau; \Q_u(d,\tau))$ is in  $\tilde{V}^-(t)$  for any $
\mathscr{T}_{-\frac{1}{2}}(d,\tau)<t<\tau$
and it is in  $\tilde{V}^+(t)$  for any $\mathscr{T}_{-1}(d,\tau)<t< \mathscr{T}_{-\frac{1}{2}}(d,\tau)$.
\end{knowntheorem}

Recalling that there is $\bar{c}^*>0$ such that
   \begin{equation}\label{trajWuWs}
   \begin{split}
      \| \x(t,\tau ; \P_u(\tau)) -\ga^-(t-\tau) \| \le \bar{c}^* \ep & \qquad \qquad \textrm{for any $t \le \tau$}, \\
        \| \x(t,\tau ; \P_s(\tau)) -\ga^+(t-\tau) \| \le \bar{c}^* \ep & \qquad \qquad \textrm{for any $t \ge \tau$}.
   \end{split}
 \end{equation}
 we find the following results.

 \begin{knownproposition}\label{forPropC}\cite[Proposition 4.7]{CFPrestate} 
  Assume \assump{F0}, \assump{F1}, \assump{F2}, \assump{K}, \assump{G}, then we can find $\ep_0$ such that for
  any $0< \ep \le \ep_0$ we have the following.

  Fix $\tau \in \R$, $d \in J_0$ and $\nu \ge \nu_0$, then there is $c^*>0$
  (independent of $d$, $\nu$ and $\ep$) such that
  $$
  \| \x(t,\tau; \Q_s(d,\tau))-\ga^+(t-\tau)\| \le \frac{c^*}{2} \ep $$
  for any $\tau \le t \le \TTT_{\frac{1}{2}}(d,\tau)$, and
   $$
  \| \x(t,\tau; \Q_s(d,\tau))-\ga^-(t-\TTT_1(d,\tau))\| \le \frac{c^*}{2} \ep $$
  for any $\TTT_{\frac{1}{2}}(d,\tau) \le t \le \TTT_1(d,\tau)$.

  Further,
  $$
  \| \x(t,\tau; \Q_u(d,\tau))-\ga^-(t-\tau)\| \le \frac{c^*}{2} \ep $$
  for any $  \TTT_{-\frac{1}{2}}(d,\tau)\le t \le \tau$, and
   $$
  \| \x(t,\tau; \Q_u(d,\tau))-\ga^+(t-\TTT_{-1}(d,\tau))\| \le \frac{c^*}{2} \ep $$
  for any $\TTT_{-1}(d,\tau) \le t \le \TTT_{-\frac{1}{2}}(d,\tau)$.
\end{knownproposition}

\begin{knownlemma}\label{forgottentimes}\cite[Lemma 4.9]{CFPrestate} 
  Let  $d \in J_0$    and $c^*>0$ be as in Proposition \ref{forPropC} and denote
  \begin{equation}\label{TaTb}
  	T_a:=\frac{1}{\la_u^-}|\ln \ep|,\qquad T_b:=\frac{1}{|\la_s^+|}|\ln\ep|.
  \end{equation}
  Then
  \begin{equation}\label{T.half}
  \begin{split}
       & \tau +T_b  \le \TTT_{\frac{1}{2}}(d,\tau) \le \TTT_1(d,\tau)-T_a  \qquad  \textrm{and}   \\
       &  \|x(t,\tau ; \Q_s(d,\tau))\| \le  \frac{3}{4}  c^* \ep  \qquad  \textrm{for any $t \in [\tau +T_b ,\TTT_1(d,\tau)-T_a]$.}
  \end{split}
  \end{equation}
  Further,
   \begin{equation}\label{T.halfbis}
  \begin{split}
       & \TTT_{-1}(d,\tau) +T_b  \le \TTT_{-\frac{1}{2}}(d,\tau) \le \tau-T_a  \qquad  \textrm{and}   \\
       &  \|x(t,\tau ; \Q_u(d,\tau))\| \le   \frac{3}{4}  c^* \ep  \qquad  \textrm{for any $t \in [\TTT_{-1}(d,\tau) +T_b ,\tau-T_a]$.}
  \end{split}
  \end{equation}
\end{knownlemma}

\begin{knownlemma}\label{forgotten.origin}\cite[Lemma 4.10]{CFPrestate} 
  Let  $d \in J_0$ and $c^*>0$ be as in Proposition \ref{forPropC}.
  If $t \in [\tau +T_b ,\TTT_1(d,\tau)-T_a]$ we get
  \begin{eqnarray}
   &  \|x(t,\tau ; \Q_s(d,\tau))\| \le  c^* \ep \label{est.origin},  \\
  &  \|x(t,\tau ; \Q_s(d,\tau))- \ga(t-\tau)\| \le  c^* \ep  \label{est.origin-}, \\
  &  \|x(t,\tau ; \Q_s(d,\tau))- \ga(t-\TTT_1(d,\tau))\| \le  c^* \ep \label{est.origin+}.
  \end{eqnarray}
\end{knownlemma}

The following is a consequence of Theorem \ref{key}.

\begin{knownlemma}\label{key1}\cite[Lemma 4.11]{CFPrestate} 
Assume \assump{F0}, \assump{F1}, \assump{F2}, \assump{K}, \assump{G}.
	Fix $\nu \ge \nu_0$ and let $d\in J_0$;
	fix $\tau \in \R$.
If there is $c_1>0$ such that
$\MM(\TTT_{\pm 1}(d,\tau))\le - 3c_1$ then
\begin{equation*}
	\begin{split}
		d_1 (d,\tau) & :=\dist (\PPP_1(d,\tau),\P_s(\TTT_1(d,\tau))) >  cc_1 \ep>0, \\
		d_{-1} (d,\tau) & :=\dist (\PPP_{-1}(d,\tau),\P_u(\TTT_{-1}(d,\tau))) <-cc_1 \ep <0 ,
	\end{split}
\end{equation*}
while if  $\MM(\TTT_{\pm 1}(d,\tau))\ge
3c_1$ then
$\pm d_{\pm1}(d,\tau)<- cc_1\ep<0$.
\end{knownlemma}

\section{A connection argument}    \label{S.connection}

Let $\ee^{+} \in \EE^{+}$ and $\tau\in[b_0,b_1]$;
following \cite{CFPrestate} we say that $\x(t,\tau; \xxi)$ satisfies the property  ``chaotic in the future",
 $\bs{C_{\ee^+}^+}$ if
there   exist
  $\alpha_j(\ep)= \al_j^{\ee^+}(\ep,\tau,\TT^+)$ such that the following conditions are satisfied
 \begin{description}
  \item[$\bs{C_{\ee^+}^+}$] $\x(t,\tau; \xxi) \in \tilde{V}(t)$ for any $t \ge \tau$,
 and	if $\ee_j=1$, then
	\begin{equation}\label{ej+=1w}
	\begin{split}
	&   \|\x(t,\tau; \xxi )-\ga(t-T_{2j}-\al_j(\ep)) \| \le c^*\ep \quad \textrm{when $t \in [T_{2j-1}, T_{2j+1}]$},
	\end{split}
	\end{equation}
	while if $\ee_j=0$, we have
	\begin{equation}\label{ej+=0w}
	\|\x(t,\tau; \xxi) \| \le c^*\ep  \quad \textrm{when $t \in [T_{2j-1}, T_{2j+1} ] $}
	\end{equation}
	for any $j \in \Z^+$. Further
	\begin{equation}\label{ej+=tau}
	\|\x(t,\tau; \xxi)-\ga(t-\tau) \| \le c^*\ep  \quad \textrm{when $t \in [\tau, T_1] $}.
	\end{equation}
\end{description}
Analogously let $\ee^{-} \in \EE^{-}$ and $\tau\in[b_0,b_1]$; we say that $\x(t,\tau; \xxi)$ satisfies the property
 ``chaotic in the past", $\bs{C_{\ee^-}^-}$ if
there exist $\alpha_j(\ep)= \al_j^{\ee^-}(\ep,\tau,\TT^-)$ such that
 the following holds:
 \begin{description}
  \item[$\bs{C_{\ee^-}^-}$]  $\x(t,\tau; \xxi) \in \tilde{V}(t)$ for any $t \le \tau$, and
	if $\ee_j=1$, then
	\begin{equation}\label{ej-=1w}
	\begin{split}
	&   \|\x(t,\tau; \xxi )-\ga(t-T_{2j}-\al_j(\ep)) \| \le c^*\ep \quad \textrm{when $t \in [T_{2j-1}, T_{2j+1} ]$},
	\end{split}
	\end{equation}
	while if $\ee_j=0$, we have
	\begin{equation}\label{ej-=0w}
	\|\x(t,\tau; \xxi) \| \le c^*\ep  \quad \textrm{when $t \in [T_{2j-1}, T_{2j+1} ] $}
	\end{equation}
	for any $j \in \Z^-$. Further
	\begin{equation}\label{ej-=tau}
	\|\x(t,\tau; \xxi)-\ga(t-\tau) \| \le c^*\ep  \quad \textrm{when $t \in [T_{-1}, \tau ] $}.
	\end{equation}
\end{description}

From Theorems 3.4 and 3.5 in \cite{CFPrestate}
 we see that for any $\tau \in [b_0, b_1]$ there are non-empty compact and connected sets
$X^{\pm}(\tau,\TT^{\pm},\ee^{\pm})$ such that if $\xxi^{\pm} \in X^{\pm}(\tau,\TT^{\pm},\ee^{\pm})$ then
$\x(t,\tau; \xxi^+)$ has property $\bs{C_{\ee^+}^+}$ and $\x(t,\tau; \xxi^-)$ has property $\bs{C_{\ee^-}^-}$.
So we can define
\begin{equation}\label{defX+la}
 \begin{split}
    \aleph= & \left\{ (\xxi,\tau) \mid  \xxi \in L_0(\sqrt{\ep})
     \, , \;\tau \in  [b_0, b_1]  \right\} \,; \\
 \chi^{\Lambda,+}(\ee^+)=     &   \left\{ (\xxi,\tau) \in \aleph \middle|  \begin{array}{l}
                                      \x(t,\tau ; \xxi) \; \textrm{ has property $\bs{C^+_{\ee^+}}$ and } \\
                                      |\alpha_j| \le  \adown_j-\aup_j \le \l1  \; \textrm{ for any $j \in \Z^+$}
                                    \end{array}  \right\},\\
\chi^{\Lambda,-}(\ee^-)= & \left\{ (\xxi,\tau) \in\aleph \middle|  \begin{array}{l}
                                      \x(t,\tau ; \xxi) \; \textrm{ has property $\bs{C^-_{\ee^-}}$ and } \\
                                     |\alpha_j| \le \adown_j-\aup_j \le  \l1  \; \textrm{ for any $j \in \Z^-$}
                                    \end{array}  \right\} .
 \end{split}
\end{equation}

The next result is formulated under the assumptions of Theorems
\ref{main.periodic1},
\ref{main.weak1} and gives $|\alpha_j|$ bounded or infinitesimal.

\begin{theorem}\label{T.connectionLa}
Assume  that  $\f^{\pm}$ and $\g$ are $C^r$, $r>1$ and that \assump{F0}, \assump{F1}, \assump{F2}, \assump{K} and \assump{G} hold true; assume further
\assump{P1},
and fix $\nu \ge \nu_0$ for $\nu_0$ as in \eqref{defK0-new} and   $\tau \in [b_0, b_1]$; let $\ee^+ \in \EE^+$ and $\ee^- \in \EE^-$.
Then we can choose $\ep_0$ small enough so that for any $0< \ep \le \ep_0$
and any increasing sequence $\mathcal{T}^+=(T_j)$ satisfying
 \eqref{TandKnu} for $j \in \Z^+$ and
\begin{equation}\label{TandKnu_tau+}
	 T_1-b_1 >   \l1 + K_0(1+\nu)   |\ln(\ep)| 
\end{equation}
 the set $ \chi^{\Lambda,+}(\ee^+)$  contains a closed connected set $\tilde{\chi}^{\Lambda,+}(\ee^+)$
which intersects the lines $\tau=b_0$ and $\tau=b_1$.

Analogously,     for any $0< \ep \le \ep_0$
and any increasing sequence $\mathcal{T}^-=(T_j)$  satisfying
 \eqref{TandKnu} for $j \le -2$ and
\begin{equation}\label{TandKnu_tau-}
	 b_0-T_{-1} >   \l1 + K_0(1+\nu)   |\ln(\ep)| 
\end{equation}
 the set $ \chi^{\Lambda,-}(\ee^-)$  contains a closed connected set $\tilde{\chi}^{\Lambda,-}(\ee^-)$
which intersects the lines $\tau=b_0$ and $\tau=b_1$.

  Moreover for any $j \in \mathbb{Z}\setminus \{0\}$ we have the following estimates
    \begin{description}
    \item[a)] if \eqref{nondeg} holds and $r \ge 2$   there is $c_{\alpha}> 0$  such that $|\al_{j}^{\ee^+}(\ep)| \le c_{\alpha}\ep$;
    \item[b)] if \eqref{isolated} holds and $r>1$ then
    $$|\al_{j}^{\ee^+}(\ep)| \le \omega_{\alpha} (\ep) ,$$
    where $\omega_{\alpha}(\cdot)$ is an increasing continuous function such that $\omega_{\alpha}(0)=0$;
    \item[c)] if \eqref{minimal} holds and $r>1$ then $|\al_{j}^{\ee^+}(\ep)| \le \l1$.
  \end{description}
  The constants $\ep_0$, $c^*$, $c_{\alpha}$  and the function  $\omega_{\alpha}$ are  independent of
   $\ep$,
    $\nu$, $\tau$, $\TT^{\pm}$,
  $\ee^{\pm}$.
\end{theorem}

Now we consider the case in  which $(B_j)$ may become unbounded, as in Theorem \ref{veryweak}, so the distance
 from consecutive zeros of $\MM$ may become unbounded as well.

Let us set
\begin{equation}\label{defX+}
\begin{split}
\chi^+(\ee^+)=& \left\{ (\xxi,\tau) \in \aleph \middle|  \begin{array}{l}
                                      \x(t,\tau ; \xxi) \; \textrm{ has property $\bs{C^+_{\ee^+}}$ and } \\
                                      |\alpha_j| \le  B_{2j}  \; \textrm{ for any $j \in \Z^+$}
                                    \end{array}  \right\},\\
\chi^-(\ee^-)= & \left\{ (\xxi,\tau) \in \aleph  \middle|  \begin{array}{l}
                                      \x(t,\tau ; \xxi) \; \textrm{ has property $\bs{C^-_{\ee^-}}$ and } \\
                                     |\alpha_j| \le   B_{2j}  \; \textrm{ for any $j \in \Z^-$}
                                    \end{array}  \right\} .
    \end{split}
\end{equation}
\begin{theorem}\label{T.connection}
   Assume that the hypotheses of Theorem \ref{T.connectionLa} hold  but
replace \eqref{TandKnu}  by \eqref{TandKnunew}  and   \eqref{TandKnu_tau+} by
\begin{equation}\label{TandKnunew_tau+}
	T_1-b_1 >  \max\{  B_{1}; B_{0}   \} + K_0(1+\nu)   |\ln(\ep)| , 
\end{equation}
and \eqref{TandKnu_tau-} by
	\begin{equation}\label{TandKnunew_tau-}
		b_0-T_{-1} >  \max\{  B_{0}; B_{-1}   \} + K_0(1+\nu)   |\ln(\ep)| . 
\end{equation}
Then the set $\chi^+(\ee^+)$  contains a closed connected set $\tilde{\chi}^+(\ee^+)$
which intersects the lines $\tau=b_0$ and $\tau=b_1$; analogously
 the set $\chi^-(\ee^-)$  contains a closed connected set $\tilde{\chi}^-(\ee^-)$
which intersects the lines $\tau=b_0$ and $\tau=b_1$. However
 we just have
     $$|\al_{j}^{\ee^+}(\ep)| \le B_{2j}, \qquad |\al_{j}^{\ee^-}(\ep)| \le B_{2j}.$$
\end{theorem}

\subsection{Proof of Theorem \ref{T.connection}: the $|\alpha_j|\le B_{2j}$ case}\label{S.theorem2.1}

In this section we prove Theorem \ref{T.connection} using a topological argument together with  the selection scheme developed in
\cite{CFPrestate}: this   will be a step necessary to prove Theorem \ref{veryweak}.

First we need to introduce the following sets
\begin{equation}\label{defS+}
\begin{split}
A_0= & \left\{ (d,\tau) \mid  d \in  J_0 =
[0, \ep^{(1+\nu)/\underline{\sigma}}  ] \, , \;\tau \in [b_0, b_1] \right\} \,,
\\
S^+=& \left\{ (d,\tau) \in A_0 \middle|  \begin{array}{l}
                                      \x(t,\tau ; \Q_s(d,\tau)) \; \textrm{ has property $\bs{C^+_{\ee^+}}$ and } \\
                                      |\alpha_j| \le  B_{2j}  \; \textrm{ for any $j \in \Z^+$}
                                    \end{array}  \right\},\\
S^-= & \left\{ (d,\tau) \in A_0 \middle|  \begin{array}{l}
                                      \x(t,\tau ; \Q_u(d,\tau)) \; \textrm{ has property $\bs{C^-_{\ee^-}}$ and } \\
                                     |\alpha_j| \le   B_{2j}  \; \textrm{ for any $j \in \Z^-$}
                                    \end{array}  \right\} .
    \end{split}
\end{equation}

Our argument relies on the following topological result borrowed from \cite{PZ}.

 \begin{lemma}\label{top}\cite[Lemma 4]{PZ}
Let  $\mathcal{R}=[\alpha_1,\beta_1] \times [\alpha_2,\beta_2]$  be a   full rectangle. Let
$E \subset \mathcal{R}$ be a closed set such that for any  path $\vec{\psi} : [0,1] \to \mathcal{R}$,
$\vec{\psi}(a)=(x(a),y(a))$ with $x(0)=\alpha_1$ and $x(1)=\beta_1$ there is
$\bar{a} \in ]0,1[$ (depending on $\vec{\psi}$) such that  $\vec{\psi}(\bar{a}) \in E$. Then $E$ contains a
closed connected set $\tilde{E}$  which intersects both $y=\alpha_2$ and $y=\beta_2$.
\end{lemma}

The core of our argument is made up by the following two propositions.
The proof  of Proposition~\ref{Sclosed} is postponed by a few lines and the one of  Proposition~\ref{Sintersection} is postponed by a few
pages.
\begin{proposition}\label{Sclosed}
Let $\ee^+ \in \EE^+$ be fixed; then the set $S^+ \subset A_0$ is closed.
\end{proposition}
\begin{proposition}\label{Sintersection}
Let $\ee^+ \in \EE^+$ be fixed; then for any continuous path $\vec{\psi} : [0,1] \to A_0$, $\vec{\psi}(a)=(d(a),\tau(a))$ such that
$d(0)=0$, $d(1)=\ep^{(1+\nu)/\underline{\sigma}}$ there is
$\bar{a} \in ]0,1[$ (depending on $\vec{\psi}$) such that  $\vec{\psi}(\bar{a}) \in S^+$.
\end{proposition}
Once proved Propositions \ref{Sclosed} and \ref{Sintersection} we
can apply Lemma \ref{top}
to
conclude
the following.
\begin{proposition}\label{topological}
Let the hypotheses of Theorem \ref{T.connection} be satisfied.
Then the set $S^+$  contains a closed connected set $\tilde{S}^+$
which intersects the lines $\tau=b_0$ and $\tau=b_1$, while the set $S^-$  contains a closed connected set $\tilde{S}^-$
which intersects the lines $\tau=b_0$ and $\tau=b_1$.
\end{proposition}
\begin{proof}
  The part concerning $S^+$ and $\tilde{S}^+$, follows
 from   Propositions \ref{Sclosed} and \ref{Sintersection}, simply applying Lemma \ref{top}  choosing  $\mathcal{R}= A_0$, $E= S^+$.

 The part concerning  $S^-$ and $\tilde{S}^-$ then follows using an inversion of time argument, see e.g.\ \cite[\S 5.3]{CFPrestate}.
\end{proof}

\begin{proof}[\textbf{Proof of Proposition \ref{Sclosed}}]
In this proof, we consider $\ee^+$ and $\ep$ fixed so we leave the dependence on these variables unsaid.
Assume by contradiction that $S^+$ is not closed. Then there is a sequence $(d_n, \tau_n) \in S^+$ such that
$(d_n, \tau_n)\to  (\bar{d}, \bar{\tau}) \not\in S^+$. Let us denote by  $\alpha_j^n=\alpha_j(d_n, \tau_n)$.
 By  \eqref{defS+}, $|\alpha_j^n | \le B_{2j}$ for any $j$ and $n$.

Let  $\ell \in \N$
be such that the trajectory $\x(t,\bar{\tau}; \Q_s(\bar{d},\bar{\tau}))$
satisfies $\bs{C^+_{\ee^+}}$  with   $\alpha_j = \bar{\alpha}_j$,
 for suitable
   $\alpha_j = \bar{\alpha}_j$
 for any $j \le \ell-1$, but $\bs{C^+_{\ee^+}}$ does not hold
 for $j=\ell$   for any $\alpha_{\ell}= \bar{\alpha}$ whenever $|\bar{\alpha}| \le B_{2 \ell}$.

We split the argument in the case $\ee_\ell=0$ and $\ee_\ell=1$.

Assume first $\ee_\ell=0$.
Then,  for any $|\bar{\alpha}| \le B_{2\ell}$ there is $\bar{t}=\bar{t}(\bar{\alpha}) \in [T_{2\ell-1},T_{2\ell+1}]$  such that
$$\|\x(\bar{t},\bar{\tau}; \Q_s(\bar{d},\bar{\tau})) \| >  c^*\ep ;$$
 in fact we may choose    $\bar{t}\in]T_{2\ell-1},T_{2\ell+1}[$.
So we can find $\delta>0$ and a suitable $\mu>0$ such that
$$\|\x(t,\bar{\tau}; \Q_s(\bar{d},\bar{\tau})) \| >  (c^*+2\mu) \ep $$
for any $t \in [\bar{t}-\delta, \bar{t}+\delta]$.
Then,  recalling that $\Q_s(\cdot,\cdot)$ is continuous  and using continuous dependence of the solutions of \eqref{eq-disc} on data and
parameters, see Remark \ref{data.smooth},
we can find $N>0$ large enough so that
$$\|\x(t,\tau_{n}; \Q_s(d_{n},\tau_{n})) \| >  (c^*+\mu) \ep $$
for any $t \in [\bar{t}-\delta, \bar{t}+\delta]$ and any $n \ge N$.
 But this is a contradiction since $\x(t,\tau_{N}; \Q_s(d_{N},\tau_{N}))$ satisfies   $\bs{C^+_{\ee^+}}$.

 Now assume $\ee_\ell=1$, then for any $|\bar{\alpha}| \le B_{2\ell}$ there is $\bar{t}=\bar{t}(\bar{\alpha}) \in [T_{2\ell-1},T_{2\ell+1}]$
 such that
$$\|\x(\bar{t},\bar{\tau}; \Q_s(\bar{d},\bar{\tau}))-\ga(\bar{t}- T_{2\ell}-\bar{\al}) \|> c^*\ep .$$
 Again we can assume $\bar{t} \in ]T_{2\ell-1},T_{2\ell+1}[$,
so that we can find $\delta>0$ and a suitable $\mu>0$ such that
$$\|\x(t,\bar{\tau}; \Q_s(\bar{d},\bar{\tau}))-\ga(t- T_{2\ell}-\bar{\al}) \|> (c^*+2 \mu)\ep $$
 for any $t \in [\bar{t}-\delta, \bar{t}+\delta]$.
 Then, using again  the continuity of $\Q_s(\cdot,\cdot)$ and the  continuous dependence of the solutions of \eqref{eq-disc} on data and
 parameters
 (see Remark \ref{data.smooth}),
we can find $N>0$ large enough so that
$$\|\x(t,\tau_{n}; \Q_s(d_{n},\tau_{n})) -\ga(t- T_{2\ell}-\bar{\al}) \| >  (c^*+\mu) \ep $$
for any $t \in [\bar{t}-\delta , \bar{t}+\delta]$ and any $n \ge N$.
 In particular, if we choose $\bar{\alpha}= \alpha^N_{\ell}$, we find $\bar{t}^N= \bar{t}(\alpha^N_{\ell})$ such that
$$\|\x(t,\tau_{N}; \Q_s(d_{N},\tau_{N})) -\ga(t- T_{2\ell}-\al^{N}_\ell) \| >  (c^*+\mu) \ep $$
for any $t \in [\bar{t}^N-\delta , \bar{t}^N+\delta]$.
But this contradicts the fact that the trajectory $\x(t,\tau_{N}; \Q_s(d_{N},\tau_{N}))$ satisfies   $\bs{C^+_{\ee^+}}$.

Now assume that there is $\bar{T} \in [\tau, T_1]$ such that $$\|\x( \bar{T},\bar{\tau}; \Q_s(\bar{d},\bar{\tau}))-
\ga(\bar{T}-\bar{\tau}) \| >  c^*\ep
.$$
Arguing as in the $\ee_\ell=1$ case, we
 find a contradiction.
\end{proof}

The Proof of Proposition \ref{Sintersection} is  rather lengthy and requires several lemmas.

We start from some notation borrowed from \cite{CFPrestate}.
We split $\EE^+= \{0,1\} ^{\Z^+}$ in two subsets $\hat{\EE}^+$ and $\EE^+_0$:
 \begin{equation}\label{efuture}
 \begin{array}{c}
   j^+_*(\ee^+)=\sup\{ j \mid \ee^+_j=1  \},\\
\hat{\EE}^+= \{ \ee^+=(\ee^+_j) \in \EE^+ \mid j^+_*(\ee^+)= +\infty \} ,\\
\EE^+_0= \{ \ee^+=(\ee^+_j) \in \EE^+ \mid  j^+_*(\ee^+)< +\infty \},\\
 j_\SS^+(\ee^+)=\begin{cases}\#\{j\mid 1\leq j\leq j^+_*(\ee^+),\,\ee_j^+=1\},&\text{if }j^+_*(\ee^+)<\infty,\\ \infty,& \text{if }
j^+_*(\ee^+)=\infty,\end{cases}
\end{array}
\end{equation}
 where $\#$ denotes the number of elements. So, $j^+_*(\ee^+)$ is ``the index of the last $1$ of $\ee^+$'' and $j_\SS^+(\ee^+)$ is
``the number of $1$s in $\ee^+$''.

Let us fix  $\TT^+=(T_j)$, $j \in \Z^+$ satisfying \eqref{TandKnunew},  \eqref{TandKnunew_tau+}
and $\ee^+ \in \EE^+$; adapting
\cite{BF11} we introduce a new sequence
$\SS=(\SS_j)$,  $j=0,1,\dots,j_\SS^+(\ee^+)$, which is a subsequence of $\TT^+$ depending also on $\ee^+$, and we define $(\Delta_j)$,
 $j=1,2,\dots,j_\SS^+(\ee^+)$ as
follows:
\begin{equation}\label{defSS}
\begin{split}
  \SS_0 &= \tau \,, \qquad \qquad  \SS_j= \min \{T_{2k} > \SS_{j-1} \mid \ee_k=1 \}  , \qquad 1 \le j \le j^+_{\SS}(\ee^+),\\
  \Delta_j &= \SS_j-\SS_{j-1}  \, , \quad \textrm{ if  $1 \le j \le j^+_{\SS}(\ee^+)$} \,.
\end{split}
\end{equation}
Clearly,  $\SS_j$ has finitely many values if $\ee^+ \in \EE^+_0$, while if $\ee^+ \in \hat{\EE}^+$,   then $\SS_j$ and $\Delta_j$ are
defined for any $j \in \N$ and in $\Z^+$, respectively.\\
We denote by $k_j$,   the  subsequence such that $\SS_j=T_{2k_j}$, $1 \le j  \le j^+_{\SS}(\ee^+)$, and we set
$\beta_{2k_0}=b_0$, $ \beta'_{2k_0}=b_1$, so that
\begin{equation}\label{defik}
\begin{split}
 & \beta_{2k_j} < \SS_j < \beta'_{2k_j}, \qquad 0 \le j \le j^+_{\SS}(\ee^+)
\end{split}
\end{equation}
where  $\beta_j$  and $\beta'_j$ are the ones defined in \assump{P1} and \eqref{defkj}.
Notice that there is a subsequence $n_k$ such that $\beta_{k}=b_{n_k}$, so
 $\beta_{2k_j} = b_{n_{2k_j}}$ and  $\beta'_{2k_j} = b_{n_{2k_j}+1}$.
Further, setting $B_{2k_0}=B_0=b_1-b_0$, by \eqref{TandKnunew} for all $j \geq 1$  we find the estimate
\begin{equation}\label{defDej}
\Delta_j \ge 2 K_0(1+\nu) |\ln(\ep)| + B_{2k_j}+ B_{2k_{j-1}} .
\end{equation}

In fact, we will rely mainly on the sequences $(\SS_j)$ and $(\Delta_j)$.

We start our argument with the more involved case of $\ee^+ \in \hat{\EE}^+$,  then the case where $\ee^+ \in  \EE^+_0$ will follow more
easily.

In the whole proof  $\vec{\psi}:[0,1] \to A_0$, $\vec{\psi}(a)=(d(a),\tau(a))$ is a generic path such that
$d(0)=0$ and $d(1)= \ep^{(1+\nu)/ \underline{\sigma}}$. Since $\vec{\psi}$, $\TT^+$ and $\ee^+ \in \EE^+$ are fixed in this proof, we leave
these dependencies unsaid.

 In the proof of \cite[Theorem 3.4]{CFPrestate} we have constructed a sequence of nested intervals $J_n \subset J_{n-1} \subset J_0$
 characterized by the following property:  if
$d \in J_n$ then the trajectory $\x(t,\tau ; \Q_s(d,\tau))$ has property $\bs{C^+_{\ee^+}}$ when $t \in [\tau, \TTT_n(d,\tau)]$  and it will cross
(transversely)
$L^0$ at $t= \TTT_n(d,\tau)$ after performing $n$ loops.

More precisely, following \cite{CFPrestate}
for any  $i=1,\dots,n$ we   define the $C^r$ functions
$$
  \TTT_i(\cdot ,\tau): J_i \to  [\beta_{2k_i}, \beta_{2k_{i}}'] \, , \qquad   d_i(\cdot ,\tau): J_i \to  \R,
$$
\begin{equation}\label{Tn}
\begin{split}
d_{i}(d,\tau):= & d_{1}(D,A), \quad \textrm{where $D=d_{i-1}(d,\tau)$, and $A=\TTT_{i-1}(d,\tau)$},\\
\TTT_{i}(d,\tau):= & \TTT_{1}(D,A), \quad \textrm{where $D=d_{i-1}(d,\tau)$, and $A=\TTT_{i-1}(d,\tau)$},\\
\TTT_{i-\frac{1}{2}}(d,\tau):= & \TTT_{\frac{1}{2}}(D,A), \quad \textrm{where $D=d_{i-1}(d,\tau)$, and $A=\TTT_{i-1}(d,\tau)$}
\end{split}
\end{equation}
with $d_0(d,\tau)=d$, $\TTT_0(d,\tau)=\tau$
such that, if
$d \in J_n$, the trajectory $\x (t,\tau; \vec{Q^s}(d,\tau))$ performs $n$ loops close to $\bs{\Gamma}$ when $t \in [\tau,  \TTT_n(d,\tau)]$,
and intersects (transversely) $L^0$ exactly at $t=\tau$ and at $t= \TTT_i(d ,\tau)$ and $L^{\inn}$ at $\TTT_{i-\frac{1}{2}}(d,\tau)$ for $i=1, \ldots, n$.

To prove Proposition \ref{Sintersection} we just have to repeat the argument of \cite[Theorem~3.4]{CFPrestate} but this time working on the path $\vec{\psi}(a)$
instead of on the horizontal segment $J_0\times \{\tau \}$.
The argument relies mainly on Theorem \ref{key} and on its uniformity with respect to $\tau$.

So, our aim is to build a sequence of nested compact intervals $\mathcal{J}_{n+1} \subset \mathcal{J}_{n} \subset
\mathcal{J}_1 \subset [0,1]$
such that, if $a \in \mathcal{J}_{n}$ then  $\x(t,\tau(a); \Q_s(d(a),\tau(a)))$ has property $\bs{C^+_{\ee^+}}$
 whenever $t \in [\tau, \TTT_n(d(a),\tau(a))] \supset [\tau, \beta_{2k_n}]$.

Let  $D_A^1=\textrm{exp}\left(-\frac{5(\Delta_1 +B_{2k_1})}{2 \sTfwd}\right)$ and observe that $D_A^1 < \ep^{(1+\nu)/\underline{\sigma}}$, cf.~(5.5) 
 in
\cite{CFPrestate}. Then set
$I_1 = \left[ D_A^1   ,  \ep^{(1+\nu)/\underline{\sigma}}\right]$ and
$$
\mathcal{I}_1:= d^{-1}(I_1)= \left\{ a \in [0,1] \mid d(a) \in I_1  \right\}.$$
Notice that $\mathcal{I}_1$ is closed,  since $I_1$ is closed and $d(\cdot)$ is continuous; further $\mathcal{I}_1 \ne \emptyset$
since $d(\cdot):[0,1] \to  [0,\ep^{(1+\nu)/\underline{\sigma}}]$ is surjective by construction.

Now we use the stretching property of the function  $\TTT_1(\cdot, \tau)$  as in Theorem~\ref{key} (which is in fact
based on the fact that $\TTT_1(d, \tau)$ grows as  $|\ln d|$   as $d \to 0$, see \cite{CFPrestate}).

\begin{lemma}\label{key2bismod}
  There are $a',a'' \in   \mathcal{I}_1$ such that
 \begin{equation}\label{goalkey2bis}
   \begin{split}
    \TTT_1(\vec{\psi}(a')) =     \beta_{2k_1},  \qquad  \TTT_1(\vec{\psi}(a'')) =     \beta'_{2k_1}.
   \end{split}
 \end{equation}
 Hence the image of the function $\TTT_1(\vec{\psi}(\cdot)) :  \mathcal{I}_1 \to \R$
 contains the closed interval $[\beta_{2k_1} ,\beta'_{2k_1}]$.
\end{lemma}
\begin{proof}
Following the notation of \cite[Lemma 5.1]{CFPrestate}, we denote by  $D^1_B= \ep^{(1+\nu)/\underline{\sigma}}$. Let   $A'<A''$  be such that
$d(A')=D_A^1$, $d(A'')=D_B^1$; we can assume w.l.o.g. that $[A',A''] \subset \mathcal{I}_1$. As in the proof of \cite[Lemma 5.1]{CFPrestate},
using Theorem~\ref{key} we see that
$$\TTT_1 (\vec{\psi}(A'))> \SS_1+ B_{2k_1}, \qquad  \TTT_1 (\vec{\psi}(A''))< \SS_1- B_{2k_1}.$$
Since
$[\beta_{2k_1},\beta_{2k_1}']\subset[\SS_1-B_{2k_1},\SS_1+B_{2k_1}]$, from the continuity of $\TTT_1 (\vec{\psi}(\cdot))$   we see that there
are $a', a'' \in \mathcal{I}_1$ such that
\eqref{goalkey2bis} holds.
\end{proof}
Then we set
$$
\mathcal{\hat{I}}_1:=  \left\{ a \in \mathcal{I}_1 \mid \beta_{2k_1} \le \TTT_1(\vec{\psi}(a))  \le \beta'_{2k_1}   \right\}.$$
Observe that $\mathcal{\hat{I}}_1$ is closed but possibly disconnected, so we choose $\mathcal{\check{I}}_1\subset \mathcal{\hat{I}}_1$ so
that
$\mathcal{\check{I}}_1$ is a compact connected interval with the following property
 \begin{equation}\label{prop1bis}
     \begin{split}
 &   \TTT_1(\cdot,\tau):  \mathcal{\check{I}}_1  \to  [\beta_{2k_1}, \beta'_{2k_1}] \qquad \textrm{is surjective}.
                                 \end{split}
   \end{equation}

\begin{remark}\label{closest}
  For definiteness  we choose $\mathcal{\check{I}}_1$ to be   the ``interval closest to $0$ satisfying property
   \eqref{prop1bis}''.  Namely notice that by construction there are $A<B$, $A,B \in \mathcal{\hat{I}}_1$  such that $\TTT_1(\vec{\psi}(A))=\beta'_{2k_1}$ and
   $\TTT_1(\vec{\psi}(B))=\beta_{2k_1}$, then we set
$$B'_1= \min \{a \in  \mathcal{\hat{I}}_1 \mid  \TTT_1(\vec{\psi}(a))=\beta_{2k_1}   \}, $$
$$  A'_1= \max \{ a\in \mathcal{\hat{I}}_1 \mid a< B'_1, \,
\TTT_1(\vec{\psi}(a))=\beta'_{2k_1}  \}$$  and we define $\mathcal{\check{I}}_1=[A_1', B_1']$
(so it is the   smallest   interval with this property).
   \end{remark}

Then from  property \eqref{prop1bis},  \assump{P1} and   Lemma \ref{key1}   we see that there are $\mathcal{A}^-, \mathcal{A}^+ \in \mathcal{\check{I}}_1$ such that
$d_1(\vec{\psi}(\mathcal{A}^-))=-\ep^{(1+\nu)/\underline{\sigma}}$ and
$d_1(\vec{\psi}(\mathcal{A}^+))=\ep^{(1+\nu)/\underline{\sigma}}$. Assume to fix the ideas $\mathcal{A}^-<\mathcal{A}^+$, then we set
\begin{equation}\label{J1ebis}
  \mathcal{J}_1:=[\mathcal{A}^-,\mathcal{A}^+] \subset \mathcal{I}_1
\end{equation}
so we have the following lemma.
\begin{lemma}
The function $ \TTT_1(\vec{\psi}(\cdot)): \mathcal{J}_1 \to  [\beta_{2k_1}   , \beta'_{2k_1}]$
is well defined and $C^r$, while the function
$$ d_1(\vec{\psi}(\cdot)): \mathcal{J}_1 \to  \R$$
is  $C^r$ and its image contains $[-\ep^{(1+\nu)/ \underline{\sigma}} , \ep^{(1+\nu)/ \underline{\sigma}}]$.
\end{lemma}

Now we keep  in mind the definition of $\TTT_n$ and $d_n$ given in \eqref{Tn}, and  we proceed by induction
to prove the following result.
\begin{proposition}\label{inductionbis}
Let $\ee^+ \in \hat{\EE}^+$; then
there is an infinite sequence of nested compact intervals  $\mathcal{J}_{n} \subset\mathcal{J}_{n-1} \subset \mathcal{J}_1$
  such that the function
$$ \TTT_{n}(\vec{\psi}(\cdot)): \mathcal{J}_{n} \to  [\beta_{2k_n}, \beta'_{2k_n}]$$
is well defined and $C^r$, while  the image of the function
$$ d_{n}(\vec{\psi}(\cdot)): \mathcal{J}_{n} \to  \R$$
contains $[-\ep^{(1+\nu)/ \underline{\sigma}} , \ep^{(1+\nu)/ \underline{\sigma}}]$.
  \end{proposition}
\begin{proof}
We proceed by induction: we have already shown the $n=1$ step, so we assume the existence of $\mathcal{J}_{n-1}$.

Following the argument in the proof of Proposition 5.7  in \cite{CFPrestate}, we define
\begin{equation*}
\begin{split}
  D^n_\# &= \textrm{exp}\left(\frac{-5 (\Delta_n+B_{2k_n}+B_{2k_{(n-1)}}) }{2\sTfwd}\right) \le \ep^{(1+\nu)/ \underline{\sigma}},   \\
      \mathcal{I}_n &=
  \left\{ a \in \mathcal{J}_{n-1} \mid  D^n_\# \le d_{n}(\vec{\psi}(a))    \le \ep^{(1+\nu)/ \underline{\sigma}} \right\}.
  \end{split}
\end{equation*}
  Let   $A'_n<A''_n$  be such that
 $ d_n(\vec{\psi}(A'_n))=D_\#^n$, $d_n(\vec{\psi}(A''_n))=\ep^{(1+\nu)/ \underline{\sigma}}$; we can assume w.l.o.g. that $[A'_n,A''_n]
 \subset
\mathcal{I}_n$.

Then, from the proof of \cite[Lemma 5.8]{CFPrestate}  (and the uniformity of the estimates in
Theorem~\ref{key} with respect to $\tau \in [b_0,b_1]$), we see that
$$ \TTT_n(\vec{\psi}(A'_n)) \ge \SS_n +B_{2k_n}> \beta'_{2k_n} \,, \quad  \TTT_n(\vec{\psi}(A''_n)) \le \SS_n -B_{2k_n}<\beta_{2k_n}.$$
Hence, using the continuity of $\TTT_n(\vec{\psi}(\cdot))$, we see that there are $a'_n, a''_n$
such that $A'_n<a'_n<a''_n<A''_n$ and
\begin{equation}\label{hh}
   \TTT_n(\vec{\psi}(a'_n)) = \beta'_{2k_n} \,, \quad  \TTT_n(\vec{\psi}(a''_n)) = \beta_{2k_n}.
\end{equation}

 Then, following the argument of  \cite[\S 5.1.1]{CFPrestate}, we define
\begin{equation*}
\begin{split}
\mathcal{\hat{I}}_n= \mathcal{\hat{I}}_n^{\ee^+}(\tau,\TT^+):= &   \left\{ a \in \mathcal{I}_{n}  \mid
  \beta_{2k_n} \le \TTT_{n}(\vec{\psi}(a)) \le \beta'_{2k_n}   \right\}
\end{split}
\end{equation*}
 which is a closed nonempty set.
Further, possibly shrinking $\mathcal{\hat{I}}_n$,  we can assume that
 \begin{equation}\label{calpropn}
 \begin{split}
 &  \TTT_n(\vec{\psi}(\cdot)): \mathcal{\hat{I}}_n \to [\beta_{2k_n} , \beta'_{2k_n}]   \qquad \textrm{is surjective}.
                                 \end{split}
 \end{equation}
Then, reasoning as in Remark \ref{closest}, we denote by $\mathcal{\check{I}}_n=\mathcal{\check{I}}_n^{\ee^+}(\tau,\TT^+)$ the
 ``closed interval closest to $0$'' having  property \eqref{calpropn}.

Now, using \assump{P1}, \eqref{calpropn}     and   Lemma \ref{key1}
 we see that there are
   $\mathcal{A}^-_n, \mathcal{A}^+_n \in \mathcal{\check{I}}_n $ such that
   $d_n(\vec{\psi}(\mathcal{A}^-_n))=-\ep^{(1+\nu)/\underline{\sigma}}$ and
   $d_n(\vec{\psi}(\mathcal{A}^+_n))=\ep^{(1+\nu)/\underline{\sigma}}$.

Finally, we set
\begin{equation*}\label{calhatAn}
  \begin{array}{l}
    \mathcal{\check{A}^-}_n:= \min \{ a \in \mathcal{\check{I}}_n \mid  d_n(\vec{\psi}(a))= -\ep^{(1+\nu)/\underline{\sigma}}\}, \\
    \mathcal{\check{A}^+}_n:= \min \{ a \in \mathcal{\check{I}}_n \mid  d_n(\vec{\psi}(a))= \ep^{(1+\nu)/\underline{\sigma}}\} ,\\
    \underline{\mathfrak{A}}_n= \min \{ \mathcal{\check{A}^-}_n; \mathcal{\check{A}^+}_n \}, \qquad \overline{\mathfrak{A}}_n= \max \{
    \mathcal{\check{A}^-}_n;
    \mathcal{\check{A}^+}_n \},
  \end{array}
\end{equation*}
and   we denote by
 \begin{equation}\label{calJne}
  \mathcal{J}_n=\mathcal{J}_n^{\ee^+}(\tau,\TT^+):=[\underline{\mathfrak{A}}_n;\overline{\mathfrak{A}}_n]  \subset \mathcal{\check{I}}_n
  \subset \mathcal{J}_{n-1},
\end{equation}
 so the proposition is proved.
\end{proof}

Now, following again \cite[\S 5.1.1]{CFPrestate},  for any $a \in \mathcal{J}_n$ we set
  \begin{equation}\label{caldefaln}
    \al_{k_j}=\al_{k_j}^{\ee^+}(\ep,\vec{\psi}(a), \TT^+)= \TTT_j(\vec{\psi}(a))- T_{2k_j}  \qquad \textrm{for any $j=1, \ldots, n$.}
  \end{equation}

Repeating word by word the argument of the proofs of the results in Lemmas~5.10  and
 5.12    in \cite{CFPrestate},   but replacing the assumption $(d,\tau) \in J_n \times [b_0,b_1]$ by
 $a \in \mathcal{J}_n$ one can reprove in this context the analogous results to get the following.

\begin{lemma}\label{call.stepn}
Let $\ee^+ \in \EE^+$; then for any $1 \le n \le  j^+_S=j^+_S (\ee^+)$ (where $j^+_S=+\infty$ if $\ee^+ \in \hat{\EE}^+$, and $j^+_S<+\infty$
if
$\ee^+ \in
\EE^+_0$), and
  any $a \in \mathcal{J}_n$, the trajectory $\x(t,\tau(a); \Q_s(\vec{\psi}(a)))$ satisfies $\bs{C^+_{\ee^+}}$ for any $\tau(a) \le t \le
  \TTT_n(\vec{\psi}(a))$,
  i.e., for any $\tau(a) \le t \le T_{2 k_n} +\al_{k_n}$.
  Further,    $|\al_{k_n}| \le B_{2k_n}$.
\end{lemma}

If $\ee^+ \in \hat{\EE}^+$, we set
\begin{equation}\label{Jinftypsi}
\begin{split}
  &\mathcal{J_{+\infty}} = \mathcal{J}^{\ee^+}_{+\infty}(\tau, \TT^+):= \cap_{n=1}^{+\infty}\mathcal{J}_n^{\ee^+}(\tau, \TT^+) , \\
   & a_*=a_*^{\ee^+}(\tau, \TT^+)= \min \{a \in \mathcal{J}_{+\infty}\}.
  \end{split}
\end{equation}
Notice that  we are intersecting infinitely many nested compact intervals, so $\mathcal{J_{+\infty} }$
is a  non-empty compact and connected set: whence the minimum $a_*$ exists.

\begin{proposition}\label{Sintersectionhat}
 Proposition  \ref{Sintersection} and Theorem \ref{T.connection}  hold true for any $\ee^+ \in \hat{\EE}^+$.
\end{proposition}
 \begin{proof}
 Let $\ee^+ \in \hat{\EE}^+$, then from Lemma \ref{call.stepn} we immediately obtain $a_*$ such that $\vec{\psi}(a_*) \in S^+$.
 Then by construction $\x(t,\tau; \Q_s(\vec{\psi}(a_*))$ has property $\bs{C_{\ee^+}^+}$.
 \end{proof}

\begin{proof}[\textbf{Proof of Proposition \ref{Sintersection}}]

If $\ee^+ \in \hat{\EE}^+$ the result follows from Proposition~\ref{Sintersectionhat}.

So, let  $\ee^+ \in \EE^+_0$ and consider the set
$\mathcal{J}_{j^+_S}= \cap_{n=1}^{j^+_S} \mathcal{J}_{n}$ constructed via Lemma~\ref{call.stepn}, which is again a compact interval.
The function
$$ d_{j^+_S}(\vec{\psi}(\cdot)): \mathcal{J}_{j^+_S} \to \R$$
is $C^r$ and its image contains the interval $[-\ep^{(1+\nu)/\und{\si}} , \ep^{(1+\nu)/\und{\si}}]$;
further the function
$$ \TTT_{j^+_S} (\vec{\psi}(\cdot)): \mathcal{J}_{j^+_S}  \to [\beta_{2k_{j^+_S}}, \beta'_{2k_{j^+_S}}] $$
is well defined and $C^r$.

 So   the set
$$\mathcal{I}_{+\infty}= \mathcal{I}_{+\infty}^{\ee^+}(\TT^+):= \{a \in \mathcal{J}_{j^+_S}    \mid d_{j^+_S} (\vec{\psi}(a))=0 \}$$
is a non-empty compact set. Let us denote by $a_*= \min \mathcal{I}_{+\infty}$ and by
$\mathcal{J}_{+\infty}=\mathcal{J}_{+\infty}^{\ee^+}(\tau, \TT^+)$ the largest connected
component
of $\mathcal{I}_{+\infty}$ containing $a_*$.
Then $\x(t,\tau; \Q_s(\vec{\psi}(a)))$ has property $\bs{C^+_{\ee^+}}$   for any $t \in [\tau,
\TTT_{j^+_S}(\vec{\psi}(a))]$, whenever  $a \in
\mathcal{J}_{+\infty}$.

Further  $\x(\TTT_{j^+_S}(\vec{\psi}(a)),\tau(a); \Q_s(\vec{\psi}(a)))= \P_s(\TTT_{j^+_S}(\vec{\psi}(a)))$, hence \\
\mbox{$\x(t,\tau(a); \Q_s(\vec{\psi}(a))) \in \tilde{W}^s(t) \subset \Om^+$} for any $t \ge \TTT_{j^+_S}(\vec{\psi}(a))$.
So from Proposition~\ref{forPropC}, recalling that
$\alpha_{k_j}= \TTT_{j}(\vec{\psi}(a))-T_{2k_j}$ for any $1 \le j \le j^+_S$,
 we see that
$$ \|\x(t,\tau(a); \Q_s(\vec{\psi}(a)))- \ga^+(t- T_{2k_{j^+_S}}-\alpha_{k_{j^+_S}})\| \le \frac{c^*}{2} \ep$$
for any $t \ge T_{ 2k_{j^+_S}}+\alpha_{k_{j^+_S}}$.
Moreover from Lemma \ref{forgottentimes} we see that
$$\|\x(t,\tau(a); \Q_s(\vec{\psi}(a)))\| \le c^* \ep \quad  \textrm{for any
$t \ge T_{ 2k_{j^+_S}+1} \ge  \TTT_{j^+_S}(\vec{\psi}(a)) + T_b$.}$$
Hence we see that $ \x(t,\tau(a); \Q_s(\vec{\psi}(a)))$ has property $\bs{C^+_{\ee^+}}$ for any $t \ge \tau(a)$, whenever
 $a\in J_{+\infty}$  and we conclude the proof of Proposition \ref{Sintersection}.
\end{proof}

\medskip

Now, the assertion of Theorem \ref{T.connection} immediately follows from Proposition~\ref{topological} setting
$$\tilde{\chi}^+(\ee^+)=  \left\{ (\Q_s(d,\tau),\tau)   \mid  (d,\tau) \in \tilde{S}^+   \right\},   \quad
\tilde{\chi}^-(\ee^-)=  \left\{ (\Q_u(d,\tau),\tau)   \mid  (d,\tau) \in \tilde{S}^-   \right\}.
$$

\subsection{Proof of Theorem \ref{T.connectionLa}}\label{proof.connection2}

  In this section we assume that   the hypotheses of Theorem \ref{T.connectionLa}   are satisfied.
   Further   $\ep \in ]0, \ep_0[$, $\nu \ge \nu_0$ are fixed and
  $\TT^+=(T_j)$, $j \in \Z^+$ and $\TT^-=(T_j)$, $j \in \Z^-$ are fixed sequences satisfying
    \eqref{TandKnu} for $j \in \Z^+$ and \eqref{TandKnu_tau+}, and \eqref{TandKnu} for $j \le -2$ and \eqref{TandKnu_tau-},
respectively;
  moreover we assume \eqref{minimal}.
  Finally we let $\ee^+ \in \EE^+$ and $\ee^- \in \EE^-$, then we set
 \begin{equation}\label{defS+la}
 \begin{split}
 S^{\Lambda,+}=     &   \left\{ (d,\tau) \in A_0\middle|  \begin{array}{l}
                                      \x(t,\tau ; \Q_s(d,\tau)) \; \textrm{ has property $\bs{C^+_{\ee^+}}$ and } \\
                                      |\alpha_j| \le  \adown_j-\aup_j \le \l1  \; \textrm{ for any $j \in \Z^+$}
                                    \end{array}  \right\},\\
S^{\Lambda,-}= & \left\{ (d,\tau) \in A_0 \middle|  \begin{array}{l}
                                      \x(t,\tau ; \Q_u(d,\tau)) \; \textrm{ has property $\bs{C^-_{\ee^-}}$ and } \\
                                     |\alpha_j| \le \adown_j-\aup_j \le  \l1  \; \textrm{ for any $j \in \Z^-$}
                                    \end{array}  \right\} .
 \end{split}
\end{equation}

Repeating word by word the proof of Propositions \ref{Sclosed} and \ref{Sintersection} in Section \ref{S.theorem2.1},
but replacing $B_j$ by $\adown_j-\aup_j$  for any $j \in \Z$,
we obtain the following.
\begin{proposition}\label{Sclosedbis}
Let $\ee^+ \in \EE^+$ be fixed; then the set $S^{\Lambda,+} \subset A_0$ is closed.
\end{proposition}
\begin{proposition}\label{Sintersectionbis}
Let $\ee^+ \in \EE^+$ be fixed; then for any continuous path $\vec{\psi} : [0,1] \to A_0$, $\vec{\psi}(a)=(d(a),\tau(a))$ such that
$d(0)=0$;
$d(1)=
\ep^{(1+\nu)/\underline{\sigma}}$ there is
$\bar{a} \in ]0,1[$ (depending on $\vec{\psi}$) such that  $\vec{\psi}(\bar{a}) \in S^{\Lambda,+}$.
\end{proposition}
Then, as in  the proof of Proposition \ref{Sintersection},     we can apply Lemma \ref{top}, choosing  $\mathcal{R}= A_0$, $E= S^{\Lambda,+}$,
to
conclude
the following.
\begin{proposition}\label{topologicalbis}
Let the hypotheses of Theorem \ref{T.connectionLa} be satisfied.
Then the set $S^{\Lambda,+}$  contains a closed connected set $\tilde{S}^{\Lambda,+}$
which intersects the lines $\tau=b_0$ and $\tau=b_1$,
while the set $S^{\Lambda,-}$  contains a closed connected set $\tilde{S}^{\Lambda,-}$
which intersects the lines $\tau=b_0$ and $\tau=b_1$.
\end{proposition}
\begin{proof}[\textbf{Proof of Theorem \ref{T.connectionLa}}]
Let us focus first on forward time, so let $\TT^+=(T_j)$, $j \in \Z^+$, be a fixed sequence satisfying
\eqref{TandKnu} for $j \in \Z^+$   and \eqref{TandKnu_tau+} for any   $\ee^+ \in \EE^+$. Then the existence of the compact connected subset
$\tilde{\chi}^{\Lambda,+}(\ee^+) \subset \chi^{\Lambda,+}(\ee^+)$ is obtained simply by setting
$$\tilde{\chi}^{\Lambda,+}(\ee^+)=  \left\{ (\Q_s(d,\tau),\tau)   \mid  (d,\tau) \in \tilde{S}^{\Lambda,+}   \right\} .$$
  The estimates on the sequence $\alpha_j^{\ee^+}(\ep)$ are obtained repeating word by word the proof of  Lemma 5.27 and Remark 5.28 in \cite{CFPrestate}.
  The part of the proof concerning backward time, i.e., $\TT^-=(T_j)$, $j \in \Z^-$ and   $\ee^- \in \EE^-$, is then obtained
  with an inversion of time argument as in \cite[\S 5.3]{CFPrestate}.
\end{proof}

From Proposition 5.30   in \cite{CFPrestate} we find the following.
\begin{proposition}\label{closeW}
  Let $\x(t,\tau; \xxi)$ be a trajectory  constructed using the first part of either of Theorems \ref{T.connectionLa} or \ref{T.connection}, i.e.~a trajectory satisfying property $\bs{C_{\ee^+}^+}$. Then
  \begin{equation}\label{e.closeW}
    \x(t,\tau; \xxi) \in [B(\tilde{W}(t), \ep^{(1+\nu)/2}) \cap \tilde{V}(t)]
  \end{equation}
  for any $t \ge \tau$.

  Analogously
  let $\x(t,\tau; \xxi)$ be a trajectory  constructed using the
  second part of either of Theorems \ref{T.connectionLa} or \ref{T.connection}, i.e.
  a trajectory satisfying property $\bs{C_{\ee^-}^-}$. Then
   \eqref{e.closeW} holds for any $t \le \tau$.
\end{proposition}

 \section{Proof of the main results -- Theorems \ref{main.periodic1}, \ref{main.weak1},  \ref{veryweak}}\label{S.mainproof}

In this section
 we glue together   trajectories  chaotic in the past with the ones chaotic in the future, obtained
in the results of Section  \ref{S.connection}.
We  start by proving Theorem~\ref{theoremkey} below, which is stated in the setting of Theorem \ref{veryweak}.
Namely, we look for some pairs $(d^{\pm}_{\star},\tau_{\star})$ such that
$\Q_s(d^+_{\star},\tau_{\star})=\Q_u(d^-_{\star},\tau_{\star})$ and $(d^+_{\star},\tau_{\star})
\in \tilde{S}^+(\TT^+, \ee^+) $,
$(d^-_{\star},\tau_{\star}) \in \tilde{S}^-(\TT^-, \ee^-) $:
see  Proposition \ref{topological}.
For this purpose we rely on the
intermediate value theorem
 and on the sign changes of $\MM(\tau)$.
We focus on the following stripe:
$$\hat{A}:= \{(d,\tau) \mid  d \in  [- \sqrt{\ep}, \sqrt{\ep} ] ,\, \tau \in [b_0,b_1]    \}. $$
 We recall that $d^+$ and $d^-$ measure respectively the distance of $\Q_s(d^+,\tau)$ from $\P_s(\tau)$ and of
$\Q_u(d^-,\tau)$ from $\P_u(\tau)$ and that they are both $o(\ep)$ (in fact smaller than $\ep^{(1+\nu)/\und{\si}}$
with $\nu \ge \nu_0$, see \eqref{J0}).
Further the distance between
$\P_s(\tau)$ and $\P_u(\tau)$ is $O(\ep)$ and it is measured by the Melnikov function, up to the first order in $\ep$.
For this reason we need to work on the stripe  $\hat{A}$, which is  a bit larger than $O(\ep)$.

We are ready to prove the following result.
\begin{theorem}\label{theoremkey}
  Assume that  the hypotheses of Theorem  \ref{veryweak} hold. Let $\ee \in \EE$ be a sequence satisfying
   $\ee_0=1$. Then there is a compact  connected  set $X(\ee, \TT)\subset L^0$ and a sequence $(\al_j^{\ee}(\ep))$, $j \in
   \Z$,
   such that if $\xxi \in X(\ee, \TT)$,
   then $\x(t,T_0+\al_0^\ee(\ep); \xxi)$ has property  $\bs{C_{\ee}}$.  Further $|\alpha_0| \le B_0$.
\end{theorem}

 Before starting the proof, we collect some facts for future reference.

Let us introduce the set
\begin{equation}\label{defhatS}
\begin{split}
     & \hat{S}^- = \hat{S}^-(\TT^-, \ee^-):= \{ (d+\DDD(\P_u(\tau), \P_s(\tau)), \tau) \mid   (d,\tau) \in \tilde{S}^-(\TT^-, \ee^-)
     \},
\end{split}
\end{equation}
 where $\tilde{S}^-$ is as in Proposition \ref{topological}.
Notice that by construction $\hat{S}^- \subset \hat{A}$.
Further, $\hat{S}^-$ is connected, see Proposition \ref{topological}, and splits $\hat{A}$ in two relatively open and connected
components say $E^r$ and $E^l$, see Figure \ref{intersection}.  We can assume that
$E^r$ contains the line $d= \sqrt{\ep}$, while
$E^l$ contains the line $d= -\sqrt{\ep}$.

  Let us set
  \begin{gather*}
  	\tilde{J}_{+\infty}(\tau)= \{ d\in\R\mid (d,\tau) \in \tilde{S}^+ \},\\
  	\hat{J}_{-\infty}(\tau)= \{ d\in\R\mid (d,\tau) \in \hat{S}^- \},
  \end{gather*}
 $\tilde{S}^+$ being as in Proposition \ref{topological}.

 \begin{remark}
       If $ \hat{d}  \in \tilde{J}_{+\infty}(\tau) \cap \hat{J}_{-\infty}(\tau)$ and
       $ \hat{d}= d^- + \DDD(\P_u(\tau), \P_s(\tau))$ then
       $$\hat{d}=\DDD(\Q_s(\hat{d},\tau), \P_s(\tau))=\DDD(\Q_u(d^-,\tau), \P_s(\tau)).$$
       Hence we find $\Q_s(\hat{d},\tau)=\Q_u(d^-,\tau)$. In fact $\hat{d}$ measures the directed distance  on $\Om^0$ between $\Q_u$ and
       $\P_s$.
     \end{remark}
     To prove Theorem \ref{theoremkey} we need a further classical result, again borrowed from \cite{CFPrestate}, see e.g.
     the nice introduction of \cite{JPY} or \cite[\S 4.5]{GH} for a proof.
\begin{remark}\label{rMelni}
	Assume that $\f^{\pm}$ and $\g$ are $C^r$, $r>1$.
Observe that if $0<\ep \ll 1$ then,  for any fixed $\tau \in \R$,
	there are $c>0$ (independent of $\ep$ and $\tau$) and a monotone increasing and continuous function
$\bar{\omega}(\ep)$ such that $\bar{\omega}(0)=0$ and
	$$\dist(\P_s(\tau),\P_u(\tau))= c \ep \MM(\tau)+ \ep  \omega(\tau,\ep) \, , \qquad  |\omega(\tau,\ep)| \le
\bar{\omega}(\ep).$$
Further, if $r \ge 2$, we can find $C_M>0$ such that $\bar{\omega}(\ep) \le C_M \ep$.
\end{remark}

\begin{proof}[\textbf{Proof of Theorem \ref{theoremkey}}]

We claim that
\\ $\bullet$ \ \emph{there is  $(d_{\star},\tau_{\star}) \in (\hat{S}^-(\TT^-, \ee^-) \cap  \tilde{S}^+(\TT^+, \ee^+))$. }
\\
In fact for any
$d^+_i \in \tilde{J}_{+\infty}(b_i)$ and any $\hat{d}^-_i= d^-_i+ \DDD(\P_u(b_i), \P_s(b_i)) \in \hat{J}_{-\infty}(b_i)$ we have
$$d^+_i-\hat{d}^-_i = d^+_i- \DDD(\P_u(b_i), \P_s(b_i))- d^-_i$$
for $i=0,1$.  From \assump{P1} we see that $\MM(b_0)<-\bar{c}<0<\bar{c}<\MM(b_1)$, so, recalling that $d^{\pm}_i \in [0, \ep^{(1+\nu)/ \underline{\sigma}}]$ for $i=0,1$,
  using Remark \ref{rMelni},   for any $0<\ep \le \ep_0$  we find
\begin{equation}\label{d+id-i}
  \begin{split}
     d^+_0-\hat{d}^-_0 \le  &   \frac{c}{2} \MM(b_0) \ep +\ep^{1+\bar{\nu} } \le    \frac{c}{4} \MM(b_0) \ep  <  -
     \frac{c\bar{c}}{4} \ep , \\
     d^+_1-\hat{d}^-_1 \ge  &  \frac{c}{2} \MM(b_1)\ep  -\ep^{1+\bar{\nu} } \ge   \frac{c}{4} \MM(b_1) \ep  >
     \frac{c\bar{c}}{4} \ep
  \end{split}
\end{equation}
where $\bar{\nu}= \frac{1-\underline{\sigma}+\nu}{\underline{\sigma}}>1$.
Since \eqref{d+id-i} holds true for any $d^+_i \in \tilde{J}_{+\infty}(b_i)$ and any $\hat{d}^-_i \in \hat{J}_{-\infty}(b_i)$, we see
that
$\tilde{J}_{+\infty}(b_0)\times\{b_0\}  \subset  E^l$ while $\tilde{J}_{+\infty}(b_1)\times\{b_1\}  \subset  E^r$.
Then recalling that $\tilde{J}_{+\infty}(b_i)\times\{b_i\} \subset \tilde{S}^+$ for $i=0,1$ and that $\tilde{S}^+$
is connected we prove the claim.

Now let $d^-_\star=d_\star- \DDD(\P_u(\tau_\star), \P_s(\tau_\star))$; then $(d^-_\star,\tau_\star) \in
\tilde{S}^-=\tilde{S}^-(\TT^-, \ee^-)$ and
$$\xxi_\star = \Q_s(d_\star,\tau_\star)= \Q_u(d^-_\star,\tau_\star)$$
is such that $\x(t,\tau_\star; \xxi_\star)$ has property $\bs{C_{\ee}}$ for any $t \in \R$.

\begin{figure}[t]
\centerline{\epsfig{file=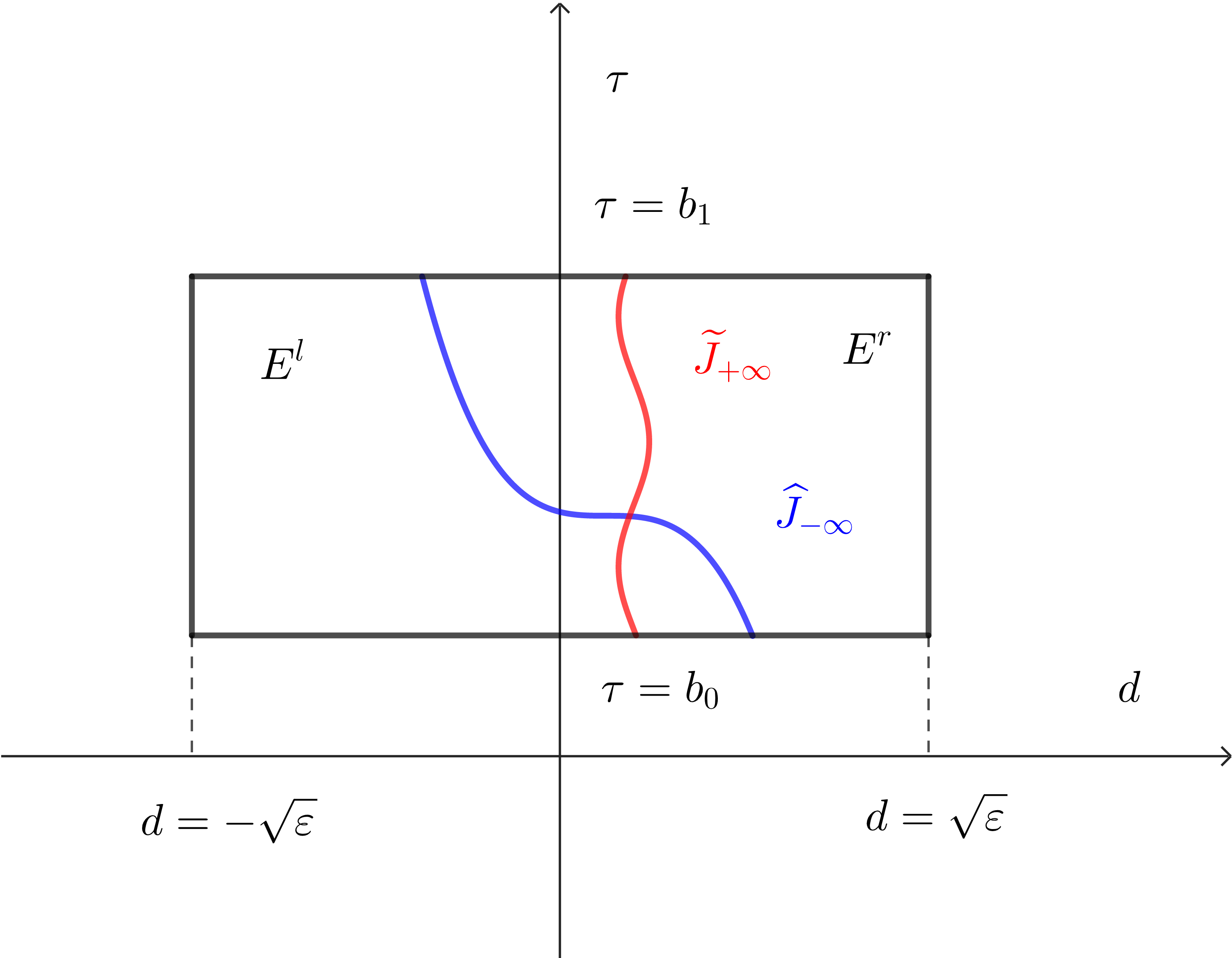, width = 8 cm}}
\caption{In this picture we illustrate the set $\hat{A}$ and the proof of Theorem \ref{theoremkey}.
}\label{intersection}
\end{figure}

Finally, define
\begin{equation}\label{alpha0}
  \alpha_0= \alpha^{\TT, \ee}_0 (\ep)= \tau_\star-T_0;
\end{equation}
then by construction $\x(t,T_0+ \alpha_0 ; \xxi_\star)$ has property $\bs{C_{\ee}}$.
Further $|\tau_\star-T_0|<b_1-b_0=B_0$ so the result is proved.
\end{proof}

 \begin{remark}\label{uselesspos}
 Note that $\xxi_\star$ lies on $\Om^0$ but it is not placed between $\P_u(\tau_\star)$, $\P_s(\tau_\star)$:
 indeed $\DDD(\xxi_\star, \P_u(\tau_\star))= d^-_\star>0$ and $\DDD(\xxi_\star, \P_s(\tau_\star))=  d_\star>0$.
 \end{remark}

\subsection{Setting of Theorems \ref{main.periodic1},
\ref{main.weak1}:  $|\alpha_j|$ bounded or infinitesimal}\label{S.theorem2.2}

In this section
we work in the setting of Theorems \ref{main.periodic1}
and \ref{main.weak1}. Here we glue together the trajectories chaotic in the past and in the future  obtained via Theorem \ref{T.connectionLa}, so we will always assume  \assump{P1}, \eqref{minimal} and we consider a sequence
$\TT$ satisfying \eqref{TandKnu}  for $\nu \ge\nu_0$.
In fact we just need to adapt slightly the argument of Section~\ref{S.theorem2.1} in order to obtain better estimates on $|\alpha_j|$.
Therefore we replace the set $\hat{A}$ by the set $\hat{A}^{\Lambda}$ defined as
$$\hat{A}^{\Lambda}:= \{(d,\tau) \mid  d \in  [- \sqrt{\ep}, \sqrt{\ep} ] ,\, \tau \in [\aup_0,\adown_0]    \}, $$
 further we replace $\tilde{S}^{\pm}$  defined in Proposition  \ref{topological}   by the sets
$\tilde{S}^{\Lambda,\pm}$ 
defined in Proposition  \ref{topologicalbis},
  and the set $\hat{S}^{-}$ by the set $\hat{S}^{\Lambda,-}$ defined as
\begin{equation}\label{defhatSbis}
\begin{split}
      \hat{S}^{\Lambda,-} &= \hat{S}^{\Lambda,-}(\TT^-, \ee^-)\\
     &:= \{ (d+\DDD(\P_u(\tau), \P_s(\tau)), \tau) \mid   (d,\tau) \in
     \tilde{S}^{\Lambda,-}(\TT^-, \ee^-)
     \}.
\end{split}
\end{equation}

So we can prove the following.
\begin{theorem}\label{theoremkeybis}
  Assume that  the hypotheses of Theorem \ref{main.weak1} hold;   in particular let $\TT$ be a sequence satisfying \eqref{TandKnu}
  and \eqref{minimal}.
   Let $\ee \in \EE$ be a sequence satisfying
   $\ee_0=1$. Then there is a compact  connected  set $X(\ee, \TT)\subset L^0$  and a sequence $(\al_j^{\ee}(\ep))$, $j \in
   \Z$,
   such that if $\xxi \in X(\ee, \TT)$,
   then $\x(t,T_0+\al_0^\ee(\ep); \xxi)$ has property  $\bs{C_{\ee}}$.  Further $|\alpha_0| \le \l1$.
\end{theorem}
\begin{proof}
  We just need to adapt the argument of the proof of Theorem \ref{theoremkey}.

  We claim that
\\ $\bullet$ \emph{there is}
$(d_{\star}^{\Lambda},\tau_{\star}^{\Lambda}) \in (\hat{S}^{\Lambda,-}(\TT^-, \ee^-) \cap  \tilde{S}^{\Lambda,+}(\TT^+, \ee^+))$.
\\
In fact let
$d^{\pm}_{\uparrow} \in \tilde{J}_{\pm\infty}(\aup_0)$, $d^{\pm}_{\downarrow} \in \tilde{J}_{\pm\infty}(\adown_0)$,
 and  $\hat{d}^-_{\uparrow}= d^-_{\uparrow}+ \DDD(\P_u(\aup_0), \P_s(\aup_0)) \in \hat{J}_{-\infty}(\aup_0)$,
 $\hat{d}^-_{\downarrow}= d^-_{\downarrow}+ \DDD(\P_u(\adown_0), \P_s(\adown_0)) \in \hat{J}_{-\infty}(\adown_0)$;
from \eqref{minimal} and \eqref{defbetaj} we can assume  $\MM(\aup_0)=-\bar{c} \delta<0<\bar{c} \delta =\MM(\adown_0)$
(or we have opposite sign conditions but the argument goes through without significant changes);
hence using Remark \ref{rMelni}, for any $0<\ep \le \ep_0$ we find
\begin{equation}\label{d+id-ibis}
  \begin{split}
     d^{+}_{\uparrow}-\hat{d}^{-}_{\uparrow} \le  &   \frac{c}{2} \MM(\aup_0) \ep +\ep^{1+\bar{\nu} } \le    \frac{c}{4} \MM(\aup_0) \ep =  -
     \frac{c\bar{c}}{4} \delta \ep , \\
     d^{+}_{\downarrow}-\hat{d}^{-}_{\downarrow}\ge  &  \frac{c}{2} \MM(\adown_0)\ep  -\ep^{1+\bar{\nu} } \ge   \frac{c}{4} \MM(\adown_0) \ep
     =
     \frac{c\bar{c}}{4}  \delta\ep ,
  \end{split}
\end{equation}
 with $\bar{\nu}= \frac{1-\underline{\sigma}+\nu}{\underline{\sigma}}>1$.
Then, arguing as in the proof of Theorem \ref{theoremkey}, we prove the claim.

Now let $d^{\Lambda,-}_\star=d_\star^{\Lambda}- \DDD(\P_u(\tau_\star^{\Lambda}), \P_s(\tau_\star^{\Lambda}))$; then
$(d^{\Lambda,-}_\star,\tau_\star^{\Lambda}) \in
\tilde{S}^{\Lambda,-}=\tilde{S}^{\Lambda,-}(\TT^-, \ee^-)$ and
$$\xxi_\star = \Q_s(d_\star^{\Lambda},\tau_\star^{\Lambda})= \Q_u(d^{\Lambda,-}_\star,\tau_\star^{\Lambda})$$
is such that $\x(t,\tau_\star^{\Lambda}; \xxi_\star)$ has property $\bs{C_{\ee}}$ for any $t \in \R$ or, in other words,
$\x(t,T_0+\alpha_0,\xxi_\star)$ with $\alpha_0=\tau_\star^{\Lambda}-T_0$ has property $\bs{C_{\ee}}$ for any $t \in \R$.

Further by construction   $ |\alpha_0| = |\tau_\star^{\Lambda}-T_0| \le \adown_0-\aup_0 \leq \l1$, so the theorem is proved.
\end{proof}

\begin{cor}\label{c.alpha0}
  Let the assumptions of Theorem \ref{theoremkeybis} be satisfied but replace \eqref{minimal} by \eqref{isogen}.
 Then there is a continuous increasing function $\bar{\omega}_{\al}(\ep)$
 (independent of $\ee \in \EE$) such that $\bar{\omega}_{\al}(0)=0$
  and $|\al_0^{\ee}|\le \bar{\omega}_{\al}(\ep)+  \lzero$. In particular if  \eqref{isolated}  holds then
  $|\al_0^{\ee}|\le \bar{\omega}_{\al}(\ep)$.

  Now replace \eqref{minimal} by \eqref{nondeg} and assume that $\f$ and $\g$ are $C^r$ with $r \ge 2$; then we find $c_{\al}>0$ such that $|\al_0^{\ee}|\le c_{\al} \ep$,
  where $c_{\al}$ depends neither on $\ep$ nor on $\ee \in \EE$.
\end{cor}

\begin{proof}
With the same notation as in the proof of Theorem \ref{theoremkeybis},
observe first that
 \begin{equation}\label{estDDDbis}
				\begin{split}
					&|\DDD(\P_s(\tau_\star^{\Lambda}), \P_u(\tau_\star^{\Lambda}))|= |c \ep \MM(\tau_\star^{\Lambda})+\ep
\omega(\tau_\star^{\Lambda},\ep)|    \\
					& =|\DDD(\P_s(\tau_\star^{\Lambda}), \xxi_\star )+\DDD( \xxi_\star, \P_u(\tau_\star^{\Lambda}))| \le
\ep^{\frac{1+\nu}{\underline{\sigma}}}=
\ep^{1+\bar{\nu}}
				\end{split}
		\end{equation}
  with $\bar{\nu}= \frac{1-\underline{\sigma}+\nu}{\underline{\sigma}}>1$, where we have used Remark \ref{uselesspos}.

 We recall that $\f^{\pm}$ and $\g$ are $C^r$ with $r>1$, and that the Melnikov function $\MM$ inherits the same regularity.
 Hence there is
 $c_r>0$ such that $|\omega(\tau_\star^{\Lambda},\ep)|\le c_r (\ep+\ep^{r-1})$.
 Assume first  \eqref{isogen},  so that  $|\MM(\tau)| \ge \omega_M(|\tau-T_0-\lzero|)$  when $ \tau \in [T_0+\lzero, \adown_0]$.

 Then from \eqref{estDDDbis}, we find
  \begin{equation}\label{xxi12bis}
  \begin{split}
      & |\omega_M(|\tau_\star^{\Lambda}-T_0-\lzero|)| \le  |\MM(\tau_\star^{\Lambda})| \le   \frac{c_r (\ep+ \ep^{r-1})+
      \ep^{\bar{\nu}}}{c},
  \end{split}
  \end{equation}
  when $ \tau_\star^{\Lambda} \in [T_0+\lzero, \adown_0]$. Hence there is $\bar{\omega}_{\alpha}(\ep)$ as in the statement such
  that
  \begin{equation}\label{tau12}
  \begin{split}
      &   |\alpha_0| =|\tau_\star^{\Lambda}-T_0|  \le \lzero + \omega_M^{-1}\left(  \frac{c_r (\ep+ \ep^{r-1})+ \ep^{\bar{\nu}}}{c}
      \right)=\lzero +\bar{\omega}_{\alpha}(\ep),
  \end{split}
  \end{equation}
    when $ \tau_\star^{\Lambda} \in [T_0+\lzero, \adown_0]$. The estimate when  $\tau_\star^{\Lambda} \in [\aup, T_0-\lzero]$
  is analogous and it is omitted;
    when $\tau_\star^{\Lambda} \in[T_0-\lzero, T_0+\lzero]$, so $|\alpha_0|= |\tau_\star^{\Lambda}-T_0| \le \lzero$, whence the result
    is proved.
  When \eqref{isolated} holds we obtain the estimates by setting $\lzero=0$.

Now we assume   $r \ge 2$ and that \eqref{nondeg} holds so $|\MM'(T_0)| \ge C>0$. Then, from \eqref{xxi12bis}, we  find $C_M>0$ such
that
  \begin{equation*}
  \begin{split}
      &  \frac{|\MM'(T_0)|}{2} |\tau_\star^{\Lambda}-T_0| \le |\MM(\tau_\star^{\Lambda})| \le  \frac{2 c_r \ep + \ep^{\bar{\nu}}}{c}  \le C_M
      \ep
      \, ,
  \end{split}
  \end{equation*}
  \begin{equation}\label{tau12*}
  \begin{split}
      &   |\alpha_0| =|\tau_\star^{\Lambda}-T_0|  \le  \frac{2 C_M \ep}{C}=:c_{\alpha} \ep.
  \end{split}
  \end{equation}
So the proof is concluded.
 \end{proof}

\begin{remark}
  Assume that $\MM^{(j)}(T_0)=0$ for $j=0, \ldots, k-1$ and $\MM^{(k)}(T_0)\ne 0$: then $|\MM(T_0+h)| \geq  C_k |h|^k$ for some
  $C_k>0$ when $|h|$ is small enough. If $\f^{\pm}$ and $\g$ are $C^r$ with $r \ge k \ge 2$, we can adapt slightly the
  previous argument
  in order to set $\omega_M(\ep)= C_k \ep^k$. Then we
  get
 $$|\alpha_0| = |\tau_\star^{\Lambda}-T_0| \le \left[\frac{2 C_M}{C_k} \ep \right]^{\frac{1}{k}} = O(\ep^{1/k}) .$$
\end{remark}
\begin{remark}\label{foralpha0}
  Let the assumptions of any of Theorems \ref{main.weak1}  and \ref{veryweak} be satisfied.
Assume now that $|\MM'(T_0)|=C > 0$, then  there is $c_{\alpha}>0$ such that
$|\al^{\ee}_0|< c_{\alpha} \ep$, uniformly  for any $\ee \in \EE$.
To prove this fact, we can repeat word by word the proof of Corollary \ref{c.alpha0}.
\end{remark}

\subsection{Proof of Theorems \ref{main.periodic1}, \ref{main.weak1}, \ref{veryweak}}

 Here we complete the proof of our main results.
Actually, we  just consider the setting of Theorem \ref{main.weak1} (and \ref{main.periodic1}): the proof of Theorem \ref{veryweak} follows exactly the same lines and it is omitted.

\begin{proof}[\textbf{Proof of Theorem \ref{main.weak1}}]
Let $\TT=(T_j)$, $j \in \Z$ be a fixed sequence satisfying \eqref{TandKnu}.
Then, via Theorem
  \ref{theoremkeybis}, for any  $\tilde{\ee} \in \EE$ such that $\tilde{\ee}_0=1$,
  we construct a (compact) set $X(\tilde{\ee})\subset L^0$ such that if $\xxi(\tilde{\ee}) \in X(\tilde{\ee})$
  then there is a sequence $\alpha_j^{\tilde{\ee}}$ such that the trajectory $\x(t, T_0+\alpha_0^{\tilde{\ee}}; \xxi(\tilde{\ee}))$
  exhibits property
  $\bs{C_{\tilde{\ee}}}$.
  Let us denote by
  \begin{equation}\label{defsigma0}
    \Sigma':= \bigcup \{   X(\ee) \mid \ee_0=1 \}.
  \end{equation}
Then fix $k \in (\Z \setminus \{0 \})$ and denote by $\tilde{\TT}^k= (\tilde{T}^k_j)$, $j \in \Z$, the sequence where
$\tilde{T}^k_j= T_{2k+j}$.
Notice that $\tilde{T}_0^k \in [ \aup_{k} ,  \adown_{k} ] \subset [\beta_{2k}, \beta'_{2k}]$ by
\eqref{defbetaj}. Now let $\bar{\ee}$ be
such that $\bar{\ee}_k=1$, we
denote by
$\tilde{\ee}^k=(\bar{\ee}_{k+j})_{j \in \Z}$, so that $\tilde{\ee}^k_0=\bar{\ee}_k=1$: we can apply Theorem
   \ref{theoremkeybis}  on the shifted sequence to construct the set $X^k( \tilde{\ee}^k)\subset L^0$ such that if
  $\xxi^k \in X^k( \tilde{\ee}^k)$ then the trajectory $\x(t, \tilde{T}^k_0+\tilde{\alpha}_0^{\tilde{\ee}^k}; \xxi^k)$
  has property   $\bs{C_{\tilde{\ee}^{\boldsymbol{k}}}}$  using the time sequence $\tilde{\TT}^k$.

  Let us denote by   $\xxi^0=\x(T_0+\tilde{\alpha}_{-k}^{\tilde{\ee}^k}, \tilde{T}^k_0+\tilde{\alpha}_0^{\tilde{\ee}^k}; \xxi^k)$
   and by
  $\bar{\alpha}_j^{\bar{\ee}}=\tilde{\alpha}_{j-k}^{\tilde{\ee}^k}$, so that
  $$\x(t,T_0+\bar{\alpha}_{0}^{\bar{\ee}} ; \xxi^0) \equiv  \x(t, \tilde{T}^k_0+\tilde{\alpha}_0^{\tilde{\ee}^k}; \xxi^k),$$
  for any $t \in \R$;
  then by construction $\x(t,T_0+\bar{\alpha}_{0}^{\bar{\ee}}; \xxi^0)$ has property $\bs{C_{\bar{\ee}}}$. So we have  constructed $X(\ee)$
    so that $\x(t,T_0+\alpha_{0}^{\ee}; \xxi^0)$ has property $\bs{C_{\ee}}$ for any $\ee$ such that $\ee_k=1$.
  Let us  define
     \begin{equation}\label{defsigmak}
  \begin{split}
       & \tilde{\Sigma}_k:= \bigcup \{   X(\ee) \mid \ee_{k}=1 \},
       \qquad \qquad \Sigma = (\bigcup_{k \in \Z \setminus \{ 0\}  } \tilde{\Sigma}_{k} ) \cup \Sigma' \cup \{\vec{0} \}.
  \end{split}
  \end{equation}
     Then the set $\Sigma=\Sigma (\TT)$ in \eqref{defsigmak} has the required properties, i.e., for any $\ee \in \EE$ there is
  $X(\ee) \subset \Sigma$ such that, for any $\xxi \in X(\ee)$, the   trajectory $\x(t,T_0+\alpha_{0}^{\ee}; \xxi)$ has property
  $\bs{C_{\ee}}$. Clearly the
  initial condition $\xxi=\vec{0}$
  corresponds to the null sequence $\ee_j=0$ for any $j \in \Z$.
\end{proof}

  \begin{remark}\label{r.disjointX}
    We emphasize that $\tilde{\Sigma}_k \cap \tilde{\Sigma}_j$  is nonempty. However $X(\ee) \cap X(\ee') = \emptyset$ if $\ee \ne \ee'$.
  \end{remark}

Now we use Theorem \ref{main.weak1}  to prove  Theorem
 \ref{main.periodic1}.

\begin{proof}[\textbf{Proof of Theorem \ref{main.periodic1}.}] 

Let us set
\begin{equation}\label{ctilde}\tilde{c} = 2 \sup
\{\|\dot{\ga}^+(t)\|; \|\dot{\ga}^-(s)\| \mid s<0<t \},
\end{equation}
and observe that $\tilde{c}>0$ is finite.
From the proof of Corollary 3.7 in \cite{CFPrestate} we see that 
 for any $j \in \Z$,  $t\in\R$,  we get
  \begin{equation}\label{easy}
     \| \ga(t+ \alpha_j(\ep))- \ga(t)\| \le  \tilde{c} |\alpha_j(\ep)|;
   \end{equation}
further
 \begin{itemize}
   \item assume \eqref{nondeg}, then   for any $j \in \Z$ we find
   \begin{equation}\label{easier}
 \sup_{t \in \R} \| \ga(t-T_{2j}- \alpha_j(\ep))- \ga(t-T_{2j})\|  \le  \tilde{c}   c_{\alpha} \ep
     \end{equation}
  where $c_{\alpha}>0$ is given in \eqref{tau12*} and it is independent of $j \in \Z$;
   \item assume
   \eqref{isolated}, then   for any $j \in \Z$ we find
 $$  \sup_{t \in \R} \| \ga(t-T_{2j}- \alpha_j(\ep))- \ga(t-T_{2j})\|   \le  \tilde{c}   \bar{\omega}_{\alpha} (\ep ) =:
 \omega_{\alpha} (\ep )$$
  where  $\bar{\omega}_{\alpha} (\ep )$ is given in \eqref{tau12}
  and it is independent of $j\in \Z$.
 \end{itemize}
 Theorem \ref{main.periodic1} follows from Theorem  \ref{main.weak1}
 simply choosing
 \begin{equation}\label{deftildecstar}
 \tilde{c}^*=c^*+ \tilde{c}   c_{\alpha}.
 \end{equation}
\end{proof}

  Corollary \ref{r.all} is proved with the same idea.

\section{Semi-conjugacy with the Bernoulli shift}\label{proof.Bernoulli}

Let $\si:{\cal E}\to{\cal E}$ be the (forward) Bernoulli shift, defined by
$\si(\ee):=(\ee_{m+1})_{m\in\Z}$ where $\ee=(\ee_m) \in \EE$.
In this section, using a classical argument, we show that the action of the flow of \eqref{eq-disc}
on the set $\Sigma$ constructed either via Theorem~\ref{main.periodic1} or via Corollary \ref{c.sigma} is semi-conjugated to the Bernoulli shift.
In the whole section we follow quite closely \cite[\S 6]{BF11} and we borrow some results from \cite[\S 6]{CFPrestate}.
For the analogous sets constructed via
Theorem \ref{veryweak} we are able to obtain just partial results.
Set
\begin{equation*}\label{defE}
\begin{gathered}
\hat{\cal E}:=\left\{\ee\in{\cal E} \mid \inf\{m\in\Z \mid \ee_m=1\}=-\infty,\, \sup\{m\in\Z \mid \ee_m=1\}=\infty\right\}, \\
{\cal E}_0^+:=\left\{\ee\in{\cal E} \mid \inf\{m\in\Z \mid \ee_m=1\}>-\infty,\, \sup\{m\in\Z \mid \ee_m=1\}=\infty\right\}, \\
{\cal E}_0^-:=\left\{\ee\in{\cal E} \mid \inf\{m\in\Z \mid \ee_m=1\}=-\infty,\, \sup\{m\in\Z \mid \ee_m=1\}<\infty\right\}, \\
{\cal E}_0:=\left\{\ee\in{\cal E} \mid \inf\{m\in\Z \mid
\ee_m=1\}>-\infty,\, \sup\{m\in\Z \mid \ee_m=1\}<\infty\right\}.
\end{gathered}
\end{equation*}
Note that $\hat{\cal E}$, ${\cal E}_0^-$, ${\cal E}_0^+$, ${\cal E}_0$ are
invariant under the Bernoulli shift.

The set ${\cal E}$ becomes a
totally disconnected compact metric space with the distance
\begin{equation}\label{metric}
d(\ee',\ee'') = \sum_{m\in\Z}\frac{|\ee'_{m}-\ee''_{m}|}{2^{|m|+1}}\, .
\end{equation}

Let $\mathcal{T}= (T_m )$, $m\in\Z$ be a fixed sequence of values satisfying \eqref{TandKnu}.
Following  \cite[\S 6]{BF11}
 we set
 $\mathcal{T}^{(k)}= ( T_{m+2k} )$  for any $k \in \Z$.
Further   let us denote by $$\sigma^k= \stackrel{k \;\text{times}}{\overbrace{\sigma \circ \cdots \circ \sigma}}.$$

Let $X(\ee, \TT)$  and $\Sigma(\TT)$ be the sets constructed via Theorem \ref{main.periodic1},
for any $k \in \Z$ we introduce the sets
\begin{equation}\label{defSk}
\Sigma_k := \left\{ \xxi_k= \x(T_{k}, T_0; \xxi_0) \mid \xxi_0 \in X(\ee,\TT), \, \ee \in \EE \right\}.
\end{equation}

\begin{remark}\label{notunique}
  Note that we  might have  $\xxi_0' \ne \xxi_0''$, such that $\xxi_0',\xxi_0'' \in  X(\ee,\TT)$, while in the classical case
  $X(\ee,\TT)$ is a singleton.
\end{remark}

Notice that the set $\Sigma(\TT)$ constructed via Theorem \ref{main.periodic1} satisfies
$\Sigma(\TT)= \Sigma_0$, where $\Sigma_0$ is as  in \eqref{defSk}.
 Further, using  \eqref{ej=1} and \eqref{ej=0} we get the following.

 \begin{remark}\label{remtwo}
 Assume the hypotheses of Theorem \ref{main.periodic1}, then for any $k \in \Z$ we find
$$ \Sigma_{2k} \subset [B(\ga(0), \tilde{c}^* \ep) \cup B(\vec{0}, \tilde{c}^* \ep)].$$

 Similarly  in the setting of Corollary  \ref{c.sigma}, if \eqref{isolated} holds for any $k \in \Z$ we find
$$
\Sigma_{2k} \subset [B(\ga(0), \omega_\alpha(\ep)) \cup B(\vec{0}, c^* \ep)].$$

Finally   in the setting of Corollary  \ref{c.sigma}, if \eqref{minimal} holds for any $k \in \Z$ we find
$$
\Sigma_{2k} \subset [B(\bs{\Gamma^{\Lambda^1}}, c^* \ep) \cup B(\vec{0}, c^* \ep)]$$
 where $\bs{\Gamma^{\Lambda^1}}= \{ \ga(t)
\mid |t| \le \Lambda^1\}$.
 \end{remark}

Further   $\x(t,T_0; \xxi_0) \in B(\bs{\Gamma}, \sqrt{\ep})$
for any $t \in \R$; hence from Remark  \ref{data.smooth} we see that local
 uniqueness and continuous dependence on initial data is ensured for any
trajectory $\x(t,T_0; \xxi_0)$ such that $\xxi_0 \in \Sigma_0$.
So we get the following.
\begin{remark}\label{remone}
  By construction, if $\xxi_0 \in \Sigma_0$ then $\xxi_k= \x(T_{2k},T_0; \xxi_0) \in \Sigma_{2k}$ and we have
  \begin{equation}\label{identity}
    \x(t,T_0; \xxi_0) \equiv \x(t,T_{2k}; \xxi_k)   \, , \quad \textrm{for any $t \in \R$}.
  \end{equation}
 Viceversa, if $\xxi_k \in \Sigma_{2k}$ then  $\xxi_0= \x(T_0,T_{2k}; \xxi_k) \in \Sigma_0$ and \eqref{identity} holds.
  Then
  $\xxi_0 \in X(\bar{\ee}, \TT)$ if and only if $\xxi_k \in X(\sigma^k(\bar{\ee}), \TT^{(k)})$.
 \end{remark}

 Let $\xxi_0 \in X(\ee, \TT)$, then we set
  $\Psi_0(\xxi_0)=\ee$, so that
    $\Psi_0: \Sigma_0 \to \EE$ is well defined and onto.

  Similarly, let $\xxi_k \in \Sigma_{2k}$, then there is a  uniquely defined
  $\xxi_0\in \Sigma_0$  such that $x(T_0, T_{2k}; \xxi_k)=\xxi_0$; further there is a uniquely defined
  $\ee \in \EE$ such that $\xxi_0 \in X(\ee, \TT)$. Let $\ee^k= \sigma^k(\ee)$, then we
  set $\Psi_k(\xxi_k)=\ee^k$, so that $\Psi_k: \Sigma_{2k} \to \EE$ is well defined.

 Notice that   $\Psi_k$ is onto by construction  but it might  be not injective, see Remark \ref{notunique}.

Adapting \cite[Proposition 6.4]{CFPrestate} we find the following
\begin{proposition}\label{continuityine}
Let the assumptions of   Theorem \ref{main.periodic1}  be satisfied,
then the map $\Psi_k: \Sigma_{2k} \to \EE$ is continuous.
\end{proposition}

Now let us fix $\rho>0$ small enough, independent of $\ep$, so that in $B(\bs{\Gamma}, \rho)$ no sliding phenomena may take place.
Then from Remark \ref{data.smooth} we see that for any $k \in \Z$ the function $F_k: B(\bs{\Gamma}, \rho) \to \Omega$,
$F_k(\xxi)= \x(T_{2k+2}, T_{2k}; \xxi)$ is a homemomorphism onto its image;
however, notice that $F_k$ is not a diffeomorphism, since the flow of \eqref{eq-disc} is continuous in the domain but not smooth.
Hence
by construction $F_k : \Sigma_{2k} \to \Sigma_{2(k+1)}$ is a well-defined homeomorphism too;
so adapting \cite[Theorem 6.5]{CFPrestate} we find the following.
\begin{theorem} \label{thm.bernoulli}
Let the assumptions of   Theorem \ref{main.periodic1}  be satisfied.
  Then for any $0<\ep \le \ep_0$, we find
\begin{equation}\label{eq.com}
\Psi_{k+1} \circ F_{k} (\xxi) = \sigma \circ \Psi_k (\xxi) \, , \qquad k \in \Z, \, \xxi \in \Sigma_{2k},
\end{equation}
i.e., for all $k \in \Z$
the following diagram
commutes:
$$\xymatrix { {\Sigma_{2k}} \ar[rrrrr]^{F_k}
\ar[d]_{\Psi_k} &&&&&
{\Sigma_{2(k+1)}}\ar[d]^{\Psi_{k+1}}\\
{\EE} \ar[rrrrr]_{\sigma} &&&&& {\EE}}
$$
Moreover, for all $k \in \Z$ the map $\Psi_k$ is continuous and onto.
\end{theorem}
 \begin{proof}
  We borrow the argument from the proof of Theorem 6.1 in \cite{BF11}.

Note that, by Proposition \ref{continuityine} the map $\Psi_k$ is continuous and onto.
Let us show that the diagram commutes.
 Let $\xxi_k \in \Sigma_{2k}$ and let $\xxi_{k+1}=  F_k(\xxi_k)= \x(T_{2k+2},T_{2k}; \xxi_k)$, and $\xxi_0=\x(T_0, T_{2k}; \xxi_k) \in \Sigma_0$.
 Let us denote by $\ee= \Psi_0(\xxi_0)$,  $\ee'= \Psi_k(\xxi_k)$, $\ee''= \Psi_{k+1}(\xxi_{k+1})$; then by construction
 $\ee''= \sigma^{k+1}(\ee)$ and $\ee'= \sigma^{k}(\ee)$ so that $\ee''= \sigma (\ee')$. Hence
\begin{equation*}
  \begin{split}
   \Psi_{k+1}\circ F_k(\xxi_k)=  \Psi_{k+1}(\xxi_{k+1})= \ee''= \sigma(\ee')=\sigma \circ \Psi_k (\xxi_k),
  \end{split}
\end{equation*}
so the diagram commutes.
 \end{proof}

 Now we briefly consider the setting of Theorem \ref{veryweak} following \cite[\S 6]{CFPrestate}. In this case
we can still define in a similar way as above the sets $\Sigma_k$ and the mappings
$\Psi_k$ and $F_k$, where $\Psi_k$ is onto and $F_k$ is a homeomorphism. However
Remark~\ref{remtwo} is replaced by the following weaker result.
 \begin{remark}\label{remthree}
 Assume the hypotheses of Theorem \ref{veryweak}, then
 $$
\Sigma_{2k} \subset [B(\bs{\Gamma}^{B_{2k}}, \tilde{c}^* \ep) \cup B(\vec{0}, \tilde{c}^* \ep)] \subset B(\bs{\Gamma}, \tilde{c}^* \ep), \qquad \textrm{for any } \; k \in \Z,$$
 where
  $\bs{\Gamma}^{B_{2k}}= \{ \ga(t) \mid |t| \le B_{2k} \}$, $\tilde{c}^*$ is given by \eqref{deftildecstar}, $\Sigma_k$ is given by \eqref{defSk} and
  $X(\ee, \TT)=\{\x(T_0,T_0+\al_0^\ee(\ep);\xxi)\mid \xxi\in\bar{X}(\ee,\TT)\}$.
\end{remark}
 Notice that if $(B_k)$ is unbounded, then the two sets $B(\bs{\Gamma}^{B_{2k}}, \tilde{c}^* \ep)$
 and $B(\vec{0}, \tilde{c}^* \ep)$ may intersect, so $\Psi_k$ might not be continuous.

In conclusion, $\Psi_k$ is onto but possibly discontinuous, so we do not have a real semi-conjugation with the Bernoulli shift even if the diagram in Theorem~\ref{thm.bernoulli} still commutes.

\section{Examples}\label{s.ex}
In this section we present several examples of simple piecewise-smooth systems satisfying the assumptions
of  Theorems \ref{main.periodic1}, \ref{main.weak1} and \ref{veryweak}: the examples are meant mainly to illustrate
assumptions
 \assump{P1}, \assump{P2} or \assump{P3}.

Let us consider the following system:
	\begin{equation}\label{eq_ex_ps}
		\begin{gathered}
			\begin{aligned}\dot{x} &= y+x^2+\ep g_1(t,x,y)\\
				\dot{y} &= x-2x^2\end{aligned} \quad \text{if }y>0,\\[1em]
			\begin{aligned}\dot{x} &= y-x^2+\ep g_1(t,x,y)\\
				\dot{y} &= x-2x^2\end{aligned} \quad \text{if }y<0
		\end{gathered}
	\end{equation}
on $\Om=]-2,2[\times ]-2,2[$.
We let $G(x,y)=-y$, so $\Om^-=]-2,2[\times ]0,2[$, $\Om^+=]-2,2[\times ]-2,0[$, $\Om^0=]-2,2[\times \{0\}$.	
Moreover, $\g= (g_1,g_2)$ with $g_2\equiv 0$ and we will choose $g_1(t,x,y)$ so that
 $g\in C_b^r(\R\times\Omega,\R^2)$.
Note that, if $\ep=0$,
system \eqref{eq_ex_ps} has the homoclinic solution (see Figure \ref{figEx-hom})
\begin{equation}\label{eqexhom}
	\ga(t)=(\gamma_1(t),\gamma_2(t))=\begin{cases}
		(\eu^t,\eu^t-\eu^{2t}),& t\leq 0,\\
		(\eu^{-t},-\eu^{-t}+\eu^{-2t}),& t>0\end{cases}
\end{equation}
corresponding to the homoclinic point $\vec{0}$. Note that for $t>0$, we have
\begin{gather*}
	\gamma_1(-t)=\eu^{-t}=\gamma_1(t),\\
	\gamma_2(-t)=\eu^{-t}-\eu^{-2t}=-\gamma_2(t).
\end{gather*}
\begin{figure}[t]
	\begin{center}
		\includegraphics[width=6cm]{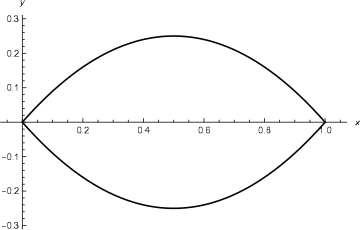}
		\caption{Homoclinic solution of system \eqref{eq_ex_ps} with $\ep=0$.}\label{figEx-hom}
	\end{center}
\end{figure}

\noindent
Further, $\la_u^{\pm}=1$, $\la_s^{\pm}=-1$, $\vec{v}^{\pm}_u=\mp\frac{1}{\sqrt{2}}(1,1)$, $\vec{v}^{\pm}_s=\pm\frac{1}{\sqrt{2}}(1,-1)$, so
\assump{F0, F1, F2, K} are verified. Notice that $\und{\sigma}=\ov{\sigma}=1/2$.
Since $\tr\boldsymbol{\f_x^\pm}(x,y)=0$ for any $(x,y)\in\Om^\pm$ and $c_{\perp}^{\pm}=1$, the corresponding Melnikov function given by \eqref{melni-disc} has the form
\begin{equation}\label{eq_ex_Meln}
	\begin{gathered}
		\mathcal{M}(\aaa)= -\int_{-\infty}^{0}  (\gamma_1(t)-2\gamma_1^2(t))g_1(t+\aaa,\ga(t))\,dt\\
		-\int_{0}^{+\infty} (\gamma_1(t)-2\gamma_1^2(t))g_1(t+\aaa,\ga(t))\,dt\\
		=\int_{-\infty}^{0} (2\eu^{2t}-\eu^t)g_1(t+\aaa,\ga(t))\,dt
		+\int_{0}^{+\infty} (2\eu^{-2t}-\eu^{-t})g_1(t+\aaa,\ga(t))\,dt.
	\end{gathered}
\end{equation}
In the examples below the perturbation term  $\g= (g_1,0)$ will always be of the form $g_1(t,x,y)=xh(t)$ with different choices of $h(t)$, so
that \eqref{eq_ex_Meln} reads
\begin{equation}\label{eq_ex_M}
\begin{split}
	\mathcal{M}(\aaa)&=\int_{-\infty}^{0} (2\eu^{2t}-\eu^t)\eu^{t}h(t+\aaa)\,dt
	+\int_{0}^{+\infty} (2\eu^{-2t}-\eu^{-t})\eu^{-t}h(t+\aaa)\,dt\\
	&=2A(3,\aaa)-A(2,\aaa)
\end{split}
\end{equation}
where we denote
\begin{equation}\label{eq_ex_A}
	A(k, \aaa)=\int_{-\infty}^0\eu^{kt}h(t+\aaa)\,dt+\int_0^\infty \eu^{-kt}h(t+\aaa)\,dt.
\end{equation}

\begin{example}\label{ex1}
In this first example we choose $g_1(t,x,y)=xh(t)$, $h(t)=\sin (2\pi t)$.
\end{example}

 Plugging the given perturbation into \eqref{eq_ex_A}, we get
by direct computation
\[
A(k,\aaa)=\frac{2 k}{4 \pi^2+k^2}\sin(2 \pi\aaa),
\] which gives
$\mathcal{M}(\aaa)=C_1 \sin (2 \pi\aaa)$ where $C_1=\frac{ 8\pi^2+3}{(4\pi^2+9)(\pi^2+1)}>\frac{2}{13}>0$. So in this case assumption
\assump{P3} is satisfied.

Observe that $\MM(j/2)=0$ and $|\MM'(j/2)|= 2 \pi C_1>0$ for any $j \in \Z$ so \eqref{nondeg} holds.
Therefore we can apply Theorem \ref{main.periodic1} to find an $\ep_0$ such that we have the following.
For any $0<\ep \le \ep_0$ and any $\nu \ge \nu_0=1$ (see \eqref{defK0-new}) we can define

$$ \ogap = \big[K_0(1+\nu) |\ln(\ep)|+2 \big]= \big[3(1+\nu) |\ln(\ep)|+2 \big]  $$
where $[ \cdot ]$ denotes the integer part; then we see that any sequence $\TT=(T_j)$, $j \in \Z$ where
$T_j=j \ogap$ verifies \eqref{TandKnu} (in fact also \eqref{TandKnunew}  since $b_j= -1/4 +j/2$  and
$B_j=1/2$).

So we can apply Theorem \ref{main.periodic1} and we find a 
set of initial conditions
$\Sigma(\TT)$ such that $\x(t,T_0;\Q)$ has property \assump{C} for any $Q \in \Sigma(\TT)$.

\smallskip

 Note that    using a  sequence of the form $\bar{\TT}=(\bar{T}_k)$, $\bar{T}_k= k \ogap+1/2$,
$k \in \Z$, we find a different chaotic pattern starting from a new
set, say $\Sigma(\bar{\TT})$.
Further, starting from sequences jumping between
$T_j$ and $\bar{T}_j$ (even randomly), e.g.
$\tilde{\TT}=(\tilde{T}_k)$ where
$$\tilde{T}_k= \left\{ \begin{array}{cc}
                         k \ogap & \textrm{if $k$ is odd},\\
                         k\ogap +1/2 & \textrm{if $k$ is even,}
                       \end{array}    \right.$$
                       we find a different chaotic pattern $\Sigma(\tilde{\TT})$.
                       Moreover we can choose also sequences with non constant gaps, e.g.\ we can allow $T_{j+1}-T_j$ to vary randomly in the
                       semi-integers between
                       $\ogap$ and $2\ogap$.

 \smallskip

Analogously, one can verify the next example.

\begin{example}
	System \eqref{eq_ex_ps} with $g_1(t,x,y)=x (\sin t+\sin(2\pi t))$ satisfies \assump{P2} with
	$\mathcal{M}(\aaa)= \frac{2}{5}\sin\aaa+C_1\sin(2\pi\aaa)$.
\end{example}
Notice that $\aaa=0$ is a non-degenerate zero of $\MM$, further $\MM$ is quasi periodic, hence almost periodic:
it follows that in each quasi period $\bar{N}$ (cf.~\assump{P2}) $\MM$ admits a non degenerate zero.
So we can find a sequence $\TT=(T_j)$ of non-degenerate zeros of $\MM$ such that $T_{j+1}-T_j \ge \ogap$, for any $j \in \Z$: thus \eqref{TandKnu}    and \eqref{nondeg} are satisfied and we can apply
Theorem \ref{main.periodic1}.

\begin{example}\label{exgen0}
	Let us consider system \eqref{eq_ex_ps} with $g_1(t,x,y)=x h(t)$, and  $h(t)=3 \sin (2 \pi t) +R(t)$,
where $R(t)$ is a $C^1$ function such that $|R(t)| \le 1$.
\end{example}
In this case $R(t)$ can be regarded as a noise or a disturbance on which we have little information.
Repeating the previous computation we see that  in this case we get
$$\MM(\aaa)=3 C_1 \sin (2 \pi t) +M_r(t)\, , \quad \textrm{where $|M_r(t)| \le  1/3$}.$$
The estimate on $M_r(t)$ is obtained by plugging $|R(t)|\le 1$ in \eqref{eq_ex_M}.
Note that $\MM(\aaa)$ admits at least one zero in each interval of length $1$,  since $3C_1>6/13>1/3$  but we have no information on the
non-degeneracy of these zeros and in fact we do not know their precise locations.
Further $g$,  and hence $\MM$ do not satisfy any classical recurrence property.

However \eqref{eq_ex_ps} satisfies \assump{P1}  and also \eqref{minimal}, so we deduce the existence of a sequence
of zeros of $\MM(\aaa)$, say
$\bar{\TT}=(\bar{T}_j)$ where $\bar{T}_j \in [\ogap j , \ogap j +1[$, which satisfies \eqref{TandKnu} and \eqref{TandKnunew},
 so that we can apply Theorem \ref{main.weak1}.

Hence for any $\ee \in \EE$ we find (a not uniquely defined)   $\xxi(\ee)$ and a sequence $\al_j^{\ee}=\al_j^{\ee}(\ep)$ such that the
trajectory
$\x(t,T_0 +\al_0^{\ee}; \xxi(\ee))$
has property $\mathbf{C_e}$, but we just know that $|\al_j^{\ee}(\ep)| \le 1$.

Finally notice that   if we happen to know that $R(t)$ is such that $M_r(0)=M'_r(t) \equiv 0$ in an arbitrarily small
neighborhood of the origin, then we can apply Remark \ref{foralpha0} and we find that  $\al_0^{\ee}=O(\ep)$.

\begin{example}\label{exgen}
Let us consider system \eqref{eq_ex_ps} with $g_1(t,x,y)= x h(t)$, $h(t)= \sign(t)\sin^{2r+1}\sqrt{|t|}$, $r\in\Z^+$, $r\geq 2$,
where $\sign(t)=1$ if $t>0$, $\sign(t)=-1$ if $t<0$, and $\sign(t)=0$ if $t=0$.
\end{example}

It can be shown that $h\in C^r(\R)$ so $\g\in C^r_b(\R\times \Om,\R^2)$.
Notice that $h$ is odd, and consequently $\mathcal{M}$ is odd. So, in what follows, we assume that $\aaa>0$.

Once again
plugging the given perturbation into \eqref{eq_ex_A}, we get
\begin{gather*}
	A(k, \aaa)=-\int_{-\infty}^{-\aaa}\eu^{kt}\sin^{2r+1}\sqrt{-(t+\aaa)}\,dt
	+\int_{-\aaa}^0\eu^{kt}\sin^{2r+1}\sqrt{t+\aaa}\,dt\\
	+\int_0^\infty \eu^{-kt}\sin^{2r+1}\sqrt{t+\aaa}\,dt.
\end{gather*}
Now, we estimate these three integrals:
\begin{gather*}
	I_1= -\int_{-\infty}^{-\aaa}\eu^{kt}\sin^{2r+1}\sqrt{-(t+\aaa)}\,dt
	= -\int_0^\infty\eu^{-k(t+\aaa)}\sin^{2r+1}\sqrt{t}\,dt, \\
	|I_1| \le  \eu^{-k\aaa}\int_0^\infty\eu^{-kt}\,dt \to 0
\end{gather*}
 as $\aaa\to\infty$.
Next, integrating by parts, we derive
\begin{gather*}
	I_2=\int_{-\aaa}^0\eu^{kt}\sin^{2r+1}\sqrt{t+\aaa}\,dt
	=2\int_0^{\sqrt{\aaa}}t\eu^{k(t^2-\aaa)}\sin^{2r+1}t\,dt\\
	=\frac{1}{k}\sin^{2r+1}\sqrt{\aaa}-\frac{2r+1}{k}\int_0^{\sqrt{\aaa}}\eu^{k(t^2-\aaa)}\sin^{2r}t\cos t\,dt	.
\end{gather*}
Similarly,
\begin{gather*}
	I_3=\int_0^\infty \eu^{-kt}\sin^{2r+1}\sqrt{t+\aaa}\,dt
	=2\int_{\sqrt{\aaa}}^\infty t\eu^{-k(t^2-\aaa)}\sin^{2r+1}t\,dt\\
	=\frac{1}{k}\sin^{2r+1}\sqrt{\aaa}+\frac{2r+1}{k}\int_{\sqrt{\aaa}}^\infty \eu^{-k(t^2-\aaa)}\sin^{2r}t\cos t\,dt.
\end{gather*}
So we have
 \begin{gather*}
	A(k,\theta)= -\int_0^\infty\eu^{-k(t+\aaa)}\sin^{2r+1}\sqrt{t}\,dt+\frac{2}{k}\sin^{2r+1}\sqrt{\aaa}\\
	{}+\frac{2r+1}{k}\left[\int_{\sqrt{\aaa}}^\infty \eu^{-k(t^2-\aaa)}\sin^{2r}t\cos t\,dt
	-\int_0^{\sqrt{\aaa}}\eu^{k(t^2-\aaa)}\sin^{2r}t\cos t\,dt\right].
\end{gather*}
For the integrals in the bracket, we have the following estimates
$$\left|\int_{\sqrt{\aaa}}^\infty \eu^{-k(t^2-\aaa)}\sin^{2r}t\cos t\,dt\right|
\leq \eu^{k\aaa}\int_{\sqrt{\aaa}}^\infty\eu^{-kt^2}\,dt,$$
where
$$\lim_{\aaa\to\infty}\frac{\int_{\sqrt{\aaa}}^\infty \eu^{-kt^2}\,dt}{\eu^{-k\aaa}}
=\lim_{\aaa\to\infty}\frac{\eu^{-k\aaa}}{2k\sqrt{\aaa}\eu^{-k\aaa}}=\lim_{\aaa\to\infty}\frac{1}{2k\sqrt{\aaa}}=0$$
by de l'Hopital's rule, and analogously
$$\left|\int_0^{\sqrt{\aaa}}\eu^{k(t^2-\aaa)}\sin^{2r}t\cos t\,dt\right|
\leq \eu^{-k\aaa}\int_0^{\sqrt{\aaa}}\eu^{kt^2}\,dt,$$
where
$$\lim_{\aaa\to\infty}\frac{\int_0^{\sqrt{\aaa}}\eu^{kt^2}\,dt}{\eu^{k\aaa}}
=\lim_{\aaa\to\infty}\frac{\eu^{k\aaa}}{2k\sqrt{\aaa}\eu^{k\aaa}}=\lim_{\aaa\to\infty}\frac{1}{2k\sqrt{\aaa}}=0.$$
In conclusion, for $\aaa>0$ we find
$$\mathcal{M}(\aaa)=\frac{1}{3}\sin^{2r+1}\sqrt{\aaa}+ R_1(\aaa),$$
where $R_1(\aaa)\to 0$ as $\aaa\to +\infty$.

Since $\mathcal{M}$ is odd, it follows
$$\mathcal{M}(\aaa)=\sign(\aaa)\frac{1}{3}\sin^{2r+1}\sqrt{|\aaa|}+\sign(\aaa)R_1(|\aaa|).$$
This means that   the Melnikov function has a doubly infinite sequence of negative minima $\MM(b_{2i})$
and positive maxima $\MM(b_{2i+1})$  such that
\begin{equation}\label{MMadd}
  \begin{array}{cc}
        \MM(b_{2i})<-\frac{1}{6}<0 <\frac{1}{6}< \MM(b_{2i+1}) ,  &  i \in \Z, \\
          b_j= \sign (2j-1)\left( \frac{\pi}{2}(2j-1)  \right)^2  ,    & j \in \Z
  \end{array}
\end{equation}

\noindent
so assumption \assump{P1} holds with $\bar{c}=1/6$     (see also Remark
\ref{remP1}).
Moreover, note that the zeros are getting further and further away from each other as $\aaa$ tends to $+\infty$ or
$-\infty$.

In fact $\MM(0)=0$ since $\MM$ is odd. Further we find
$$\MM'(\aaa)=  2A_\theta(3,\aaa)-A_\theta(2,\aaa),$$
$$A_\theta(k,\aaa)= \int_{-\infty}^{+\infty} \eu^{-k|t|} h'(t+\aaa) dt$$
for $A_\theta(k,\theta)=\frac{\partial}{\partial\theta}A(k,\theta)$,
where we used the exponential behaviour of the integrand and the  Lebesgue's dominated convergence theorem. So when
 $\aaa=0$ we find
$$A_\theta(k,0)= 2 (2r+1)\int_{0}^{+\infty} \eu^{-ks^2 }\sin^{2r}s\cos s\, ds  ,$$
 \begin{equation}\label{aprimo}
   \MM'(0)=  2A_\theta(3,0)-A_\theta(2,0)
 \end{equation}
taking into account that the integrand is even and we set $t=s^2$.

 One can compute numerically $\MM'(0)$ for some specific values of $r$; for instance, we find the following:
$$ \textrm{ for  } r=2: \MM'(0) \cong 0.023; $$
$$\textrm{ for  } r=3: \MM'(0) \cong -0.020. $$
Note that in both cases $\MM'(0) \neq 0$.
However in this case, \eqref{nondeg} and even \eqref{isogen}  are not satisfied because we do not have a uniform
control of the slope of $\MM(\aaa)$ in its zeroes.
So we can just apply Theorem \ref{veryweak}, and we
obtain a
chaotic pattern
similar to the one of Example \ref{ex1}. That is, for any sequence of zeros of   $\MM(\cdot)$, say $\TT=(T_j)$, $j\in\Z$, such that $T_0=0$,
$T_{j+1}-T_j
\ge
\ogap$, we obtain a  set $\Sigma(\TT)$ such that for any $\ee \in \EE$ there is a
$\xxi(\ee) \in \Sigma(\TT)$, and a sequence $(\al_j^\ee(\ep))$, $j\in\Z$ such that $\x(t,T_0+\al_0^\ee(\ep); \xxi(\ee))$ has property $\mathbf{C_e}$.

 Further
$\MM'(T_0)\ne 0$, so  we can   apply  Remark \ref{foralpha0} to find $|\alpha_0^{\ee}(\ep)|\le C_{\alpha} \ep$.
 However for $k \ne 0$ we just have the a priori
 estimate
 \begin{equation}\label{aBk}
   |\alpha_k^{\ee}(\ep)|
 	\le B_{2k}= b_{n_{2k}+1}-b_{n_{2k}}
    =\begin{cases} 2\pi^2 |n_{2k}|,& n_{2k}\in\Z\setminus\{0\},\\ \frac{\pi^2}{2},& n_{2k}=0.\end{cases}
 \end{equation}

 As far as we are aware, even in the smooth case, it was not possible to apply
the results available in literature to this type of Melnikov functions (with distance between  consecutive  zeros going to infinity).

\medskip

We can go even further with a more general aperiodic perturbation satisfying~\assump{P1}:

\begin{example}\label{exgen2}
	Let us consider system \eqref{eq_ex_ps} with $g_1(t,x,y)=x h(t)$ and odd function  $h\in C^r(\R)$ such that $|h(t)|\leq H_0$,
$|h'(t)|\leq H_1$ for all $t\in\R$ and some constants $H_0,H_1>0$; further we assume that there is a  sequence $(\aaa_j)_{j \in
\N}$,  $\theta_j>0$ for all $j\in\N$, $\aaa_j \to + \infty$ as
$j \to + \infty$
such that $h(\aaa_{2j})<-c$, $h(\aaa_{2j+1})>c$ for some $c>0$  and  $h'(t) \to 0$ as $t \to +\infty$.
\end{example}

 As before we find  that
$$\mathcal{M}(\aaa)=2A(3,\aaa)-A(2,\aaa),$$
\begin{equation}\label{Mprimo}
  \mathcal{M}'(\aaa)=2A_\theta(3,\aaa)-A_\theta(2,\aaa).
\end{equation}
We assume firstly $\theta>0$ and we investigate the integrals:
$$\left|\int_{-\infty}^{-\aaa}\eu^{kt}h(-(t+\aaa))\,dt\right|	
\leq\int_0^\infty\eu^{-k(t+\aaa)}|h(t)|\,dt \leq H_0\eu^{-k\aaa}\int_0^\infty\eu^{-kt}\,dt$$
which tends to $0$ as $\aaa\to\infty$. Integrating by parts we get
\begin{gather*}
	\int_{-\aaa}^0\eu^{kt}h(t+\aaa)\,dt = \int_0^\aaa \eu^{k(t-\aaa)}h(t)\,dt\\
	=\frac{h(\aaa)}{k}-\frac{\eu^{-k\aaa}h(0)}{k}
	-\frac{\eu^{-k\aaa}}{k}\int_0^\aaa \eu^{kt}h'(t)\,dt.
\end{gather*}
Here if $  \eu^{kt}h'(t) \in L^1(0,\infty)$, we can find $K>0$ such that
$$\left|\frac{\eu^{-k\aaa}h(0)}{k}
	+ \frac{\eu^{-k\aaa}}{k}\int_0^\aaa \eu^{kt}h'(t)\,dt\right| \le \frac{K}{k} \eu^{-k\aaa};$$
otherwise, if $  \eu^{kt}h'(t) \not\in L^1(0,\infty)$, using de l'Hopital's rule, as $\aaa \to \infty$ we find
$$ \frac{\eu^{-k\aaa}h(0)}{k}
	+ \frac{\eu^{-k\aaa}}{k}\int_0^\aaa \eu^{kt}h'(t)\,dt  \sim \frac{h(0)}{k} \eu^{-k\aaa} + \frac{h'(\aaa)}{k^2}.$$
Finally similarly,
\begin{equation*}
	\int_0^\infty \eu^{-kt}h(t+\aaa)\,dt=\int_\aaa^\infty\eu^{-k(t-\aaa)}h(t)\,dt
	=\frac{h(\aaa)}{k}+\frac{\eu^{k\aaa}}{k}\int_\aaa^\infty\eu^{-kt}h'(t)\,dt	
\end{equation*}
with, for $\aaa$ large enough,  	
$$\left|\frac{\eu^{k\aaa}}{k}\int_\aaa^\infty\eu^{-kt}h'(t)\,dt\right|
\leq    2 \frac{|h'(\aaa)|}{k^2}.$$
In conclusion,
$$A(k,\aaa)=\frac{2h(\aaa)}{k}+R(\aaa,k),\quad |R(\aaa,k)|\leq   \frac{\bar{K} \eu^{-k \aaa}}{k} + \frac{3|h'(\aaa)|}{k^2}
 $$
for $\aaa>0$ large and a suitable $\bar{K}>0$.
Hence, using the fact that $\mathcal{M}$ is odd,
 and since by assumption $h'(\aaa) \to 0$ as $\aaa \to +\infty$,
 we obtain
$$\mathcal{M}(\aaa)=\frac{h(\aaa)}{3}+R(\aaa),\quad |R(\aaa)|\to 0 \text{ as $\aaa \to +\infty$}.$$
This means that $\mathcal{M}(\aaa_{2j})<-\tfrac{c}{6}<0$, $\mathcal{M}(\aaa_{2j+1})>\tfrac{c}{6}>0$ and
$\mathcal{M}(-\aaa_{2j})>\tfrac{c}{6}>0$, $\mathcal{M}(-\aaa_{2j+1})<-\tfrac{c}{6}<0$ if $|j|$ is large enough, so
assumption \assump{P1} is  fulfilled. Further from the assumptions on $h$ it follows that the distance between consecutive zeros of $\MM$ is
unbounded.

Hence we find a  sequence of zeros of   $\MM(\cdot)$, say $\TT=(T_j)$, such that $T_0=0$, $T_{j+1}-T_j \ge
\ogap$,
but we do not have control on $\MM'(T_k)$ so we can just apply Theorem \ref{veryweak}
and we obtain results analogous to the ones
of Example \ref{exgen}, in particular we just find $|\al_k| \le B_{2k}$ and $(B_{2k})$ may become unbounded.
Notice that in this result the function $h$ can again be regarded also as some kind of noise on which we have no information but the one given in
the statement.

However, if we assume   $\MM'(0) \ne 0$   (which is in fact an assumption on $h$, see \eqref{Mprimo} and \eqref{aprimo}) we can apply
Remark \ref{foralpha0} and we find $\al_0(\ep)=O(\ep)$.



\end{document}